\documentclass[11pt, amssymb,amsfonts]{amsart}

\usepackage{amsrefs}
\makeatletter
\def\l@section{\@tocline{1}{10pt plus0pt}{0pt}{}{\bfseries}}

\def\@tocline#1#2#3#4#5#6#7{\relax
    \ifnum #1>-1
  \ifnum #1>\c@tocdepth 
  \else
    \par \addpenalty\@secpenalty
    \begingroup \hyphenpenalty\@M
    \@ifempty{#4}{%
      \@tempdima\csname r@tocindent\number#1\endcsname\relax
    }{%
      \@tempdima#4\relax
    }%
    \parindent\z@ \leftskip#3\relax \advance\leftskip\@tempdima\relax
    \rightskip\@pnumwidth plus4em \parfillskip-\@pnumwidth
    #5\leavevmode\hskip-\@tempdima #6\nobreak\relax
    \hfil\hbox to\@pnumwidth{\@tocpagenum{#7}}\par
    \nobreak
    \endgroup
  \fi
\fi}

\makeatother

\usepackage{a4wide}
\usepackage[a4paper, margin=2.3cm]{geometry}
\usepackage{amsmath,amssymb,amsthm}
\usepackage{color}
\usepackage{parskip}
\usepackage{graphicx}
\usepackage{hyperref}
\usepackage{parskip}
\usepackage{setspace}
\usepackage{amsmath}
\usepackage{mathtools}
\usepackage{amssymb}
\usepackage{cases}
\usepackage{algorithm}
\usepackage{algorithmic} 
\usepackage{todonotes}
\usepackage{stmaryrd}
\usepackage{booktabs}
\usepackage[nocompress]{cite}
\newcommand{\norm}[1]{ \left|  #1 \right| }
\newcommand{\Norm}[1]{ \left\|  #1 \right\| }

\def\Var{\hbox{\bf Var}}
\def\Cov{\hbox{\bf Cov}}

\def\P{{\hbox{\bf P}}}
\def\E{{\hbox{\bf E}}}

\font \roman = cmr10 at 10 true pt

\def\be#1{ \begin{equation}\label{#1} }

\def\bas{\begin{align*}}
\def\eas{\end{align*}}
\def\bi{\begin{itemize}}
\def\ei{\end{itemize}}

\def\dim{{\hbox{\roman dim}}}

\def\emph#1{{\it #1}}
\def\textbf#1{{\bf #1}}

\def\Span{\hbox{ \rm Span} \,\,\, }

\def\ep{{\epsilon}}

\theoremstyle{plain}
 \theoremstyle{plain}
  \newtheorem{theorem}{Theorem}
  \numberwithin{theorem}{section}

  \newtheorem{lemma}{Lemma}
  
  \numberwithin{notation}{section}
  \numberwithin{lemma}{section}
  \newtheorem{corollary}{Corollary}
   \numberwithin{corollary}{section}

\theoremstyle{remark}
  \newtheorem{remark}{Remark}
  \numberwithin{remark}{section}

\theoremstyle{definition}
  \newtheorem{definition}{Definition}
\numberwithin{definition}{section}

\theoremstyle{definition}
  \newtheorem{summary}{Summary}
  \numberwithin{summary}{section}

\begin{document}

\include{psfig}
\title[Estimating eigenvectors and eigenspaces of covariance matrices]{Estimating eigenvectors and eigenspaces of covariance matrices: Optimal Bounds and Conditions for Consistency}
\pagenumbering{arabic}

\author{Phuc Tran, Van Vu }
\thanks{trandangphuc234@gmail.com - vuhavan001@gmail.com \\  Department of Mathematics,  The University of Hong Kong (HKU)}
\date{}


\begin{abstract} 
Let $X = [ \xi_1, \,\, \xi_2,...\,\, ,\xi_d]^\top$ be a zero-mean random vector of large dimension $d$ ($d \rightarrow \infty$)  with (hidden) covariance matrix $M = (m_{ij})_{1 \leq i, j \leq d},$ where $m_{ij} = m_{ji} = \Cov(\xi_i, \xi_j).$  
Let  $X_1, X_2, \dots, X_n$ be $n$ iid samples of $X$. Consider the sample covariance matrix 
$$\textstyle \tilde{M} := \frac{1}{n} \sum_{i=1}^{n} X_i X_i^\top.$$  

In practice, one frequently uses the eigenvectors and eigenspaces of $\tilde M$ as estimators for those of $M$. A central task is to provide an error analysis for these estimators.

In this paper, we provide an optimal error analysis, obtaining upper and lower bounds of matching order of magnitude, for a wide range of parameters $d$ and $n$, under mild assumptions on $M$.

As corollaries, we obtain new necessary and sufficient conditions for the consistency of the estimators. In these conditions, we only require the number of samples $n$ to depend linearly on the effective rank of $M$, which can be much smaller than the dimension $d$.


\end{abstract} 

\maketitle
\section{Introduction} \label{sec: intro}
\subsection{The General Problem} \label{subsec: general problem}
Let $X = [ \xi_1, \,\, \xi_2,\,...,\,\, \xi_d]^\top$ be a zero-mean random vector of dimension $d$ with  covariance matrix $M = (m_{ij})_{1 \leq i, j \leq d},$ where $m_{ij} = m_{ji} = \Cov(\xi_i, \xi_j).$  
Let  $X_1, X_2, \dots, X_n$ be $n$ iid samples of $X$ (referred to as the observed data).
We consider the sample covariance matrix
$$\textstyle \tilde{M} := \frac{1}{n} \sum_{i=1}^{n} X_i X_i^\top.$$  

Consider the spectral decomposition of $M$, $ M = \sum_{i=1}^{d} \lambda_i u_i u_i^\top$, in which \( \lambda_1 \geq \lambda_2 \geq \dots \geq \lambda_d  \ge 0\) are the eigenvalues of \( M \), with corresponding orthonormal eigenvectors \( u_i, 1 \leq i \leq d \). Let $H_p :=\Span \{u_1, \dots, u_p \} $ be the eigenspace spanned by the $p$ leading eigenvectors and $\Pi_p$ be the orthogonal projection onto $H_p$.  We write $M = U \Lambda U^\top$, where $U:=[u_1 \, u_2 \, \cdots \, u_d]$ and $\Lambda := \mathrm{Diag}[\lambda_1, \lambda_2, \dots, \lambda_d]$. We define $\tilde{\lambda}_i, \tilde{u}_i, \tilde{U},\tilde{\Lambda}, \tilde{H}_p$ and $\tilde{\Pi}_p$ similarly, with respect to $\tilde{M}$.

It is common practice in statistics and data science to use spectral parameters of the sample covariance matrix $\tilde M$ to estimate those of the (unknown) covariance matrix $M$. A central task here is to determine the accuracy of these estimates.  In particular, one would like to know under which conditions these estimators are consistent (i.e., the error term tends to zero with probability $1-o(1)$); see \cite{JL1, JW2, jung2009pca, shen2016statistics, srivastava2013covariance, RV1, adamczak2010quantitative}. 


In this paper, we consider the estimation of the leading eigenvectors and eigenspaces. Technically speaking, we study the errors
$\|\tilde{u}_p \tilde{u}_p^\top - u_p u_p^\top\|$ and 
$\|\tilde{\Pi}_p - \Pi_p\|$, for small $p$. We will discuss the estimation of the leading eigenvalues and low-rank approximations (PCA) in a subsequent paper.

There is a large literature studying 
$\|\tilde{u}_p \tilde{u}_p^\top - u_p u_p^\top\| $ ($\|\tilde{\Pi}_p - \Pi_p\|$); see the books \cite{pearson1901liii, hotelling1933analysis, jolliffe2011principal, lorenz1956empirical, BYZ1, WW1} and the references therein.  The problem has been completely solved in the classical setting when $d$ is fixed. Among others, it was shown that if $d$ is fixed and $n$ tends to infinity, then the error tends to zero \cite{AndersonPCA2003, dauxois1982asymptotic}. However, in the modern setting (high-dimensional statistics),  where $d$ is large ($d \rightarrow \infty$), and  $n$ can be smaller than $d$, the situation is far from well understood. 

In the high-dimensional setting, one way to attack the problem is to use methods from random matrix theory to obtain limiting results for the error terms; see, for example,  \cite{J1, BY1, BYZ1, FW1, BKYY1}. By the nature of the approach, bounds obtained by this method are asymptotically sharp. However, this approach requires strong assumptions. First, in the setup, one needs to consider an infinite sequence of random vectors $X$ with increasing dimensions.  In this sequence,  the ratio $d/n$ should tend to a constant $\gamma >0$ as $d$ tends to infinity, and the eigenvalues of $M$ must be constants and have a limiting spectral distribution.  Furthermore, the limit is always a positive constant, which means that the estimators are inconsistent;  see, for example, \cite[Theorem 1]{JL1} or \cite[Section 11 - Theorem 11.5]{BYZ1}.

Another approach uses perturbation theory. 
One writes $\tilde M$ as $M +E$, where $E$ is a random matrix with zero mean, and uses results from perturbation theory, such as the Davis-Kahan theorem;  see, for instance, \cite{WW1, JW1, JW2, nadler2008finite, zhu2022high}. This method makes use of modern estimates on the norm of $E$; see \cite{koltchinskii2017concentration, T1, srivastava2013covariance, tikhomirov2018sample}. For a representative result, see \cite[Section 8]{WW1}. Recently, there have been several attempts to refine this approach, either using the contour integral method from numerical analysis \cite{KX1, KL1, JW1, JW2,  mas2003perturbation} or exploiting the fact that the matrix $E$ is random, using tools from high-dimensional probability  \cite{OVK13, OVK22, CCF1, XZ1}.

The perturbation  approach applies in  very general settings and is more applicable in practice, where there is only one matrix $M$
and one pair $d$ and $n$. As a matter of fact, in high-dimensional statistics, it is of considerable interest to make $n$ significantly smaller than $d$.

To the best of our knowledge, all results obtained using perturbation theory provide only {\it upper bounds}, and little has been known about their sharpness.  The main goal of this paper is to provide both upper and lower bounds that match up to a constant factor. 
In the general setting
($d,n,M$  arbitrary), it seems 
very hard to obtain
limiting theorems. Thus, determining the correct order of magnitude  seems to be  the next best thing 
that one may achieve.

{\bf \noindent Problem 1.} 
Determine the order of magnitude of $\|\tilde{u}_p \tilde{u}_p^\top - u_p u_p^\top\| $  and $\|\tilde{\Pi}_p - \Pi_p\|$.

\noindent We will also address the consistency problem:

{\bf \noindent Problem 2.}  Find necessary and sufficient conditions for 
the estimators $\tilde{u}_p$ and  $\tilde{\Pi}_p $ to be consistent.

\vskip2mm

{\bf \noindent Our main results.} 
In this paper, we give an answer to Problem 1, for a wide range of the parameters $d$ and $n$, under mild assumptions on the covariance matrix $M$. In this range, we determine the error $\|\tilde{u}_p \tilde{u}_p^\top - u_p u_p^\top\|$ ($\| \tilde \Pi_p - \Pi_p \|$) up to a constant factor.

Using these results, we obtain answers for Problem 2, again in a wide range of $d, n$ and $M$. We would like to point out that in these new necessary and sufficient conditions,  the lower bound required on $n$, the number of samples, {\it does not} depend on $d$, the dimension. 
Instead, it depends on the effective rank and the signal-to-gap ratio of $M$, both of which can be significantly smaller than $d$.

Our main tool is the combinatorial expansion method, which we developed in a recent sequence of papers \cite{DKTranVu1, tran2025davis, TranVulowrank}.  We start with a contour representation of the error in question. 
This is a well-known technique in perturbation theory and has been used by many authors; see  \cite{Kato1, Higham2008, KX1, KL1, JW1, JW2, mestre2008asymptotic, mas2003perturbation, el2010second, koltchinskii2000random}. In order to analyze this integral, we define two extra expansions, rewriting the integral as a sum of a huge number of terms. 
The critical step is to design combinatorial profiles that group these terms in a suitable manner, so we can estimate them efficiently; see \cite[Section 4 and Section 8]{DKTranVu1}.

By refining the arguments
from \cite{DKTranVu1}, we will be able to write  $\|\tilde u_p \tilde{u}_p^\top -u_p u_p^\top \|$ (and similarly $\| \tilde \Pi_p - \Pi_p \|$) as a sum of a few components.
With a careful analysis, we can identify and compute the dominant component in a number of settings. As a result, we obtain both upper and lower bounds, which are only a constant factor apart. The combinatorial expansion method is of independent interest and has recently found applications in many other problems in statistics and numerical analysis; see \cite{TranVishnoiVu2025, TranVishnoi2025, TranVUeigenvalue, AnyTranVucasualcompletion2026, TranlinhVu}.

Among previous works, our study is most related to a recent study
\cite{JW2}. In particular, in certain settings, 
one can use \cite[Theorem 7]{JW2} to recover our
upper bounds, under a stronger assumption on the number of samples. On the other hand,  \cite{JW2} does not provide lower bounds. \cite[Theorem 8]{JW2} provides a necessary and sufficient condition for consistency in estimating the first eigenvector, given that $\frac{n}{d^2} \rightarrow \infty$.
We will discuss these points in more detail in Section \ref{subsec: existPopu}.

{\bf \noindent Effective rank and signal-to-gap ratio.} Two quantities that play an important role in our analysis are the effective rank $r_{\mathrm{eff}}$ (also referred to as the intrinsic rank) and 
the signal-to-gap ratio  $\frac{\lambda_p}{\delta_{(p)}}$, where $\delta_{(p)} $ is the distance from $\lambda_p$ to the closest eigenvalue.

The notion of effective rank (also known as intrinsic rank) was introduced in \cite{Vershynin2012} and later studied systematically in the context of covariance estimation in \cite{BuneaXiao2015}. Since then, it has become a fundamental concept in high-dimensional statistics. Among its many applications, effective rank plays a key role in the formulation of matrix Bernstein inequalities; see, for example, \cite{RV2, koltchinskii2017concentration, rudelson2013hanson, martinsson2020randomized, T1, Ver1book}.
For matrix $A$ with singular values $\sigma_1 \ge \dots \ge \sigma_m \ge 0$, we set 
$$r_{\mathrm{eff}} (A) : = \frac{\sum_{i=1}^m \sigma_i }{\sigma_1}. $$
As  $M$ is positive semi-definite, its singular values and eigenvalues are the same, so 
\begin{equation} \label{defrstable} r_{\mathrm{eff}} (M) : = \frac{\sum_{i=1}^d \lambda_i }{\lambda_1}=\frac{\mathrm{Trace}\, M }{\lambda_1}. \end{equation}

Our study reveals that the relation between these two parameters and the magnitude of $n$ governs the order of magnitude of $\left\| \tilde{u}_p \tilde{u}_p^\top - u_p u_p^\top \right\|$ 
and $\|\tilde \Pi_p- \Pi_p\|$.  Our method is effective even if the signal-to-gap ratio $\frac{\lambda_p}{\delta_{(p)}}$ tends to infinity with $d$. In fact, this ratio can be as large as a polynomial in $d$. This is an important difference compared to many previous works, where it was assumed that this ratio is bounded (e.g., \cite{FW1, BYZ1, JP1, J1}). This means that we can obtain a good estimator for an eigenvector even when its corresponding eigenvalue is very close to the nearest one.

{\bf \noindent Models for the random vector $X$.} A standard model for $X$ is $X= M^{1/2} Y$, where $Y$ is a random vector such that 
$\E YY^T = I_d $. 
It is easy to check that $\E XX^T = M$. The generality of the model is determined by the assumptions on $M$ and $Y$, and we will consider many models with increasing generalities.  In particular, we will allow $M$ to have an arbitrary spectrum where all eigenvalues can depend on 
the dimension $d$, and $Y$ to have dependent entries. 

\vskip2mm 

\noindent {\bf Structure of the rest of the paper.}
In the next subsection (Subsection \ref{section:simplest}), we present our results in the simplest setting, 
so the reader can easily observe their nature. Stronger and more general results will be stated in the rest of Section \ref{sec: intro} and Sections 
\ref{section: population}, \ref{subsec: eigenspaces}, \ref{sec: consistency}. Sections \ref{section: population} and \ref{subsec: eigenspaces} focus on the order of magnitude of the error terms (Problem 1), and Section \ref{sec: consistency} focuses on the consistency problem (Problem 2). Section \ref{subsec: existPopu} contains a discussion concerning previous related results, further extensions of our results,  and an open question. 

The proofs start in Section \ref{sec: proof}, where we discuss our main strategy 
and set up the main inequalities. In Section \ref{subsec: proofupperM1/2}, we prove the upper bounds, and in  Section \ref{subsec: prooflowerM1/2Line1}, we prove the lower bound. In order to maintain the flow of the paper, we put the proof of a few technical lemmas into appendices.

{\bf \noindent Notation.} We use the asymptotic notations $o, O, \omega, \Theta, \Omega$ under the assumption that $d \rightarrow \infty$. We use $C,C',C'',\ldots$ to denote universal constants, and
$C_0, C_1, C_2,\ldots$ to denote constants that may depend on other bounded parameters. For a matrix $A$, $\| A \|$ denotes its spectral (operator) norm.

\subsection{Treatment of the leading eigenvector when \texorpdfstring{$M$}{M} has a block form} \label{section:simplest}

We first consider the case $p=1$, i.e., the quantity 
$\| \tilde u_1 \tilde u_1^\top -u_1 u_1^\top \| $. 

In this subsection, we assume that 
\(
X=M^{1/2}Y,
\)
where $M$ is a $d\times d$ positive semi-definite matrix of the form
\[
M=
\begin{pmatrix}
\Lambda_r & 0\\
0 & V_d
\end{pmatrix},
\] and $Y$ is a random vector with sub-Gaussian coordinates. 

This model is one of the most popular models for spiked covariance matrices, whose study was initiated in \cite{J1} and has generated considerable interest; see 
\cite{BY2, BYZ1, JW1, JW2, BS1,shen2013surprising, P1, J1} and the references therein. 
Here, $\Lambda_r$ is an $r\times r$ matrix whose eigenvalues are $\lambda_1, \dots, \lambda_r$, the largest $r$ eigenvalues  of $M$.
We call these eigenvalues the spikes of $M$. $V_d$ is a $(d-r)\times(d-r)$ matrix whose eigenvalues are $\lambda_{r+1}, \dots, \lambda_n$, which form the non-spiked part of the spectrum;  see, e.g., \cite[Chapter 11]{BYZ1} for more details. 
We do not make any assumption about the eigenvalues. They can grow with $d$, and their distribution is arbitrary. 

\begin{definition}[Sub-Gaussian random variable]
    A random variable $\xi$ is called \textit{sub-Gaussian} with $\|\xi\|_{\psi_2}^2 \leq K$ if $\E \exp \big(\frac{\xi^2}{K} \big) \leq 2$. In particular, we have $\E \xi^4 \leq 4K^2.$
\end{definition}

We write $Y= (y_1, \dots, y_d)$,  and assume that $y_i$ are independent (not necessarily iid) sub-Gaussian random variables with mean 0 and variance 1, satisfying $\max_{1 \leq i \leq d} \|y_i\|_{\psi_2}^2 \leq K$ for some $K \geq 1$.  We assume that the parameters $K$ and $r$ are bounded, i.e., \begin{equation} \label{cond1} K,r =O(1) . \end{equation}

To distinguish between the spiked and non-spiked eigenvalues, we assume that 
\begin{equation} \label{cond2} \lambda_{r+1} \le \frac{1}{2} \lambda_1 .  \end{equation}
The constant $\frac{1}{2}$ is chosen for convenience; 
it can be replaced by any constant $0 < c < 1$. We also assume that  
\begin{equation} \label{cond3} r_{\mathrm{eff}} + \frac{\lambda_1}{\delta_1} \ge C_0, \end{equation} where $C_0 $ is a sufficiently large  constant depending on $K$ and $r$, to be chosen.

We will consider three ranges of the effective rank, depending on the relation between $r_{\mathrm{eff}} $ and $\frac{\lambda_1}{\delta_1}$.  We show that under assumptions \eqref{cond1}, \eqref{cond2}, \eqref{cond3} (with a proper choice of $C_0$),  the following three theorems, corresponding to these ranges, hold. 

\begin{theorem} [Small effective rank] \label{thm:small} For any constants $\epsilon, K, r$, there are constants $C_0, C_1, C_2 $ so that the following holds. 
If $r_{\mathrm{eff}} \leq  \frac{r\lambda_1}{\delta_1}$,
then with probability at least $1-\epsilon$,  
    $$   \left\| \tilde{u}_1 \tilde{u}_1^\top - u_1 u_1^\top \right\| \leq C_1 \cdot \min \bigg\{ \frac{ \lambda_1 }{\sqrt{n} \delta_1}, 1 \bigg\},$$
\noindent and  with probability at least $\frac{1}{16}$,
$$ C_2 \cdot \min \bigg\{ \frac{ \lambda_1 }{\sqrt{n} \delta_1}, 1 \bigg\} \leq  \left\| \tilde{u}_1 \tilde{u}_1^\top - u_1 u_1^\top \right\| .$$
\end{theorem}
\begin{theorem} [Medium effective rank] \label{thm:medium} 
For any constants $\epsilon, K,r $, there are constants $C_0, C_1, C_2, C_3$ so that the following holds.  If $\frac{r\lambda_1}{\delta_1} < r_{\mathrm{eff}} \leq r \big(\frac{\lambda_1}{\delta_1}\big)^2 $ and $n \geq C_3 \cdot \frac{\lambda_1}{\delta_1} r_{\mathrm{eff}}$, then with probability at least $1-\epsilon$,
\[
\left\| \tilde{u}_1 \tilde{u}_1^\top - u_1 u_1^\top \right\| \leq C_1 \cdot \frac{ \lambda_1 }{\sqrt{n} \delta_1},
\]
and with probability at least $\frac{1}{16}$, 
   \[
  C_2\cdot \frac{\lambda_1}{\sqrt{n} \delta_1} \leq \left\| \tilde{u}_1 \tilde{u}_1^\top - u_1 u_1^\top \right\|.
   \]
\end{theorem}


\begin{theorem} [Large effective rank] \label{thm:large} 
 For any constant $\epsilon, K,r $, there are constants $C_0, C_1, C_2, C_3$ so that the following holds. If $ r_{\mathrm{eff}} > r \big(\frac{\lambda_1}{\delta_1}\big)^2 $ and  $n \geq C_3 \cdot  r_{\mathrm{eff}}  \frac{\lambda_1}{\delta_1}   $, then with probability at least $1-\epsilon$, 
\[
\left\| \tilde{u}_1 \tilde{u}_1^\top - u_1 u_1^\top \right\| \leq C_1\cdot \sqrt{\frac{r_{\mathrm{eff}}}{n}},
\]
and with probability at least $\frac{1}{16}$,
\[
C_2 \cdot \sqrt{\frac{r_{\mathrm{eff}}}{n}}  \leq \left\| \tilde{u}_1 \tilde{u}_1^\top - u_1 u_1^\top \right\| .
\]
\end{theorem}

 In all three theorems, the lower bound and upper bound match up to the constants involved.  
 The dependence of the constants $C_i$ on the parameters 
 $K, r, \epsilon$ is polynomial, and can be made explicit;  see Remark \ref{remark:explicit}.

{\bf \noindent Necessary and sufficient conditions for consistency.}  
As immediate corollaries of the above theorems, we obtain the following results, addressing  Problem 2.

\begin{corollary} \label{consistent1}
  Assume that  $r_{\mathrm{eff}} \leq  \frac{r\lambda_1}{\delta_1}$. Then 
    $  \left\| \tilde{u}_1 \tilde{u}_1^\top - u_1 u_1^\top \right\| =o(1) $
 if and only if 
    $\frac{ \lambda_1 }{\sqrt{n} \delta_1} =o(1). $
\end{corollary}


%
\begin{corollary} \label{consistent2}
    Assume that  $\frac{r\lambda_1}{\delta_1} < r_{\mathrm{eff}} \leq r \big(\frac{\lambda_1}{\delta_1}\big)^2 $ and $n \geq C_3 \cdot r_{\mathrm{eff}} \frac{\lambda_1}{\delta_1}$. Then
     $  \left\| \tilde{u}_1 \tilde{u}_1^\top - u_1 u_1^\top \right\| =o(1) $ 
     if and only if 
    $\frac{ \lambda_1 }{\sqrt{n} \delta_1} =o(1)$ . 
\end{corollary}

\begin{corollary} \label{consistent3}
    Assume that $ r_{\mathrm{eff}} > r \big(\frac{\lambda_1}{\delta_1}\big)^2 $ and $n \geq  C_3 \cdot r_{\mathrm{eff}} \frac{\lambda_1}{\delta_1}$. Then
     $  \left\| \tilde{u}_1 \tilde{u}_1^\top - u_1 u_1^\top \right\| =o(1) $
 if and only if
    $ \sqrt{\frac{r_{\mathrm{eff}}}{n}} =o(1)$ .
\end{corollary}

As mentioned earlier, the lower bounds on the number of samples $n$ in these results do not 
involve the dimension $d$. For example, in Corollary \ref{consistent1}, we need 
$$\textstyle \frac{n} {(\frac{\lambda_1}{\delta_1})^2 } \rightarrow \infty,  $$ 
\noindent while in Corollary \ref{consistent3},
we need $n \geq C_3 \cdot \frac{\lambda_1}{\delta_1} r_{\mathrm{eff}}$
and 
$$\frac{n}{r_{\mathrm{eff}} } \rightarrow \infty.$$
The latest result we found that provides a necessary and sufficient condition for consistency is \cite[Theorem 8]{JW2}. In this result,  $n$ is required to be super-quadratic in $d$, namely, $\frac{n}{d^2} \rightarrow \infty$; see Section \ref{subsec: existPopu} for more discussion.

\subsection{Treatment of \texorpdfstring{$u_p$}{up} for general \texorpdfstring{$p$}{p}}
Let us now turn to the treatment 
of $\left\| \tilde{u}_p \tilde{u}_p^\top - u_p u_p^\top \right\| $ for a general $p$. We keep the model of $X$ the same as in the previous section. Instead of the gap $\delta_1 = \lambda_1 -\lambda_2$, we consider 
\(
\delta_{(p)}:=\min\{\delta_{p-1}, \delta_p\},
\) which is the gap between $\lambda_p$ and its nearest eigenvalue. 
Another new parameter is $\kappa_{p} :=\frac{\lambda_1}{\lambda_p}.$ In the previous case when $p=1$, $\kappa_1=1$ and was omitted. In what follows, we assume that $\kappa_p =O(1)$.

We can extend the above three theorems for a general index $p$, basically by replacing $\frac{\lambda_1}{\delta_1}$ by $\frac{\lambda_p}{\delta_{(p)}}$ in all statements. Instead of \eqref{cond3}, we assume that 
\begin{equation} \label{cond4} r_{\mathrm{eff}}+\frac{\lambda_p}{\delta_{(p)}} \geq C_0. \end{equation} 

\begin{theorem} [Small effective rank] \label{thm:smallp} For any constants $\epsilon, K, r, \kappa_p$, there are constants $C_0, C_1, C_2$ so that the following holds.  If $r_{\mathrm{eff}} \leq  \frac{r\lambda_p}{\delta_{(p)}}$,
then with probability at least $1-\epsilon$,  
    $$   \left\| \tilde{u}_p \tilde{u}_p^\top - u_p u_p^\top \right\| \leq C_1  \cdot \min \bigg\{ \frac{  \lambda_p }{\sqrt{n} \delta_{(p)}}, 1 \bigg\},$$
\noindent and  with probability at least $\frac{1}{16}$,
$$ C_2 \cdot \min \bigg\{ \frac{ \lambda_p }{\sqrt{n} \delta_{(p)}}, 1 \bigg\} \leq  \left\| \tilde{u}_p \tilde{u}_p^\top - u_p u_p^\top \right\| .$$
\end{theorem}


%
\begin{theorem} [Medium effective rank] \label{thm:mediump} For any constants $\epsilon, K,r $, there are constants $C_0, C_1, C_2, C_3$ so that the following holds.  If $\frac{r\lambda_p}{\delta_{(p)}} < r_{\mathrm{eff}} \leq \frac{r}{\kappa_p^2} \big(\frac{\lambda_p}{\delta_{(p)}}\big)^2 $, and $n \geq C_3 \cdot r_{\mathrm{eff}}  \frac{\lambda_p}{\delta_{(p)}} $, then with probability at least $1-\epsilon$,
$$   \left\| \tilde{u}_p \tilde{u}_p^\top - u_p u_p^\top \right\| \leq C_1 \cdot \frac{ \lambda_p }{\sqrt{n} \delta_{(p)}},$$
\noindent and  with probability at least $\frac{1}{16}$,
$$ C_2 \cdot \frac{ \lambda_p }{\sqrt{n} \delta_{(p)}} \leq  \left\| \tilde{u}_p \tilde{u}_p^\top - u_p u_p^\top \right\| .$$
\end{theorem}


\begin{theorem} [Large effective rank] \label{thm:largep} 
 For any constant $\epsilon, K,r $, there are constants $C_0, C_1, C_2, C_3$ so that the following holds. If $ r_{\mathrm{eff}} > \frac{r}{\kappa_p^2} \big(\frac{\lambda_p}{\delta_{(p)}}\big)^2 $, and  $n \geq C_3 \cdot  r_{\mathrm{eff}}  \frac{\lambda_p}{\delta_{(p)}}   $, then with probability at least $1-\epsilon$, 
\[
\left\| \tilde{u}_p \tilde{u}_p^\top - u_p u_p^\top \right\| \leq C_1  \cdot \sqrt{\frac{r_{\mathrm{eff}}}{n}},
\]
and with probability at least $\frac{1}{16}$,
\[
C_2 \cdot \sqrt{\frac{r_{\mathrm{eff}}}{n}}  \leq \left\| \tilde{u}_p \tilde{u}_p^\top - u_p u_p^\top \right\| .
\]
\end{theorem}

In the next few sections, we present new results for more general settings.
While these statements will become a bit more technical, 
they all share the same structure as the ones above. 
We will have matching upper and lower bounds in the three 
ranges of the effective rank, under a similar condition on $n$.

\section{Results in more general settings}
\label{section: population}

 Recall that  $X=M^{1/2} Y$ with $Y=[ y_1 ,\,\,y_2, \,...\, , y_d]^\top$. To generalize the results in the last section, we
 first remove the condition that $M$ has block structure. Next, we allow the entries of $Y$ to be dependent. We will only require that the entries $\{y_i\}_{i=1}^{d}$ are 8-wise independent sub-Gaussian random variables with mean $0$, variance $1$, and $\|y_i\|_{\psi_2}^2 \leq K$. 
 
 The treatment of this general model leads to the following new parameters, which are critical to the bounds. We define
\begin{equation} \label{def: Sij}
    S^{(i,j)}:= (u_i^\top Y) \cdot (u_j^\top Y) \,\,\text{for $1\leq i,j \leq d$}. 
\end{equation}
A  routine direct computation shows that
\begin{equation} \label{iden: expect and variance Sij}
    \begin{split}
        \E  S^{(i,j)} &  =\begin{cases}
            0 & \,\,\, \text{if $i \neq j$}\\
            1 & \,\,\,\text{if $ i= j$}
        \end{cases},\,\, \Var S^{(i,j)} = \begin{cases}
    1+  \sum_{k=1}^d \E (y_k^4 -3) u_{ik}^2 u_{jk}^2 &\,\,\text{if $i \neq j$}\\
    2 + \sum_{k=1}^d \E (y_k^4 -3) u_{ik}^4 &\,\,\,\text{if $i=j$}
\end{cases}\,\,.
    \end{split}
\end{equation}

$\Var S^{(i,j)}$ is typically of order $\Theta (1)$.  The following lemma provides concrete estimates.  
\begin{lemma} \label{lem: S value} We have 

\begin{itemize} 

\item Let $c \ge  0$ be such that $\min_{1 \leq k \leq d} \E y_k^4 \geq 1+2c$. Then 
\[
c \leq \Var S^{(i,j)} \leq 4K^2.
\]
 \item If $Y$ has iid Gaussian $\mathcal{N}(0,1)$ coordinates, then $\Var S^{(i,j)}= \begin{cases}
     1 \,\,& \text{if $i \neq j$} \\
     2 \,\,& \text{if $i=j$}
 \end{cases}$.

\item If $\max_{1 \leq k \leq d}|u_{ik} u_{jk}|=o(1)$, then
\(
\Var S^{(i,j)} = \begin{cases}
     1-o(1) \,\,& \text{if $i \neq j$} \\
     2 - o(1) \,\,& \text{if $i=j$}
\end{cases}
\).

\end{itemize} 
\end{lemma}


We also redefine the parameter $r$, which previously was the size of the block. In the general setting, we define the set of {\it important } indices
\[ \textstyle
I_p:= \big\{ i \in [d]: |\lambda_i -\lambda_p|<\frac{\lambda_p}{2} \big\}\,\,\text{and let $r:=|I_p|$.} 
\]
The constant $\frac{1}{2}$ is chosen for convenience and can be replaced by any constant $0 < c <1$. Intuitively, we want to identify the important eigenvalues that are close to $\lambda_p$ and could have an impact on the perturbation of 
$u_p$. Unlike in the previous section, here  $r$ is not necessarily greater than $p$. In fact, $r$ depends only on the distribution of eigenvalues around $\lambda_p$, not on the actual value of $p$. 

We now present extensions of Theorems~\ref{thm:smallp}, \ref{thm:mediump}, and~\ref{thm:largep}, which bound
\(
\|\tilde u_p\tilde u_p^\top-u_pu_p^\top\|
\)
.
Recall that  $\delta_p := \lambda_p - \lambda_{p+1}$ and $\delta_{(p)} : =  \min \{\delta_{p-1}, \delta_p \} $ for $p \ge 2$ and $\kappa_p := \frac{\lambda_1}{\lambda_p}$. For $p \ge 1$, we set 
\[
\gamma_p:= \begin{cases}
    p+1 &\,\,\text{if $\delta_{(p)}=\delta_{p}$} \\
    p-1 & \,\,\text{if $\delta_{(p)}=\delta_{p-1}$}
\end{cases}.
\]

As in the previous section, we assume
\[
K, r, \kappa_p = O(1),
\]
\noindent and also that \eqref{cond4} holds.

\begin{theorem} [Small effective rank] \label{thm:small-dependentp} For any constants $K,r, \kappa_p, \epsilon,$ there are constants $C_0, C_1, C_2, $ so that the following holds.  If $r_{\mathrm{eff}} \leq  \frac{r\lambda_p}{\delta_{(p)}}$,
then with probability at least $1-\epsilon$,  
    $$ \textstyle  \left\| \tilde{u}_p \tilde{u}_p^\top - u_p u_p^\top \right\| \leq C_1 \cdot \min \big\{ \frac{ \lambda_p }{\sqrt{n} \delta_{(p)}}, 1 \big\},$$
\noindent and with probability at least $\frac{1}{16}$,
$$\textstyle  C_2 \cdot \min \big\{ \frac{ \lambda_p }{\sqrt{n} \delta_{(p)}}, 1 \big\} \leq  \left\| \tilde{u}_p \tilde{u}_p^\top - u_p u_p^\top \right\| .$$
\end{theorem}


\begin{theorem} [Medium effective rank] \label{thm:medium-dependentp} For any constants $K,r, \kappa_p, \epsilon,$ there are constants $C_0, C_1, C_2, C_3 $ so that the following holds.   If $\frac{r\lambda_p}{\delta_{(p)}} < r_{\mathrm{eff}} \leq \frac{r}{\kappa_p^2}\big(\frac{\lambda_p}{\delta_{(p)}} \big)^2$ and $ n \geq C_3 \cdot r_{\mathrm{eff}} \frac{\lambda_p}{\delta_{(p)}}$ ,
then with probability  at least $1-\epsilon$,  
    $$ \textstyle  \left\| \tilde{u}_p \tilde{u}_p^\top - u_p u_p^\top \right\| \leq C_1\cdot  \frac{ \lambda_p }{\sqrt{n} \delta_{(p)}},$$
\noindent and with probability at least $\frac{1}{16}$,
$$\textstyle C_2 \cdot  \frac{ \lambda_p }{\sqrt{n} \delta_{(p)}} \leq  \left\| \tilde{u}_p \tilde{u}_p^\top - u_p u_p^\top \right\| .$$
\end{theorem}
\begin{theorem} [Large effective rank] \label{thm:large-dependentp} For any constants $K,r, \kappa_p, \epsilon,$ there are constants $C_0, C_1, C_2, C_3 $ so that the following holds.   If $r_{\mathrm{eff}} > \frac{r}{\kappa_p^2}\big(\frac{\lambda_p}{\delta_{(p)}} \big)^2$ and $ n \geq C_3 \cdot r_{\mathrm{eff}} \frac{\lambda_p}{\delta_{(p)}}$ ,
then with probability at least $1-\epsilon$,  
    $$ \textstyle  \left\| \tilde{u}_p \tilde{u}_p^\top - u_p u_p^\top \right\| \leq C_1 \cdot   \sqrt{\frac{r_{\mathrm{eff}}}{n}},$$
\noindent and with probability at least $\frac{1}{16}$,
$$\textstyle C_2 \cdot \sqrt{\frac{r_{\mathrm{eff}}}{n}} \leq  \left\| \tilde{u}_p \tilde{u}_p^\top - u_p u_p^\top \right\| .$$
\end{theorem}

%
\begin{remark} \label{remark:explicit}
   We can make the dependence of the constants $C_0, C_1, C_2, C_3$ on $K,r, \epsilon, \kappa_p$ explicit, as follows: 
    \[ \textstyle
    C_0 = (CKr)^2, C_1= CK \kappa_p \sqrt{\frac{r^3}{\epsilon}}, C_2=\frac{1}{16}\min \big\{\sqrt{\Var S^{(p, \gamma_p)}}, \sqrt{\frac{s_p}{\kappa_p}} \big\}, \,\,\text{and}\,\,  C_3= \frac{CK^4 \kappa_p^5  r^6}{\epsilon},
    \]
    where  $C$ is a universal constant and 
\[ \textstyle
s_p:= \frac{\sum_{i=1}^d \lambda_i \Var S^{(p,i)} }{\sum_{i=1}^d \lambda_i} = \frac{\sum_{i=1}^d \lambda_i \Var S^{(p,i)} }{\mathrm{Trace}\,\,M}.
\]
By Lemma \ref{lem: S value}, it is clear that $s_p= O(1)$. The exponents of $K,\kappa_p,r$ in the explicit formulae are generous, but 
we do not try to optimize them. 
\end{remark}
%

 We now present the main technical results of our paper, 
 the Upper Bound Theorem and the Lower Bound Theorem. 
 In Section \ref{subsection:routine}, we will use Theorems \ref{theo: upUpperM1/2 subGau} and \ref{theo: uplowerboundM1/2 subGau}  to derive Theorems \ref{thm:small-dependentp}, \ref{thm:medium-dependentp}, and \ref{thm:large-dependentp}. 
 
 Notice that the 
     bounds in Theorems \ref{theo: upUpperM1/2 subGau} and \ref{theo: uplowerboundM1/2 subGau} below are explicit in terms of the parameters $r, \epsilon, K, \kappa_p$. This leads to the 
     explicit formulae of $C_0, C_1, C_2, C_3$ in 
     Remark \ref{remark:explicit}.

\begin{theorem}[Upper bound] \label{theo: upUpperM1/2 subGau}
    There is a universal constant $C > 0$ so that for any $t > 0$, if 
    \[
    n \geq C(K \kappa_p t)^2 \cdot \max \bigg\{ r_{\mathrm{eff}} \cdot \frac{r\lambda_p}{\delta_{(p)}}, \left(\frac{r \lambda_p}{\delta_{(p)}} \right)^2 \bigg\},
    \]
     then with probability at least $1-e^{-4r_{\mathrm{eff}}t^2} - \frac{2r^2}{t^2}$, 
    \[
    \left\| \tilde{u}_p \tilde{u}_p^\top - u_p u_p^\top \right\|     \leq 280 Kt\left( \kappa_p \sqrt{\frac{r_{\mathrm{eff} }  }{n}} +  \frac{ \sqrt{r} \lambda_p}{ \sqrt{n} \delta_{(p)} } \right).
    \]
\end{theorem}
\begin{theorem}[Lower bound] \label{theo: uplowerboundM1/2 subGau}  For any $C' > 24$ and $t > 0$, if 
    \[
    n \geq (40C'K \kappa_p t)^2 \cdot \max \bigg\{ r_{\mathrm{eff}} \cdot \frac{r\lambda_p}{\delta_{(p)}}, \left(\frac{r \lambda_p}{\delta_{(p)}} \right)^2 \bigg\},
    \]
     then with probability at least $\frac{1}{8}- e^{-4r_{\mathrm{eff}}t^2} - \frac{2r^2}{t^2}$, 
    \[
    \left\| \tilde{u}_p \tilde{u}_p^\top - u_p u_p^\top \right\|     \geq \frac{1}{4} \left[ \sqrt{\Var S^{(p,\gamma_p)}} \cdot \frac{\lambda_p}{\sqrt{n} \delta_{(p)}} + \sqrt{\frac{s_p}{\kappa_p}} \cdot  \sqrt{\frac{r_{\mathrm{eff}}}{n}} \right] - \frac{4 Kt}{\kappa_p\sqrt{n}}  -\frac{900Kt}{C'}\left[ \kappa_p\sqrt{\frac{r_{\mathrm{eff} }  }{n}} +  \frac{ \sqrt{r} \lambda_p}{ \sqrt{n} \delta_{(p)} } \right].
    \]
     \end{theorem}

      The constants $40,280, 8, 16, 900$ are ad hoc, and can
      be reduced by carefully optimizing the constants in the computations in Section \ref{sec: proof} and Appendix \ref{section: moments}.

\section{Perturbation of Eigenspaces}  \label{subsec: eigenspaces}

In this section, we extend previous results to the perturbation of leading eigenspaces.
Our goal is to analyze
$$\|\tilde{\Pi}_p - \Pi_p\|.$$

We keep the general setting as in the last section. After fixing the index $p$, 
 we let $r \geq p$ to be the largest integer such that $\lambda_r \geq \frac{\lambda_p}{2}$ (again, the constant $1/2$ is ad hoc and can be replaced by any fixed constant $0<c<1$).  As a result, the set of important indices is $I_{[p]}:=\{1,2,\dots, r\}$.

As in previous sections, we assume 
$K, r, \kappa_p:= \frac{\lambda_1}{\lambda_p } = O(1)$. Instead of \eqref{cond4}, we assume 
\begin{equation} \label{cond5} r_{\mathrm{eff}}+\frac{\lambda_p}{\delta_p} \geq C_0,\end{equation}
for a sufficiently large constant $C_0$ which may depend on $K$ and $r$. 


\begin{theorem} [Small effective rank] \label{thm:smallpi} For any constants $\epsilon, K, r, \kappa_p$, there are constants $C_0, C_1, C_2$ so that the following holds.  If $r_{\mathrm{eff}} \leq  \frac{r\lambda_p}{\delta_{p}}$,
then with probability at least $1-\epsilon$,  
    $$ \textstyle  \|\tilde{\Pi}_p - \Pi_p\| \leq C_1 \cdot \min \big\{ \frac{ \lambda_p }{\sqrt{n} \delta_{p}}, 1 \big\},$$
\noindent and  with probability at least $\frac{1}{16}$
$$ \textstyle C_2 \cdot  \min \big\{ \frac{ \lambda_p }{\sqrt{n} \delta_{p}}, 1 \big\} \leq  \|\tilde{\Pi}_p - \Pi_p\| .$$
\end{theorem}
\begin{theorem} [Medium effective rank] \label{thm:mediumpi} For any constants $\epsilon, K,r, \kappa_p $, there are constants $C_1, C_2, C_3$ so that the following holds.  If $\frac{r\lambda_p}{\delta_{p}} < r_{\mathrm{eff}} \leq \frac{r}{\kappa_p^2} \big(\frac{\lambda_p}{\delta_{p}}\big)^2 $, and $n \geq C_3 \cdot  r_{\mathrm{eff}}  \frac{\lambda_p}{\delta_{p}} $, then with probability at least $1-\epsilon$,
$$  \textstyle  \|\tilde{\Pi}_p - \Pi_p\|\leq C_1 \cdot  \frac{ \lambda_p }{\sqrt{n} \delta_{p}},$$
\noindent and  with probability at least $\frac{1}{16}$
$$ \textstyle C_2 \cdot  \frac{ \lambda_p }{\sqrt{n} \delta_{p}} \leq  \|\tilde{\Pi}_p - \Pi_p\| .$$
\end{theorem}
\begin{theorem} [Large effective rank] \label{thm:largepi} 
 For any constant $\epsilon, K,r, \kappa_p $, there are constants $ C_1, C_2, C_3$ so that the following holds. If $ r_{\mathrm{eff}} > \frac{r}{\kappa_p^2} \big(\frac{\lambda_p}{\delta_{p}}\big)^2 $, and  $n \geq C_3 \cdot   r_{\mathrm{eff}}  \frac{\lambda_p}{\delta_{p}}   $, then with probability at least $1-\epsilon$, 
\[ \textstyle
\|\tilde{\Pi}_p - \Pi_p\| \leq C_1 \cdot \sqrt{\frac{r_{\mathrm{eff}}}{n}},
\]
and with probability at least $\frac{1}{16}$,
\[ \textstyle
C_2 \cdot \sqrt{\frac{r_{\mathrm{eff}}}{n}}  \leq\|\tilde{\Pi}_p - \Pi_p\| .
\]
\end{theorem}

Similar to the previous section,  we obtain these results by combining the following upper-and lower-bound theorems. 

 \begin{theorem} \label{theo: Piupper}
   %
   There is a universal constant $C > 0$ so that for any $t > 0$, if 
    \[
    n \geq C(K \kappa_p t)^2 \cdot \max \bigg\{ r_{\mathrm{eff}} \cdot \frac{r\lambda_p}{\delta_{p}}, \left(\frac{r \lambda_p}{\delta_{p}} \right)^2 \bigg\},
    \]
     then with probability at least $1-e^{-4r_{\mathrm{eff}}t^2} - \frac{2r^2}{t^2}$, 
    \[
    \Norm{\tilde{\Pi}_p  - \Pi_p }     \leq 280K t\left( \kappa_p\sqrt{\frac{r_{\mathrm{eff} }  }{n}} +  \frac{ \sqrt{r} \lambda_p}{ \sqrt{n} \delta_{p} } \right).
    \]
\end{theorem}

Next, the lower bound is an analogue of Theorem \ref{theo: uplowerboundM1/2 subGau}. We replace $\delta_{(p)}$ with $\delta_p$, and $\gamma_p$ with $p+1$, since in this setting, only the gap $\lambda_p -\lambda_{p+1}$ matters. 
\begin{theorem} \label{theo: PipM1/2lower} 
For any $C'>24$ and any $t > 0$, if  \[
        n \geq  (40C'K \kappa_p t)^2  \cdot \max \bigg\{ r_{\mathrm{eff}} \cdot \frac{r\lambda_p}{\delta_{p}}, \left(\frac{r \lambda_p}{\delta_{p}} \right)^2 \bigg\},
        \]
        then with probability at least $\frac{1}{8}- e^{-4r_{\mathrm{eff}}t^2} - \frac{2pr^2}{t^2}$, 
        \[ 
          \Norm{\tilde{\Pi}_p  - \Pi_p } \geq \frac{1}{4} \left[ \sqrt{\Var S^{(p,p+1)}} \cdot \frac{\lambda_p}{\sqrt{n} \delta_{p}} + \max_{1 \leq i \leq p}\sqrt{\frac{s_i}{\kappa_i}} \cdot \sqrt{\frac{r_{\mathrm{eff}}}{n}} \right] - \frac{4Kt}{\kappa_p\sqrt{n}} - \frac{900Kt}{C'} \left[\kappa_p \sqrt{\frac{r_{\mathrm{eff}}}{n}} + \frac{\sqrt{r} \lambda_p}{\sqrt{n}\delta_p}\right] .
        \]

\end{theorem}

\section{Consistency Theorems} \label{sec: consistency}

In this section, we extend the consistency results for the leading eigenvector in Subsection \ref{section:simplest} (Corollaries \ref{consistent1},\ref{consistent2}, and \ref{consistent3}) to general eigenvectors $\tilde u_p$ and eigenspaces $\tilde \Pi_p$. 
We use the general setting in Section \ref{section: population}. To ease the presentation, we also assume
\[
K, r, \kappa_p = O(1).
\]

\subsection{Consistency in estimating eigenvectors} Using Theorems \ref{thm:small-dependentp},\ref{thm:medium-dependentp}, and \ref{thm:large-dependentp}, we derive the following results, providing the consistency conditions for $\tilde{u}_p$. We assume that \eqref{cond4} holds and  $\Var S^{(p,\gamma_p)}, s_p \geq c$ for some fixed constant $c > 0$.

\begin{theorem} [Small effective rank] \label{thm:small-consistency} 
Assume that $r_{\mathrm{eff}} \leq  \frac{r\lambda_p}{\delta_{(p)}}$.
Then,
\[
\left\| \tilde{u}_p \tilde{u}_p^\top - u_p u_p^\top \right\| = o(1)\,\,\,\text{if and only if}\,\, \frac{ \lambda_p }{\sqrt{n} \delta_{(p)}} = o ( 1).
\]
\end{theorem}

\begin{theorem} [Medium effective rank] \label{thm:medium-consistency} Assume that $\frac{r\lambda_p}{\delta_{(p)}} < r_{\mathrm{eff}} \leq \frac{r}{\kappa_p^2}\big(\frac{\lambda_p}{\delta_{(p)}} \big)^2$.
For any constants $K,r, \kappa_p,$ there are constants $C_0, C_3 $ so that if $ n \geq C_3 \cdot r_{\mathrm{eff}} \frac{\lambda_p}{\delta_{(p)}}$, then the following holds:  %
\[
\left\| \tilde{u}_p \tilde{u}_p^\top - u_p u_p^\top \right\| = o(1)\,\,\,\text{if and only if}\,\, \frac{ \lambda_p }{\sqrt{n} \delta_{(p)}} = o (1).
\]
\end{theorem}
\begin{theorem} [Large effective rank] \label{thm:large-consistency} Assume that $r_{\mathrm{eff}} > \frac{r}{\kappa_p^2}\big(\frac{\lambda_p}{\delta_{(p)}} \big)^2$.  For any constants $K,r, \kappa_p,$ there are constants $C_0, C_3 $ so that if $ n \geq C_3 \cdot r_{\mathrm{eff}} \frac{\lambda_p}{\delta_{(p)}}$, then the following holds:  %
\[
\left\| \tilde{u}_p \tilde{u}_p^\top - u_p u_p^\top \right\| = o(1)\,\,\,\text{if and only if}\,\, \sqrt{\frac{r_{\mathrm{eff}}}{n}}= o(1).
\]
\end{theorem}
 
\subsection{Consistency in estimating eigenspaces}

Similarly, we can use Theorems \ref{thm:smallpi}, \ref{thm:mediumpi} and \ref{thm:largepi} to derive the following consistency theorems for $\tilde{\Pi}_p$. We assume that \eqref{cond5} holds and $\Var S^{(p,p+1)}, s_p \geq c$ for some fixed constant $c > 0$. 
\begin{theorem} [Small effective rank] \label{thm:small-consistencyPi} 
Assume that $r_{\mathrm{eff}} \leq  \frac{r\lambda_p}{\delta_{p}}$.
Then,
\[
\left\| \tilde{\Pi}_p  - \Pi_p \right\| = o(1)\,\,\,\text{if and only if}\,\, \frac{ \lambda_p }{\sqrt{n} \delta_{p}} = o (1).
\]
\end{theorem}
\begin{theorem} [Medium effective rank] \label{thm:medium-consistencyPi} Assume that $\frac{r\lambda_p}{\delta_{p}} < r_{\mathrm{eff}} \leq \frac{r}{\kappa_p^2}\big(\frac{\lambda_p}{\delta_{p}} \big)^2$. For any constants $K,r, \kappa_p,$ there are constants $C_0, C_3$ so that if  $ n \geq C_3 \cdot r_{\mathrm{eff}} \frac{\lambda_p}{\delta_{p}}$, then the following holds.
\[
\left\| \tilde{\Pi}_p  - \Pi_p \right\| = o(1)\,\,\,\text{if and only if}\,\, \frac{ \lambda_p }{\sqrt{n} \delta_{p}} = o (1).
\]
\end{theorem}
\begin{theorem} [Large effective rank] \label{thm:large-consistencyPi} Assume that $r_{\mathrm{eff}} > \frac{r}{\kappa_p^2}\big(\frac{\lambda_p}{\delta_{p}} \big)^2$.  For any constants $K,r, \kappa_p,$ there are constants $C_0, C_3 $ so that if $ n \geq C_3 \cdot r_{\mathrm{eff}} \frac{\lambda_p}{\delta_{p}}$, then the following holds.  
\[
\left\| \tilde{\Pi}_p  - \Pi_p \right\|= o(1)\,\,\,\text{if and only if}\,\, \sqrt{\frac{r_{\mathrm{eff}}}{n}}= o(1).
\]
\end{theorem}
  For the explicit dependence of $C_0$ and $C_3$ on $K,r, \kappa_p$, see Remark \ref{remark:explicit}. 

\section{Remarks and Open Questions} \label{subsec: existPopu}

 \subsection{Upper bounds} 
 As mentioned in the introduction, most 
 previous results using perturbation theory 
 provide only upper bounds, with no matching lower bound. Among these results, the most recent and closely related to ours is that of \cite{JW2}.
We refer to this paper for a discussion of earlier results.

In \cite{JW2},  given an index $p $ (independent of $n$), the authors defined the relative rank
$$\textbf{r}_{(p)} := \sum_{d \geq j \neq p \geq 1} \frac{\lambda_j}{|\lambda_j - \lambda_p|} + \frac{\lambda_p}{\delta_{(p)}}. $$ They  used this notion to obtain new bounds on $\| \tilde u_p \tilde u_p^\top -u_p u_p^\top \|$; see  \cite[Theorem 7]{JW2}.

In the common setting where both this theorem and ours apply, it turns out that the relative norm ${\bf r}_{(p)} $ 
and the quantity $\max \{ r_{\mathrm{eff}}, \frac{\lambda_p}{\delta_{(p)} } \}$ are comparable. In this situation, one can derive the upper bounds in Theorems \ref{thm:smallp}, \ref{thm:mediump}, and \ref{thm:largep} from \cite[Theorem 7]{JW2}. 
However, \cite[Theorem 7]{JW2} does not provide lower bounds. 

As far as the technical assumptions are concerned,  \cite[Theorem 7]{JW2} requires a stronger bound on the number of samples. 
To be precise, this theorem requires 
$$ \frac{n}{  r_{\mathrm{eff}}^2+ \big(\frac{\lambda_p}{\delta_{(p)}}\big)^2  } \rightarrow \infty, $$ 
while our theorems require 
$$ n \ge C \left[r_{\mathrm{eff}}\cdot \frac{\lambda_p}{\delta_{(p)}} + \big(\frac{\lambda_p}{\delta_{(p)}}\big)^2 \right], $$ for a sufficiently large constant  $C$.

Our requirement is always weaker, and the difference between the two bounds can be significant in the case when the effective rank 
$r_{\mathrm{eff}} $ is significantly larger than the ratio $\lambda_p / \delta_{(p)}$.  
On the other hand, the assumption on the vector $X$ in \cite{JW2} is more general.

\subsection{Consistency conditions}

In principle, an upper bound provides a sufficient condition for consistency. 

As far as necessary and sufficient conditions are concerned,  the latest result we can find is again from \cite{JW2}. Considering the leading eigenvector, the authors showed that 
$\tilde u_1$ is consistent 
if and only if $\frac{n} {\textbf{r}_{(1)}^2} \rightarrow \infty$. However, 
in order for this to apply, they also need to assume that
$n $ is super-quadratic 
in $d$, namely, $\frac{n}{d^2} \rightarrow \infty$; see \cite[Theorem 8] {JW2}. 

The authors pointed out that the 
extension to general $\tilde u_p$ is technical, as illustrated in \cite[Section 5.4]{JW2}. Our theorems, on the other hand, hold for all $p$ with the same proof. 

Both  \cite{JW2} and our paper use the contour integral formula to represent the error in question. However, the analysis of the integral is different, leading to different quantitative estimates and conditions.

\subsection{The value of \texorpdfstring{$p$}{p}}

In the introduction, we focus on estimators of the leading eigenvectors $u_p$, where $p$ is small. This is indeed the case of most interest in practice. On the other hand, our technical analysis does not require $p$ to be bounded. For instance, one can have $p$ tending to infinity with $d$  in all of our theorems.

\subsection{Extensions with general \texorpdfstring{$K$}{K}, \texorpdfstring{$r$}{r}, and \texorpdfstring{$\kappa_p$}{kp}}

In order to make the presentation less technical, in all our theorems, we assumed $K,r,\kappa_p=O(1)$.
However, we can also deal with the general case when these parameters tend to infinity with $d$. All one needs to do is to use the explicit formulae in the upper bound and lower bound theorems, which treat 
$K, r, \kappa_p$ as general unbounded parameters; see Remark~\ref{remark:explicit} and Appendix \ref{subsection:routine}. 
In this case, the consistency condition needs to be adjusted. We need to change the condition
\(
\frac{\lambda_p}{\sqrt{n}\,\delta_p}=o(1),
\,
\sqrt{\frac{r_{\mathrm{eff}}}{n}}=o(1),
\) to
\[
\frac{\lambda_p}{\sqrt{n}\,\delta_p}
=
o\!\left(\frac{1}{K\kappa_p r^{3/2}}\right),
\,
\sqrt{\frac{r_{\mathrm{eff}}}{n}}
=
o\!\left(\frac{1}{K\kappa_p r^{3/2}}\right),
\] respectively. We leave the routine derivation as an exercise. 

\subsection{Samples with missing and noisy entries}
With some additional works, we can apply the method in this paper to treat even more general models, where the observed samples $X_1, X_2, \dots, X_n$  have missing coordinates and are affected by noise.  Details will appear in a subsequent manuscript.

\subsection{Different models of \texorpdfstring{$X$}{X}} Technically, our method first establishes deterministic perturbation bounds for eigenvectors and eigenspaces; see Subsections~\ref{subsec: upfromabove}--\ref{subsec: uplower setting}. The order-of-magnitude and consistency theorems then follow from estimates of the skewness parameters (Definition~\ref{def: xyw}) and $\|E\|$. Therefore, with appropriate modifications to the analysis of these quantities, our approach naturally extends to other models of $X$,  such as the model
\(
X=UY+Z,
\)
where $Y$ and $Z$ are independent random vectors satisfying
\(
\E YY^\top +
\E ZZ^\top=\Lambda.
\)
The sub-Gaussian assumption on the coordinates of $Y$ can also be relaxed; for example it suffices to assume bounded fourth moments.

\subsection{An open question}

The case when we do not have matching upper and lower bounds is when the effective rank is at least medium, and the number 
$n$ of samples is relatively small.  We leave this case as an open question. Here is a concrete question for the estimation of $u_1$: 

{\bf \noindent Open question.} What is the order of magnitude of $\|\tilde{u}_1 \tilde{u}_1^\top - u_1 u_1^\top \|$ if  $\frac{r\lambda_1}{\delta_1} < r_{\mathrm{eff}} $
and   $n \le  r_{\mathrm{eff}}  \frac{\lambda_1}{\delta_1} $.

\section{Our general strategy} \label{sec: proof}

We use Theorem \ref{theo: upUpperM1/2 subGau} and Theorem \ref{theo: uplowerboundM1/2 subGau} to deduce Theorems \ref{thm:small-dependentp}, \ref{thm:medium-dependentp}, and \ref{thm:large-dependentp}. It is easy to use Theorems \ref{thm:small-dependentp}, \ref{thm:medium-dependentp}, and \ref{thm:large-dependentp} to deduce Theorems \ref{thm:smallp}, \ref{thm:mediump}, and \ref{thm:largep}.  The derivation of Theorems \ref{thm:small-dependentp}, \ref{thm:medium-dependentp}, and \ref{thm:large-dependentp} from 
Theorem \ref{theo: upUpperM1/2 subGau} and Theorem \ref{theo: uplowerboundM1/2 subGau} requires only a routine calculation, which we will perform in Subsection \ref{subsection:routine}.

The main task is to prove Theorem \ref{theo: upUpperM1/2 subGau} and Theorem \ref{theo: uplowerboundM1/2 subGau}. 
We will lay the groundwork for these proofs in the 
rest of this section. Section~\ref{subsec: proofupperM1/2} is devoted to the technical ingredients for the proof of Theorem~\ref{theo: upUpperM1/2 subGau} (upper bound), while Section~\ref{subsec: prooflowerM1/2Line1} develops the corresponding ingredients for the proof of Theorem~\ref{theo: uplowerboundM1/2 subGau} (lower bound).

Roughly speaking, the proof consists of three conceptual steps.
First, we follow the contour expansion (the first expansion) in \cite[Section 4]{DKTranVu1} to present $\tilde{u}_p \tilde{u}_p^\top - u_p u_p^\top$ as a sum of infinite terms (Subsection \ref{subsec: contour expansion}).
Next, we compute almost precisely the first two terms of the expansion, which contain the dominant contribution (Subsections \ref{subsec: F1}-\ref{subsec: F2}).
Finally, using the combinatorial expansion (the second expansion), we show that the sum of all remaining terms is comparable to the first two terms (Subsection \ref{subsubsec: Fs}).

The treatment for eigenspaces ($\tilde \Pi_p$) is similar, and will be discussed in Subsection \ref{subsec: PipTreatment}.

\subsection{A few important parameters}

We view the sample covariance matrix $\tilde{M}$ as a perturbation of $M$, defining $E= \tilde M -  M$. Clearly, $E$ is a mean-zero random matrix. Given the target index $p$, define the set of \emph{important indices}
\[ \textstyle
I_p:=\left\{i\in[d]: |\lambda_i-\lambda_p|\le \frac{\lambda_p}{2}\right\}
\, \,\text{with}\,\, r:=|I_p|.
\]
By definition, for every $i\in I_p$,
\begin{equation} \label{lambdaivs lambdap}
    \textstyle
\frac{\lambda_p}{2}\le \lambda_i \le \frac{3\lambda_p}{2}.
\end{equation}

To control the  interaction between $E$ and the eigenvectors of $M$, we define the following skewness quantities:

\begin{definition} \label{def: xyw} For each given $p$, define 
\begin{itemize}
\item $x:= \max_{i,j \in I_p} \norm{u_i^\top E u_j}.$
\item $y:= \textstyle \max_{\substack{ i\neq j \in I_p}} \norm{ u_i^\top E \left(\sum_{l \notin I_p  } \frac{u_l u_l^\top}{\lambda_p - \lambda_l} \right) E u_j}$.
\item $w:= \max_{i \in I_p} \|E u_i\|.$
\end{itemize}
\end{definition}
\subsection{Contour expansion of \texorpdfstring{$\tilde{u}_p \tilde{u}_p^\top - u_p u_p^\top$}{tildeup-up}} \label{subsec: contour expansion}  Let us recall Cauchy's integral theorem~\cite{CAbook, agarwal2011introduction}.
 \begin{theorem}[Cauchy's integral theorem] \label{theo: Cauchy}
Given a simple closed contour $\Gamma$ and a complex number $\lambda$, one has
\begin{equation}\label{Cauchy0} 
   \frac{1}{2 \pi {\bf i}} \int_{\Gamma} \frac{1}{(z-\lambda)}\,dz 
   = 
   \begin{cases}
      1, & \,\text{if}\,\lambda \text{ is inside } \Gamma, \\[4pt]
      0, & \,\text{if}\, \lambda \text{ is outside } \Gamma.
   \end{cases}  
\end{equation}
 Here and later, {\bf i} denotes $\sqrt {-1} $. 
\end{theorem}
Returning to our setting, let  $\Gamma $ be a contour containing only $\lambda_p$; i.e., all $\lambda_j, j \neq p$ are outside $\Gamma$. By Cauchy's theorem, one obtains the classical contour formula \cite{Book1, Kato1}:
 \begin{equation} \label{contour-formula} 
     \frac{1} {2 \pi {\bf i }} \int_{\Gamma}    (zI-M)^{ -1} dz = \sum_{\lambda_i\,inside\, \Gamma} u_i u_i^\top = u_p u_p^\top. \end{equation} 
Applying a  similar argument to $\tilde M$, we have 
 \begin{equation} \label{contour-formula1} 
  \frac{1}{2 \pi {\bf i}}  \int_{\Gamma}   (zI-\tilde M)^{-1} dz  = \sum_{ \tilde{\lambda}_i \,\,inside\,\,\Gamma }  \tilde u_i  \tilde u_i^\top.  \end{equation} 
  Thus, we obtain a contour representation for the perturbation 
 \begin{equation}  \label{contourrep}
\big(\sum_{ \tilde{\lambda}_i \,\,inside\,\,\Gamma }  \tilde u_i  \tilde u_i^\top \big) - u_p u_p^\top =  \frac{1} {2 \pi {\bf i} } \int_{\Gamma}   [(z-\tilde M)^{-1}- (z- M)^{-1} ]  dz. 
  \end{equation} 
To recover the perturbation
$\tilde u_p \tilde u_p^\top - u_p u_p^\top$,
the contour $\Gamma$ must enclose only $\tilde\lambda_p$.
As shown later in Section~\ref{subsec: dim argument}, this is ensured with probability at least $1-e^{-4r_{\mathrm{eff}}t^2} -\frac{2r^2}{t^2}$ by 
\begin{itemize}
       \item requiring
\(
\textstyle \max\!\left\{
\sqrt p\,\sqrt{\frac{r_{\mathrm{eff}}}{n}},
\;
r\frac{\lambda_p}{\sqrt n\,\delta_{(p)}},
\;
\sqrt r\,\sqrt{\frac{\lambda_p r_{\mathrm{eff}}}{n\delta_{(p)}}}
\right\}
<
\frac{1}{C\kappa_p K t},
\)
for some constant $C>0$,
\item setting $\Gamma$ to be a rectangle whose vertical edges bisect the intervals
$(\lambda_{p+1},\lambda_{p})$ and $(\lambda_p,\lambda_{p-1})$, and whose horizontal edges lie at distance $4\delta_{(p)}$ from the real axis.
\end{itemize}

%
Thus, under the assumption of Theorem \ref{theo: upUpperM1/2 subGau} and this construction of $\Gamma$, we have 
$$ \tilde u_p  \tilde u_p^\top - u_p u_p^\top =  \frac{1} {2 \pi {\bf i} } \int_{\Gamma}   [(z-\tilde M)^{-1}- (z- M)^{-1} ]  dz.$$
 
Using  the resolvent formula 
\(
 A^{-1} - (A+B)^{-1}= (A+B)^{-1} B A^{-1} 
\)
and the fact that $zI- M = (zI - \tilde M) +E$, we obtain 
\begin{equation*} (zI-\tilde M)^{-1} - (zI- M)^{-1} = (zI-M)^{-1}  E (zI-\tilde M)^{-1} .  \end{equation*} 
\noindent Applying this identity repeatedly, we obtain 
\begin{equation} \label{TaylorEx} (zI-\tilde M)^{-1} - (zI- M)^{-1} =\sum_{s=1}^{\infty}  (zI-M)^{-1} [ E(zI-M)^{-1} ]^s .
\end{equation} 
The series on the RHS is convergent under the assumption of Theorem \ref{theo: upUpperM1/2 subGau} via \cite[Section 8.4]{DKTranVu1}. Therefore, 
\begin{equation} \label{up Perturb Iden}
  \tilde{u}_p \tilde{u}_p^\top - u_p u_p^\top = \sum_{s=1}^{\infty} F_s,\,\,\text{where}\,F_s:=\frac{1}{2 \pi \textbf{i}} \int_{\Gamma} (zI-M)^{-1}[E (zI-M)^{-1}]^s dz.  
\end{equation}

\subsection{Bounding \texorpdfstring{$\|\tilde{u}_p \tilde{u}_p^\top -u_p u_p^\top\|$}{tildeup-up} from above} \label{subsec: upfromabove}
By \eqref{up Perturb Iden} and the triangle inequality, we have 
\[
\|\tilde{u}_p \tilde{u}_p^\top - u_p u_p^\top\| \leq \|F_1\|+\|F_2\|+\sum_{s > 2}\|F_s\|.
\]
Thus, the heart of the matter is to compute $\|F_s\|$. 
Technically, we  are going to  estimate $F_1 $ and $F_2$ almost precisely, and obtain sharp bounds, 
 $\| F _1\| \le h_1$ and $\| F_2  \| \le h_2 $.
 The sum $h:= h_1 +h_2$ (with some technical estimates later in Section \ref{subsec: proofupperM1/2}) will contain all the terms in Theorem \ref{theo: upUpperM1/2 subGau}.  We next  show that the remaining terms $F_s$ satisfy
 \begin{equation} \label{Fs} \| F_s \| = O \left( \frac{1}{(C/6)^{s-1}} h_1 + \frac{s+2}{(C/2)^{s-2}} h_2 \right).\end{equation}
Here, $C$ is a real number such that 
\begin{equation} \label{def: C}
     \max \left\lbrace \frac{ \|E\|}{\lambda_p/2}, \frac{r x}{\delta_{(p)}}, \frac{\sqrt{r} w}{\sqrt{\lambda_p \delta_{(p)}/2}} \right\rbrace < \frac{1}{C}.
\end{equation}

\subsubsection{Estimation of  $\|F_1\|$} \label{subsec: F1} The next idea is to split the resolvent $(zI-M)^{-1}=\sum_{i=1}^d \frac{u_i u_i^\top}{z -\lambda_i}$ into 
$$\textstyle P+Q,\,\text{where}\,\,P:= \sum_{i \in I_p} \frac{u_i u_i^\top}{z -\lambda_i}\, \text{and}\,Q:=\sum_{j \notin I_p} \frac{u_j u_j^\top}{z -\lambda_j}.$$
Here $P$ corresponds to the important eigenspace near $\lambda_p$, whereas $Q$ corresponds to the remaining spectrum. We refer the reader to \cite[Section~4.2]{DKTranVu1} for a deeper intuition behind the decomposition into $P$ and $Q$, as well as the derivation of the quantities $x$, $y$, and $w$.

For our estimate of $F_1$, we have
$$F_1:= \frac{1}{2 \pi \textbf{i}} \int_{\Gamma} (zI-M)^{-1}E (zI-M)^{-1} dz = \frac{1}{2 \pi \textbf{i}} \int_{\Gamma} PEP dz + \frac{1}{2 \pi \textbf{i}} \int_{\Gamma} (PEQ +QEP) dz  + \frac{1}{2 \pi \textbf{i}} \int_{\Gamma} QEQ dz.$$
Next, using the explicit computation of $F_1$ from
\cite[Section~7.2, pp.~32--33]{DKTranVu1}\footnote{
In \cite{DKTranVu1}, the formula is stated in terms of a subset
$S\subset[d]$ and the associated important set $N_{\bar\lambda}(S)$.
In our setting, these correspond to $\{p\}$ and $I_p$, respectively.
}, we obtain 
\begin{equation*}
    \begin{split}
   & \textstyle \frac{1}{2 \pi \textbf{i}} \int_{\Gamma} PEP dz =  \sum_{j\in I_p\setminus\{p\}}
\left(
u_p \frac{u_p^\top E u_j}{\lambda_p-\lambda_j} u_j^\top
+
u_j \frac{u_j^\top E u_p}{\lambda_p-\lambda_j} u_p^\top
\right), \frac{1}{2 \pi \textbf{i}} \int_{\Gamma} QEQ dz =0,  \\
&\textstyle \frac{1}{2 \pi \textbf{i}} \int_{\Gamma} PEQ dz = u_pu_p^\top E
\left(
\sum_{l\notin I_p}
\frac{u_lu_l^\top}{\lambda_p-\lambda_l}
\right), \frac{1}{2 \pi \textbf{i}} \int_{\Gamma} QEP dz =  \left(
\sum_{l\notin I_p}
\frac{u_lu_l^\top}{\lambda_p-\lambda_l}
\right)
E u_pu_p^\top.
    \end{split}
\end{equation*}
Thus, 
\[
F_1=M_1+M_2, \,\,\text{where}
\]
\begin{equation}\label{iden: F1M1M2}
\begin{split}
M_1
:={}&
\sum_{j\in I_p\setminus\{p\}}
\left(
u_p \frac{u_p^\top E u_j}{\lambda_p-\lambda_j} u_j^\top
+
u_j \frac{u_j^\top E u_p}{\lambda_p-\lambda_j} u_p^\top
\right),
\\
M_2
:={}&
u_pu_p^\top E
\left(
\sum_{l\notin I_p}
\frac{u_lu_l^\top}{\lambda_p-\lambda_l}
\right)
+
\left(
\sum_{l\notin I_p}
\frac{u_lu_l^\top}{\lambda_p-\lambda_l}
\right)
E u_pu_p^\top .
\end{split}
\end{equation}

By \cite[Section 7.2 - Estimates (50)(51)(52)]{DKTranVu1}, we further have 
\begin{equation} \label{est: tranvuM1M2F1}
    \begin{split}
        &\|M_1\| \leq \frac{2\sqrt{r} x}{\delta_{(p)}},\, \|M_2\| \leq \frac{2\|E\|}{\lambda_p/2}=\frac{4\|E\|}{\lambda_p}, \text{and hence}\,\|F_1\| \leq \frac{2\sqrt{r} x}{\delta_{(p)}} + \frac{4\|E\|}{\lambda_p}.
    \end{split}
\end{equation}
\subsubsection{Estimation of  $\|F_2\|$} \label{subsec: F2} Similarly, we also split $F_2$ into eight terms, i.e., 
\begin{equation*}
    \begin{split}
F_2 & = \textstyle\frac{1}{2 \pi \textbf{i}}\int_{\Gamma} (P+Q) E (P+Q) E (P+Q) dz  \\
&\textstyle = \frac{1}{2 \pi \textbf{i}}\int_{\Gamma} Q E Q E Q dz + \frac{1}{2 \pi \textbf{i}}\int_{\Gamma} Q E P E Q dz  +\frac{1}{2 \pi \textbf{i}}  \int_{\Gamma} P E Q E P dz + \frac{1}{2 \pi \textbf{i}} \int_{\Gamma} P E P E P dz   \\
&\textstyle + \frac{1}{2 \pi \textbf{i}}\int_{\Gamma} PE Q E Q  dz + \frac{1}{2 \pi \textbf{i}} \int_{\Gamma} PE P E Q  dz + \frac{1}{2 \pi \textbf{i}} \int_{\Gamma} QE Q E P  dz + \frac{1}{2 \pi \textbf{i}} \int_{\Gamma} QE P E P  dz.      \end{split}
\end{equation*}
Using \cite[Section 7.3 - Estimates (54)(56)(60)]{DKTranVu1} and the setting in \eqref{def: C}, we have 
\begin{equation} \label{est: F2notPEQEP}
    \begin{split}
        & \textstyle  \frac{1}{2 \pi \textbf{i}}\int_{\Gamma} Q E Q E Q dz = 0;  \|\frac{1}{2 \pi \textbf{i}} \int_{\Gamma} P E P E P dz\| \leq \frac{6}{C} \cdot \frac{\sqrt{r} x}{\delta_{(p)}}; \\
        & \textstyle \| \frac{1}{2 \pi \textbf{i}}\int_{\Gamma} Q E P E Q dz \|, \|\frac{1}{2 \pi \textbf{i}}\int_{\Gamma} PE Q E Q  dz \|, \|\frac{1}{2 \pi \textbf{i}} \int_{\Gamma} QE Q E P  dz \| \leq \frac{1}{C} \cdot \frac{\|E\|}{\lambda_p/2} = \frac{2}{C} \cdot \frac{\|E\|}{\lambda_p}; \\
        &\text{and}\,\, \textstyle \|\frac{1}{2 \pi \textbf{i}} \int_{\Gamma} PE P E Q  dz \|, \|\frac{1}{2 \pi \textbf{i}} \int_{\Gamma} QE P E P  dz\| \leq \frac{3}{C} \cdot \frac{\|E\|}{\lambda_p/2} = \frac{6}{C} \cdot \frac{\|E\|}{\lambda_p}.
    \end{split}
\end{equation}
The troublesome term is $\frac{1}{2\pi \textbf{i}}\int_{\Gamma} PEQEP dz$, which by \cite[Estimate (55)]{DKTranVu1}, equals
\begin{equation} \label{splitPEQEP}
    \begin{split}
& \sum_{ \substack{j \notin I_p}} \frac{-1}{(\lambda_p -\lambda_j)^2} u_{p} u_{p}^\top E u_j u_j^\top E u_p u_p^\top + M_3,\, \text{where} \\
& M_3:= \sum_{ \substack{ j \in I_p \setminus\{p\}} } \frac{ u_{j} u_{j}^\top E}{\lambda_p -\lambda_{j}} \left(\sum_{ l \notin I_p } \frac{  u_l u_l^\top }{\lambda_p -\lambda_l} \right) E u_p u_p^\top + u_p u_p^\top E \left(\sum_{ l \notin I_p } \frac{  u_l u_l^\top }{\lambda_p -\lambda_l} \right) \frac{E u_j u_j^\top}{\lambda_p -\lambda_j}.
\end{split}
\end{equation}
By \cite[Estimate (56)]{DKTranVu1}, we further have
\begin{equation} \label{est: F2PEQEP}
    \begin{split}
       &\textstyle \big|\sum_{ \substack{j \notin I_p}} \frac{-1}{(\lambda_p -\lambda_j)^2} u_{p} u_{p}^\top E u_j u_j^\top E u_p u_p^\top  \big| \leq \frac{1}{C} \cdot \frac{\|E\|}{\lambda_p/2} = \frac{2}{C} \cdot \frac{\|E\|}{\lambda_p}; \|M_3\| \leq \frac{2\sqrt{r} y}{\delta_{(p)}}, \\
       &\textstyle \text{and hence}\,\,\|\frac{1}{2\pi \textbf{i}}\int_{\Gamma} PEQEP dz\| \leq \frac{2}{C} \cdot \frac{\|E\|}{\lambda_p} + \frac{2\sqrt{r} y}{\delta_{(p)}}.
    \end{split}
\end{equation}

Combining \eqref{est: F2notPEQEP} and \eqref{est: F2PEQEP}, we finally obtain 
\begin{equation} \label{F2 est}
    \begin{split}
    \|F_2\| \leq \frac{2 \times 3+ 2 \times 6+2}{C} \cdot \frac{\|E\|}{\lambda_p} + \frac{6}{C} \cdot \frac{\sqrt{r} x}{\delta_{(p)}} + \frac{2\sqrt{r} y}{\delta_{(p)}} \leq \frac{20}{C} \cdot \left( \frac{\|E\|}{\lambda_p} + \frac{\sqrt{r} x}{\delta_{(p)}} \right) + \frac{2\sqrt{r} y}{\delta_{(p)}}.
    \end{split}
\end{equation}

\subsubsection{Estimation of  $ \| F_s\| $  for a general $s \geq 3$.} \label{subsubsec: Fs} We  expand $(P+Q)[ E(P+Q)]^s$ into the sum of $2^{s+1}$ operators, each of which is a product of  alternating $Q$-blocks and 
$P$-blocks 
$$(QEQE \dots QE)(PE PE \dots PE) \dots (PE PE \dots PE) ( QE  QE \dots Q), $$ where we allow the first and last blocks to be empty. 

We code each operator like this by the numbers of $Q$'s and the numbers of $P$'s in each blocks. If there are $(k+1)$  $Q$-blocks and $k$-$P$ blocks, for some integer $k$,  we let 
$\alpha_1, \dots, \alpha_{k+1} $ and $\beta_1, \dots, \beta_k$ be these numbers. These numbers satisfy the following conditions 
$$ \alpha_1, \alpha_{k+1} \geq 0,$$
$$ \alpha_i, \beta_j \geq 1,\,\, \forall 1<i<k+1,\,\, \forall  1\le j \le k, $$
$$\alpha_1+...+\alpha_{k+1}+\beta_1 +...+\beta_k=s+1.$$

In what follows, we set $\alpha =(\alpha_1,...,\alpha_{k+1}) \in \mathbb{Z}^{k+1} $ and $\beta=(\beta_1,...,\beta_k) \in \mathbb{Z}^{k}$, and use $M(\alpha; \beta)$ to denote the corresponding operator.

\noindent \textit{Example:} given $s=10, k=2, \alpha_1=3, \alpha_2=2, \alpha_3=1, \beta_1=2, \beta_2=3$, we have 
 $$\int_{\Gamma} M(3,2,1;2,3) dz = \int_{\Gamma} \underset{\substack{\alpha_1=3\\ \text{numbers of $Q$}}}{\underbrace{(QEQEQE)}} \underset{\substack{\beta_1=2\\ \text{numbers of $P$}}} {\underbrace{(PEPE)}}\,\,\, \underset{\substack{\alpha_2=2 \\ \text{numbers of } \, Q}}{\underbrace{(QEQE)}} \underset{\substack{\beta_2=3 \\ \text{numbers of}\, P}}{\underbrace{(PEPEPE)}} \,\,\, \underset{\substack{\alpha_3=1 \\ \text{numbers of}\, Q}}{\underbrace{Q}} dz.$$

We bound  $\Norm{\int_{\Gamma}M(\alpha, \beta) dz} $ for each pair $\alpha, \beta$ separately, and use the triangle inequality to add up the bounds. 
For $s$ and $k$ fixed, we split the collection of pairs $(\alpha;\beta)$  into the following types, according to the values of $\alpha_1$ and $\alpha_{k+1}$
\begin{itemize}
    \item Type I: $\alpha_1, \alpha_{k+1} > 0$.
    \item Type II: $\alpha_1, \alpha_{k+1} = 0$. 
    \item Type III: either $(\alpha_1 = 0, \alpha_{k+1} > 0)$ or $(\alpha_1 > 0, \alpha_{k+1} = 0)$. 
\end{itemize}
Let $s_1 :=  \sum_{i=1}^{k+1} \alpha_i$ and  $s_2 := \sum_{j=1}^{k} \beta_j $. In \cite[Lemmas 7.2-7.4]{DKTranVu1}, Tran and Vu proved the following lemmas, bounding  $\Norm{\int_{\Gamma}M(\alpha, \beta) dz}$ with respect to the above three types. We assume that the assumption of Theorem \ref{theo: upUpperM1/2 subGau} holds in all three lemmas, which is equivalent to  \(
\textstyle   \max \left\lbrace \frac{ \|E\|}{\lambda_p/2}, \frac{r x}{\delta_{(p)}}, \frac{\sqrt{r} w}{\sqrt{\lambda_p \delta_{(p)}/2}} \right\rbrace < \frac{1}{C}\) \footnote{In the original version, all lemmas were stated with constant $12, \bar\lambda, \delta_S$ instead of $C, \lambda_p/2, \delta_{(p)}$.}.
\begin{lemma} \label{LemmaCase1M(a,b)bound} For a pair $(\alpha, \beta)$ of Type I,   
\begin{equation} \label{Case1boundM(a,b)bound}
 \textstyle   \frac{1}{2 \pi} \Norm{\int_{\Gamma} M(\alpha;\beta) dz} \leq \left( \frac{ \|E\|}{\lambda_p/2} + \frac{x}{\lambda_p/2} \right) \frac{2^{s_2+s-1}}{C^{s-1}}.
\end{equation}
\end{lemma}

\begin{lemma}\label{LemmaCase2M(a,b)bound}
  For a pair $(\alpha, \beta)$ of Type II, 
    \begin{itemize}
        \item If $M(\alpha; \beta) \neq PEQEPE ... QEP$, i.e.,  $(s_1,s_2) \neq (k-1,k)$, then 
        \begin{equation} \label{Cas2M(a,b)boundgood}
   \textstyle  \frac{1}{2 \pi} \Norm{ \int_{\Gamma} M(\alpha; \beta) dz}  \le  \left(\frac{ \|E\|}{\lambda_p/2} + \frac{\sqrt{r} x}{\delta_{(p)}} \right) \frac{2^{s_2+s-1}}{C^{s-1}}, 
 \end{equation}
 \item If $M(\alpha;\beta) = PEQEPEQE...EQEP $, i.e., $(s_1,s_2) = (k-1,k)$, then 
 \begin{equation} \label{Case2M(a,b)boundbad}
  \textstyle  \frac{1}{2 \pi} \Norm{ \int_{\Gamma} PEQEPEQE...EQEP dz}  \le \left( \frac{x}{\lambda_p/2} + \frac{ \|E\|}{\lambda_p/2} \right) \frac{2^{3s/2} }{C^{s-1}}  +   \frac{(s+2) \sqrt{r} y}{\delta_{(p)}} \left( \frac{2}{C} \right)^{s-2}. 
 \end{equation}
    \end{itemize}
     \end{lemma}

\begin{lemma} \label{LemmaCase3M(a,b)bound}
    For a pair $(\alpha, \beta)$ of Type III, 
    \begin{equation} \label{Case3M(a,b)bound}
\begin{split}
\textstyle \frac{1}{2 \pi} \Norm{\int_{\Gamma} M(\alpha;\beta) dz} \leq \left( \frac { \|E\|}{\lambda_p/2}  + \frac{ x}{\lambda_p/2} \right) \frac{2^{s_2+s-1}}{C^{s-1}}.
\end{split}
\end{equation} 
\end{lemma}

Back to estimate $\|F_s\|$. For any integer $ 1 \le s_2 \leq s$,  the expansion of $F_s$ contains  $\binom{s+1}{s_2}$ operators $M(\alpha;\beta)$, each of which has exactly $s_2$ $P$ operators. A simple consideration reveals that among these,  $\binom{s-1}{s_2}$ are of Type I,  $\binom{s-1}{s_2-2}$ are of Type II, and $2\binom{s-1}{s_2 -1}$ are of Type III. The only term in the expansion with $s_2=0$ vanishes by Cauchy's theorem.

 Putting \eqref{Case1boundM(a,b)bound}, \eqref{Cas2M(a,b)boundgood}, \eqref{Case2M(a,b)boundbad}, \eqref{Case3M(a,b)bound} together and summing over $s_2$,  we show that  $  \| F_s \|$ is at most 
 \begin{equation*}
 \begin{split}
  & \left[\sum_{s_2=1}^{s+1} \binom{s-1}{s_2}   \left( \frac{ \|E\|}{\lambda_p/2} + \frac{x}{\lambda_p/2} \right) \frac{2^{s_2+s-1}}{C^{s-1}}\right] +   \bigg[ \left( \frac{x}{\lambda_p/2} + \frac{ \|E\|}{\lambda_p/2} \right) \frac{2^{3s/2} }{C^{s-1}}+  \frac{(s+2) \sqrt{r} y}{\delta_{(p)}} \left( \frac{2 }{C} \right)^{s-2} \bigg] \\
    & + \left[ \sum_{s_2=1}^{s+1} \binom{s-1}{s_2-2}  \left(\frac{ \|E\|}{\lambda_p/2}  + \frac{\sqrt{r} x}{\delta_{(p)}} \right) \frac{2^{s_2+s-1}}{C^{s-1}} \right] + \left[ \sum_{s_2=1}^{s+1} 2\binom{s-1}{s_2 -1} \left( \frac { \|E\|}{\lambda_p/2}  + \frac{ x}{\lambda_p/2} \right) \times \frac{2^{s_2+s-1}}{C^{s-1}}  \right],
  \end{split}     
 \end{equation*}  
 where $\binom{s-1}{l} = 0$ for either $l > s-1$ or $l < 0$.
 Consider the first term 
 \begin{equation*}
     \begin{split}
         \sum_{s_2=1}^{s+1} \binom{s-1}{s_2}   \left( \frac{ \|E\|}{\lambda_p/2} + \frac{x}{\lambda_p/2} \right) \frac{2^{s_2+s-1}}{C^{s-1}} & = \sum_{s_2=1}^{s+1} \binom{s-1}{s_2} 2^{s_2}  \left( \frac{ \|E\|}{\lambda_p/2} + \frac{x}{\lambda_p/2} \right) \frac{2^{s-1}}{C^{s-1}} \\
         & = \left( \sum_{s_2=1}^{s+1} \binom{s-1}{s_2} 2^{s_2} 1^{s-1-s_2}  \right) \times \left( \frac{ \|E\|}{\lambda_p/2} + \frac{x}{\lambda_p/2} \right) \frac{2^{s-1}}{C^{s-1}}.
  \end{split}
 \end{equation*}

Moreover, 
\(
\sum_{s_2=1}^{s+1} \binom{s-1}{s_2} 2^{s_2} 1^{s-1-s_2}  \leq \sum_{s_2=0}^{s-1}  \binom{s-1}{s_2} 2^{s_2} 1^{s-1-s_2}=3^{s-1},
\)
The RHS is at most
\[
 3^{s-1}  \times \left( \frac{ \|E\|}{\lambda_p/2} + \frac{x}{\lambda_p/2} \right) \frac{2^{s-1}}{C^{s-1}}  = \left( \frac{ \|E\|}{\lambda_p/2} + \frac{x}{\lambda_p/2} \right) \frac{1}{(C/6)^{s-1}}. 
\]

Using a similar argument, we have 
$$\sum_{s_2=1}^{s+1} \binom{s-1}{s_2-2}  \left(\frac{ \|E\|}{\lambda_p/2} + \frac{\sqrt{r} x}{\delta_{(p)}} \right) \frac{2^{s_2+s-1}}{C^{s-1}} = 4  \left(\frac{ \|E\|}{\lambda_p/2} + \frac{\sqrt{r} x}{\delta_{(p)}} \right) \frac{1}{(C/6)^{s-1}},   $$
(here we pair up $2^{s_2-2}$ with $\binom{s-1}{s_2-2}$, the remaining factor is $4$),
and
$$\sum_{s_2=1}^{s+1} 2\binom{s-1}{s_2 -1} \left( \frac { \|E\|}{\lambda_p/2}  + \frac{ x}{\lambda_p/2} \right) \times \frac{2^{s_2+s-1}}{C^{s-1}} =  4 \left( \frac { \|E\|}{\lambda_p/2}  + \frac{ x}{\lambda_p/2} \right) \times \frac{1}{(C/6)^{s-1}},$$
(we pair up $2^{s_2-1}$ with $\binom{s-1}{s_2-1}$, the remaining factor is $2 \times 2 =4$). Finally, we have 
$$ \left[ \left( \frac{x}{\lambda_p/2} + \frac{ \|E\|}{\lambda_p/2} \right) \frac{2^{3s/2} }{C^{s-1}}+  \frac{(s+2) \sqrt{r} y}{\delta_{(p)}} \left( \frac{2 }{C} \right)^{s-2} \right] \le \left[ \frac{8}{C} \left( \frac{x}{\lambda_p/2} + \frac{ \|E\|}{\lambda_p/2} \right) \frac{1 }{(C/3)^{s-2}} + \frac{ \sqrt{r} y}{\delta_{(p)}} \frac{s+2}{(C/2)^{s-2}}  \right],$$
where we use the trivial estimate that $\frac{2^{3/2}}{C} \leq \frac{3}{C}=\frac{1}{C/3}$. It follows that 
\begin{equation} \label{boundFs} 
    \begin{split}
       \Norm{F_s}  & \leq  \textstyle \left( \frac{ \|E\|}{\lambda_p/2} + \frac{x}{\lambda_p/2} \right) \frac{1}{(C/6)^{s-1}} +  \left[ \frac{8}{C} \left( \frac{x}{\lambda_p/2} + \frac{ \|E\|}{\lambda_p/2} \right) \frac{1 }{(C/3)^{s-2}} + \frac{ \sqrt{r} y}{\delta_{(p)}} \frac{s+2}{(C/2)^{s-2}}  \right]  +  \\
    & + \textstyle 4  \left(\frac{ \|E\|}{\lambda_p/2} + \frac{\sqrt{r} x}{\delta_{(p)}} \right) \frac{1}{(C/6)^{s-1}} + 4 \left( \frac { \|E\|}{\lambda_p/2}  + \frac{ x}{\lambda_p/2} \right) \frac{1}{(C/6)^{s-1}}.
    \end{split}
\end{equation}

\subsubsection{Putting all $\|F_s\|$-estimates  together}
 Summing \eqref{boundFs} over $s$, for $C > 24$, we can bound  $ \sum_{s\geq 3} \|F_s\|$ by 
 \begin{equation}  \label{F2-HR} 
 \frac{36}{C^2} \cdot \left[ 13 \frac{ \|E\|}{\lambda_p/2} + 6 \frac{\sqrt{r} x}{\delta_{(p)}}    + 8 \frac{x}{\lambda_p/2}     \right] + \frac{2}{C} \cdot 6 \frac{\sqrt{r} y}{\delta_{(p)}} \leq \frac{504}{C^2} \cdot \left[ \frac{ \|E\|}{\lambda_p/2} + \frac{\sqrt{r} x}{\delta_{(p)}}     \right] + \frac{12}{C} \cdot \frac{\sqrt{r} y}{\delta_{(p)}}.     
 \end{equation} 
(Here we omit the routine calculation and do not try to optimize the constants.)

Combining \eqref{est: tranvuM1M2F1}, \eqref{F2 est}, and \eqref{F2-HR}, we obtain 
 \begin{equation} \label{F3-HR} 
 \begin{split}
   \textstyle\|\tilde{u}_p \tilde{u}_p^\top - u_p u_p^\top   \| & \le \|F_1\|+\|F_2\|+\sum_{s\geq 3} \|F_s\| \\
   & \textstyle \leq \frac{ 4   \| E\|}{ \lambda_p}   + \frac{   2\sqrt{r} x} { \delta_{(p)} } + \frac{2\sqrt{r} y}{\delta_{(p)}} + \frac{20}{C}  \left( \frac{\|E\|}{\lambda_p} + \frac{\sqrt{r} x}{\delta_{(p)}}+\frac{\sqrt{r} y}{\delta_{(p)}} \right) + \frac{504}{C^2}  \left[ \frac{ \|E\|}{\lambda_p/2} + \frac{\sqrt{r} x}{\delta_{(p)}}     \right] .    
 \end{split}
 \end{equation} 
  If we set $C \geq 24$, \eqref{F3-HR} implies 
  \begin{equation} \label{F4-HR} 
 \|\tilde{u}_p \tilde{u}_p^\top - u_p u_p^\top   \|\le  \frac{7\| E \| }{ \lambda_p} + \frac{ 5 \sqrt{r} x } {\delta_{(p)}} +  \frac{ 5\sqrt{ r} y }{ \delta_{(p)} }.
 \end{equation} 

The detailed analysis of the right-hand side, and hence the completion of the proof of Theorem~\ref{theo: upUpperM1/2 subGau}, is deferred to Section~\ref{subsec: proofupperM1/2}.

\subsection{Bounding \texorpdfstring{$\|\tilde{u}_p \tilde{u}_p^\top - u_p u_p^\top\|$}{tildeup-up} from below} \label{subsec: uplower setting} We again combine \eqref{up Perturb Iden} with the triangle inequality that  
\begin{equation} \label{up lower ine1}
    \|\tilde{u}_p \tilde{u}_p^\top - u_p u_p^\top   \| \geq \|F_1\| - \|F_2\| - \sum_{s\geq 3}\|F_s\|.
\end{equation}
For any $C >0$ such that  \(
\textstyle   \max \left\lbrace \frac{ \|E\|}{\lambda_p/2}, \frac{r x}{\delta_{(p)}}, \frac{\sqrt{r} w}{\sqrt{\lambda_p \delta_{(p)}/2}} \right\rbrace < \frac{1}{C}\), by \eqref{iden: F1M1M2},  \eqref{F2 est}, and \eqref{F2-HR}, the RHS is bounded from below by 
\[
\|M_1+M_2\| - \left[ \frac{20}{C} \cdot \left( \frac{\|E\|}{\lambda_p} + \frac{\sqrt{r} x}{\delta_{(p)}} \right) + \frac{2\sqrt{r} y}{\delta_{(p)}} \right] - \left[\frac{504}{C^2} \cdot \left( \frac{ \|E\|}{\lambda_p/2} + \frac{\sqrt{r} x}{\delta_{(p)}}     \right) + \frac{12}{C} \cdot \frac{\sqrt{r} y}{\delta_{(p)}} \right].
\]
For a large enough $C$, we have 
\begin{equation} \label{Firstlowerbound0}
   \|\tilde{u}_p \tilde{u}_p^\top - u_p u_p^\top   \| \geq  \|M_1+M_2\| - \left[ \frac{21}{C} \cdot \left( \frac{\|E\|}{\lambda_p} + \frac{\sqrt{r} x}{\delta_{(p)}} \right) + \frac{3\sqrt{r} y}{\delta_{(p)}} \right]. 
\end{equation}

The detailed analysis of the right-hand side, and hence the completion of the proof of Theorem \ref{theo: uplowerboundM1/2 subGau}, is deferred to Section~\ref{subsec: prooflowerM1/2Line1}.

\subsection{Estimating \texorpdfstring{$\|\tilde{\Pi}_p-\Pi_p\|$}{tildepip-pip} - proofs of Theorem \ref{theo: Piupper} and Theorem \ref{theo: PipM1/2lower} } \label{subsec: PipTreatment}
The proofs of Theorem \ref{theo: Piupper} and Theorem \ref{theo: PipM1/2lower} follow the same strategy presented in the previous subsections with the following technical adjustments.

First, the target set of indices is no longer $\{p\}$, but $[p]:=\{1,2,\dots, p\}$. Consequently, instead of $I_p, \delta_{(p)}$, we have the set of important indices $I_{[p]}:=\{i \in [d]: \lambda_i \geq  \frac{\lambda_p}{2}\}=\{1,2,\dots, r\}$ with $r=|I_{[p]}|$, and the gap $\delta_p=\lambda_p -\lambda_{p+1}$. The parameters $x,y,w$ and the contour $\Gamma$ become
\begin{itemize}
    \item $x_p:= \max_{1 \leq i,j\leq r} |u_i^\top E u_j|$.
    \item  $y_p:= \max_{\substack{1\leq i\neq j \leq r\\1\leq k \leq  p}}|u_i^\top E \sum_{l > r} \frac{u_l u_l^\top}{\lambda_k -\lambda_l} E u_j|.$
    \item $w_p:= \max_{1 \leq i \leq r} \|E u_i\|.$
    \item $\Gamma_p$ is a rectangle, whose left vertical edge bisects the interval $(\lambda_{p+1}, \lambda_{p})$, whose right vertical edge passes through the point $\|M\|+1.1\|E\|$, and whose horizontal edges lie at distance $4\delta_p$ from the real axis.
\end{itemize}
Therefore, the contour representation for the perturbation of eigenspaces is 
\[
\tilde{\Pi}_p - \Pi_p:= \sum_{s =1}^{\infty} F_s,\,\,\text{where}\,\,F_s:=\frac{1}{2\pi \textbf{i}} \int_{\Gamma_p}(zI -M)^{-1}[E (zI-M)^{-1}]^s dz.
\]

For the upper bound, arguing similarly to \eqref{F3-HR} and \eqref{F4-HR}, we have 
\[
\|\tilde{\Pi}_p - \Pi_p\| \leq \sum_{s=1}^{\infty}\|F_s\| \leq \frac{7\|E\|}{\lambda_p} + \frac{5\sqrt{r} x_p}{\delta_p} + \frac{5\sqrt{r}y_p}{\delta_p}.
\]
The detailed analysis of the RHS similarly proceeds as in Section \ref{subsec: proofupperM1/2}, yielding Theorem \ref{theo: Piupper}. 

For the lower bound, arguing similarly to \eqref{up lower ine1} and \eqref{Firstlowerbound0}, we obtain 
\begin{equation*}
    \begin{split}
     \|\tilde{\Pi}_p - \Pi_p\| & \geq \|F_1\| - \|F_2\| - \sum_{s\geq 3}\|F_s\| \\   
     & = \|M_{1,p}+M_{2,p}\| - \left[ \frac{21}{C} \cdot \left( \frac{\|E\|}{\lambda_p} + \frac{\sqrt{r} x_p}{\delta_{p}} \right) + \frac{3\sqrt{r} y_p}{\delta_{p}} \right],\,\,\text{where}
    \end{split}
\end{equation*}
\begin{equation*}
    \begin{split}
       M_{1,p}
:={}&
\sum_{1 \leq k \leq p} \sum_{r\geq j >p}
\left(
u_k \frac{u_k^\top E u_j}{\lambda_k-\lambda_j} u_j^\top
+
u_j \frac{u_j^\top E u_k}{\lambda_k-\lambda_j} u_k^\top
\right),
\\
M_{2,p}
:={}& \sum_{1 \leq k \leq p}
u_ku_k^\top E
\left(
\sum_{l >r}
\frac{u_lu_l^\top}{\lambda_k-\lambda_l}
\right)
+
\left(
\sum_{l > r}
\frac{u_lu_l^\top}{\lambda_k-\lambda_l}
\right)
E u_k u_k^\top.
    \end{split}
\end{equation*}
The detailed analysis of $M_{1,p}+ M_{2,p}, \|E\|, x_p, y_p$ proceeds as in Section \ref{subsec: prooflowerM1/2Line1}, yielding Theorem \ref{theo: PipM1/2lower}.

\section{Completion of the proof of the upper bound - Theorem \ref{theo: upUpperM1/2 subGau} } \label{subsec: proofupperM1/2}
In the previous section, we have shown that if 
\[
\textstyle   \max \left\lbrace \frac{ \|E\|}{\lambda_p/2}, \frac{r x}{\delta_{(p)}}, \frac{\sqrt{r} w}{\sqrt{\lambda_p \delta_{(p)}/2}} \right\rbrace < \frac{1}{C},\] 
 then 
 \[
  \|\tilde{u}_p \tilde{u}_p^\top - u_p u_p^\top   \|\le  \frac{7\| E \| }{ \lambda_p} + \frac{ 5 \sqrt{r} x } {\delta_{(p)}} +  \frac{ 5\sqrt{ r} y }{ \delta_{(p)} }.
 \]
Thus, to complete the proof of Theorem~\ref{theo: upUpperM1/2 subGau}, we estimate $\|E\|$ in Subsection~\ref{subsec: E est} and the skewness parameters $w$, $x$, and $y$ from Definition~\ref{def: xyw} in Subsections~\ref{subsec: w est} and \ref{subsec: xy est}. We then combine these estimates in Subsection~\ref{subsec: assembly upper} to establish Theorem~\ref{theo: upUpperM1/2 subGau}.
\subsection{Estimation of \texorpdfstring{$\|E\|$}{E}} \label{subsec: E est} We use the following well-known estimate for  $\|E\|$ when $Y$ is a sub-Gaussian random vector \cite{koltchinskii2017concentration, zhivotovskiy2024dimension}.

\begin{theorem} \label{theo: ESubgau} Under the above setting, for any $t >0$, with probability at least $1 -e^{-t}$, provided that $n \geq 4 r_{\mathrm{eff}}+t$, 
\begin{equation} 
    \textstyle \|E\| \leq 20 K \lambda_1  \sqrt{\frac{4r_{\mathrm{eff}}+t}{n}}.
\end{equation}
\end{theorem}
Thus, with probability at least $1-e^{-4r_{\mathrm{eff}}t^2},$
\begin{equation}\label{bound E sub-Gau}
   \textstyle \|E\| \leq 40 K t \cdot \lambda_1 \sqrt{\frac{r_{\mathrm{eff}}}{n}}. 
\end{equation}

\subsection{Estimation of \texorpdfstring{$w$}{w}} \label{subsec: w est} We prove the following lemma. 
\begin{lemma} \label{lem: w bound} Under the above setting, there is a universal constant $C > 0$ so that for any $t > 0$, with probability at least $1 -t^{-4}$, 
$$\textstyle \|E u_i\| \leq  \frac{\lambda_i t}{\sqrt{n}} \cdot \sqrt{\sum_{k=1}^d \frac{\lambda_k}{\lambda_i} \cdot \Var S^{(k,i)}   } \cdot \sqrt{2+ \frac{CK^2}{n^{1/2}}}.$$
\end{lemma}
\begin{proof}[Proof of Lemma \ref{lem: w bound}]
 We use the moment method to prove this lemma. Indeed, we compute the first and second moments of $\frac{\|Eu_i\|^2}{\lambda_i}$. Recall the notations that for each $1\leq k,i \leq d$,  
 $$S^{(k,i)}:= (u_k^\top Y)(Y^\top u_i),$$
 and $\{S_{j}^{(k,i)}:= (u_k^\top Y_j)(Y_j^\top u_i) \}_{j=1}^n$ are $n$ iid samples of $S^{(k,i)}$. Since $\E S^{(k,i)}= 1 $ if $k =i$ and $\E S^{(k,i)}= 0 $ if $k \neq i$, we can expand and rewrite 
 \begin{equation} \label{iden: NormEui}
  \frac{\|E u_i\|^2}{\lambda_i}=
\sum_{k=1}^d \frac{\lambda_k}{n^2} \big[ \sum_{j=1}^n (S^{(k,i)}_j - \E S^{(k,i)}_j )  \big]^2. 
 \end{equation}
Note that $S^{(k,i)}_j - \E S^{(k,i)}_j$ has mean zero and variance $\Var S^{(k,i)}$. Moreover, for any $j \neq j'$, $S_{j}^{(k,i)}$ and $S_{j'}^{(l,i)}$ for any $1 \leq k,l \leq d$ are independent. Thus, directly computing the first moment, we obtain
\begin{equation} \label{iden: EuiNormEx}
    \E \left(\frac{\|Eu_i\|^2}{\lambda_i} \right)= \sum_{k=1}^d \frac{\lambda_k}{n^2} \sum_{j=1}^n \E [(S^{(k,i)}_j - \E S^{(k,i)}_j )^2]= \frac{1}{n} \sum_{k=1}^d \lambda_k \cdot \Var S^{(k,i)}.
\end{equation}

Denote $s_{jk}= S^{(k,i)}_j - \E S^{(k,i)}_j$. Then, the second moment can be rewritten and expanded as 
\begin{equation} \label{eq: lem3.2-1}
    \begin{split}
    \textstyle    \E \left[ \left(\frac{\|Eu_i\|^2}{\lambda_i} \right)^2 \right] & = \textstyle \E \left[\sum_{k,k'=1}^d\frac{\lambda_k \lambda_{k'}}{n^4} \left(\sum_{j=1}^n s_{jk} \right)^2 \left(\sum_{l=1}^n s_{lk'} \right)^2 \right] \\
        & = \textstyle \E \left[\sum_{k,k'=1}^d\frac{\lambda_k \lambda_{k'}}{n^4} \left(\sum_{j,j',l,l' \leq n} s_{jk} s_{j'k} s_{lk'} s_{l'k'}  \right) \right].
    \end{split}
\end{equation}
Since $\E s_{jk} =0$, the nontrivial terms in $\sum_{j,j',l,l' \leq n} s_{jk} s_{j'k} s_{lk'} s_{l'k'} $ belong to one of the four following cases:
$$(1)\, j=j' \neq l=l';\,(2)\, j=l \neq j'=l';\,(3)\, j=l'\neq j'=l;\,\text{and}\,(4)\, j=j'=l=l'.$$
Without loss of generality, we handle the sub-sums corresponding to the first and second cases. Other cases can be argued similarly. The first sub-sum is 
\begin{equation} \label{eq: lem3.2-2}
    \begin{split}
     \sum_{k,k'\leq d} \frac{\lambda_k \lambda_{k'}}{n^4} \cdot \sum_{j \neq l} \E (s_{jk}^2) \E(s_{lk'}^2) & \leq \sum_{k,k'\leq d} \frac{\lambda_k \lambda_{k'}}{n^4} \cdot \sum_{j,l \leq n} \E (s_{jk}^2) \E(s_{lk'}^2) \\
        &  =  \sum_{k,k'\leq d} \frac{\lambda_k \lambda_{k'}}{n^4} \big(\sum_{j\leq n} \E(s_{jk}^2) \big) \big(\sum_{l\leq n} \E(s_{lk'}^2) \big) \\
        &  = \bigg[ \sum_{k \leq d} \frac{\lambda_k}{n} \cdot \Var S^{(k,i)} \bigg]^2,
    \end{split}
\end{equation}
which is exactly the square of the first moment. 

By a similar expression, the second sub-sum is 
\begin{equation} \label{eq: lem3.2-3}
    \begin{split}
   \sum_{k,k' \leq d} \frac{\lambda_k \lambda_{k'}}{n^4}  \sum_{j \neq j' \leq n} \E (s_{jk} s_{jk'}) \E(s_{j'k} s_{j'k'}) &  =  \sum_{k,k' \leq d} \frac{\lambda_k \lambda_{k'}}{n^4}  \bigg[\sum_{j \leq n} \E (s_{jk} s_{jk'})  \bigg]^2 -  \sum_{k,k' \leq d} \frac{\lambda_k \lambda_{k'}}{n^4} \sum_{j \leq d} \E (s_{jk} s_{jk'})^2   \\
   & = \sum_{k,k' \leq d} \frac{\lambda_k \lambda_{k'}}{n^2} \cdot  [\E (s_{k} s_{k'})]  ^2 -  \sum_{k,k' \leq d} \frac{\lambda_k \lambda_{k'}}{n^3} \cdot [\E (s_{k} s_{k'})]^2 \\
   & \leq \sum_{k,k' \leq d} \frac{\lambda_k \lambda_{k'}}{n^2} \cdot  [\E (s_{k} s_{k'})]  ^2.
    \end{split}
\end{equation}
where $s_k:= (u_k^\top Y)(u_i^\top Y) - \E (u_k^\top Y)(u_i^\top Y) $. Moreover, by the Cauchy-Schwarz inequality, $[\E (s_{k} s_{k'})]^2  \leq \E(s_{k}^2) \cdot \E (s_{k'}^2),$ we obtain that the second sub-sum is at most 
$$\textstyle\sum_{k,k' \leq d} \frac{\lambda_k \lambda_{k'}}{n^2} \cdot \E(s_{k}^2) \cdot \E (s_{k'}^2) = \left( \sum_{k \leq d}\frac{\lambda_k}{n} \cdot \E (s_k^2) \right)^2 = \bigg[ \sum_{k \leq d} \frac{\lambda_k}{n} \cdot \Var S^{(k,i)}  \bigg]^2.$$
Arguing similarly, we also obtain that the third and the fourth sub-sums are also at most $\bigg[ \sum_{k \leq d} \frac{\lambda_k}{n} \cdot \Var S^{(k,i)}  \bigg]^2$ and $\frac{CK^4}{n} \cdot \bigg[ \sum_{k \leq d} \frac{\lambda_k}{n} \cdot \Var S^{(k,i)}  \bigg]^2$ for some universal constant $C$, respectively. Thus, the second moment of $\frac{\|Eu_i\|^2}{\lambda_i}$ is at most 
$$\textstyle \big(3+\frac{CK^4}{n} \big) \cdot \bigg[ \sum_{k \leq d} \frac{\lambda_k}{n} \cdot \Var S^{(k,i)}  \bigg]^2.$$

These estimates on the first and second moments imply that for any $t > 0$ with probability at least $1 - t^{-2},$
$$\textstyle\frac{\|Eu_i\|^2}{\lambda_i} \leq (3+ \frac{C' K^2}{\sqrt{n}}) \cdot \frac{t}{n}  \sum_{k=1}^d \lambda_k \cdot \Var S^{(k,i)}, \text{for some universal constant $C'$}.$$
This is equivalent to 
$$\textstyle \|E u_i\| \leq \lambda_i \sqrt{\frac{t}{n}} \cdot \sqrt{\sum_{k=1}^d \frac{\lambda_k}{\lambda_i} \cdot \Var S^{(k,i)}   } \cdot \sqrt{2+ \frac{C'K^2}{n^{1/2}}}.$$
Thus, with probability at least $1-t^{-4}$, 
$$\textstyle \|E u_i\| \leq \frac{ \lambda_i t}{\sqrt{n}} \cdot \sqrt{\sum_{k=1}^d \frac{\lambda_k}{\lambda_i} \cdot \Var S^{(k,i)}  } \cdot \sqrt{2+ \frac{C'K^2}{n^{1/2}}}.$$
We complete the proof.
\end{proof}
Back to our estimate of $w$, by Lemma \ref{lem: w bound} and the union bound, with probability at least $1 - \frac{r}{t^4}$, for some universal constant $C >0$
\begin{equation} \label{est: w spiked}
\begin{split}
    \textstyle    w &:= \max_{i \in I_p} \|E u_i\|\,\,(\text{Definition \ref{def: xyw}}) \\
    & \leq 2 t \frac{\lambda_p}{\sqrt{n}} \cdot \sqrt{\sum_{k=1}^d \frac{\lambda_k}{\lambda_p} \cdot \Var S^{(k,i)}   } \leq 4 K t \sqrt{\lambda_1 \lambda_p} \cdot \sqrt{\frac{r_{\mathrm{eff}}}{n}}.
\end{split}
\end{equation}
The last inequality follows from the fact that 
\(
\Var S^{(k,i)} \leq 4K^2 \,\,\text{for all $1 \leq k,i \leq d$ (Lemma \ref{lem: S value}}). 
\)

\subsection{Estimation of \texorpdfstring{$x$}{x} and \texorpdfstring{$y$}{y}} \label{subsec: xy est} We prove the following lemma to estimate $x$ and $y$.
\begin{lemma}\label{lemma: xyModeil1}  Under the above setting, we have:
\begin{itemize}
    \item For any $t_1 >0$ and $1 \leq i,j \leq d$, with probability at least $1 - t_1^{-2}$,
    $$   \textstyle   \norm{ u_i^\top E u_j} \leq 2 K \frac{\sqrt{\lambda_i \lambda_j} t_1}{\sqrt{n}}.$$
     \item There is a universal constant $C$ so that for any $t_2  > 0$, $ i<j \in I_p$, with probability at least $1 -t_2^{-2}$, 
$$   \textstyle  \norm{ u_i^\top E \left(\sum_{l \notin I_p} \frac{u_l u_l^\top}{\lambda_p - \lambda_l} \right) E u_j} \leq 3t_2\left[ \frac{ 4K^2\sqrt{\lambda_i \lambda_j}}{n}+  \frac{CK^2 \sqrt{\lambda_i \lambda_j}}{n^{3/2}} \left( \sum_{l \notin I_p } \frac{ \lambda_l}{ \lambda_p -\lambda_{l}} \right)+ \frac{K\sqrt{\lambda_i \lambda_j} }{n} \sqrt{\sum_{l \notin I_p } \frac{\lambda_l^2}{(\lambda_p -\lambda_l)^2}} \right].$$

\end{itemize}    
\end{lemma}
We use the moment method to prove this lemma. Technically, we will bound the first and second moments of $\norm{ u_i^\top E u_j} $ and $    \textstyle   \norm{ u_i^\top E \left(\sum_{l\notin I_p} \frac{u_l u_l^\top}{\lambda_p - \lambda_l} \right) E u_j}$. The moment computations are deferred to Appendix \ref{section: moments}.

Back to our estimate of $x$ and $y$, by Lemma \ref{lemma: xyModeil1}, and the union bound, with probability at least $1- \frac{r^2}{t_1^2}$, 
\begin{equation} \label{est: x spiked1}
\begin{split}
    \textstyle   x & := \max_{i,j \in I_p} \norm{u_i^\top E u_j}\,\,(\text{Definition \ref{def: xyw}}) \\
    &\leq 2K t_1 \cdot \frac{ \max_{i \in I_p} \lambda_i}{\sqrt{n}} \\
    & \leq 3 K \frac{\lambda_p t_1}{\sqrt{n}} \,\,(\text{by \eqref{lambdaivs lambdap}}). 
\end{split}
   \end{equation}
Also by the union bound, with probability at least $1 - \frac{r^2}{t_2^2},$
\begin{equation} \label{est: yspikedpopu}
    \begin{split}
  y & := \textstyle \max_{\substack{ i\neq j \in I_p}} \norm{ u_i^\top E \left(\sum_{l \notin I_p  } \frac{u_l u_l^\top}{\lambda_p - \lambda_l} \right) E u_j} \,\,(\text{Definition \ref{def: xyw}})\\
  & \leq    \textstyle    3t_2 \max_{i\neq j \in I_p} \sqrt{\lambda_i \lambda_j} \left[ \frac{ 4K^2}{n}+  \frac{CK^2 }{n^{3/2}} \left( \sum_{l \notin I_p } \frac{ \lambda_l}{ \lambda_p -\lambda_{l}} \right)+ \frac{K} {n} \sqrt{\sum_{l \notin I_p } \frac{\lambda_l^2}{(\lambda_p -\lambda_l)^2}} \right]  \\
  & \leq    \textstyle   9K \frac{\lambda_p t_2}{n} \left[ 2K+  \frac{CK^2 }{n^{1/2}} \left(  \frac{ \sum_{l \notin I_p } \lambda_l}{ \lambda_p} \right)+ \frac{\sqrt{\sum_{l \notin I_p } \lambda_l^2 }} {\lambda_p}  \right]  \\
  & (\text{since $\lambda_i \leq \frac{3}{2} \lambda_p$ for all $i \in I_p$ and $\lambda_p -\lambda_l \geq \frac{\lambda_p}{2}$ for all $l \notin I_p$}).
    \end{split}
\end{equation}

The estimates for $x$ and $y$ complete the list of quantities appearing in the upper bound of $\|\tilde{u}_p \tilde{u}_p^\top - u_p u_p^\top\|$.
\subsection{Putting things together} \label{subsec: assembly upper} We now assemble the estimates established in Sections \ref{subsec: E est} - \ref{subsec: xy est} to complete the proof of Theorem \ref{theo: upUpperM1/2 subGau}.
Combining \eqref{F4-HR}, Theorem \ref{theo: ESubgau}, \eqref{est: w spiked}, \eqref{est: x spiked1}, and \eqref{est: yspikedpopu}, we obtain that there is a universal constant $C > 0$ so that for any $t > 0$, 
$$\text{if}\,\,\,\textstyle   \textstyle K \kappa_p t \cdot \max\!\left\{
\sqrt p\,\sqrt{\frac{r_{\mathrm{eff}}}{n}},
\;
r\frac{\lambda_p}{\sqrt n\,\delta_{(p)}},
\;
\sqrt r\,\sqrt{\frac{\lambda_p r_{\mathrm{eff}}}{n\delta_{(p)}}}
\right\}
<
\frac{1}{C},\,\text{then}  $$
\begin{equation} \label{spikedbound before final}
\begin{split}
      \textstyle \Norm{\tilde{u}_p \tilde{u}_p^\top - u_p u_p^\top} & \textstyle \leq 280K \kappa_p t\cdot \sqrt{\frac{r_{\mathrm{eff}}}{n}} + 15Kt\frac{ \sqrt{r} \lambda_p}{\delta_{(p)} \sqrt{n}}  +45 Kt \frac{ \sqrt{r}\lambda_p }{n \delta_{(p)}} \left[ 2K+  \frac{CK^2 }{n^{1/2}} \left(  \frac{ \sum_{l \notin I_p } \lambda_l}{ \lambda_p} \right)+ \frac{\sqrt{\sum_{l \notin I_p } \lambda_l^2 }} {\lambda_p}  \right]  \\
      & = \textstyle 280K \kappa_p t\cdot \sqrt{\frac{r_{\mathrm{eff}}}{n}} + 15Kt\frac{ \sqrt{r} \lambda_p}{\delta_{(p)} \sqrt{n}}  +45 Kt \frac{ \sqrt{r}\lambda_p }{ \sqrt{n} \delta_{(p)}} \left[ \frac{2K}{\sqrt{n}}+  \frac{CK^2 }{n} \left(  \frac{ \sum_{l \notin I_p } \lambda_l}{ \lambda_p} \right)+ \frac{\sqrt{\sum_{l \notin I_p } \lambda_l^2 }} {\lambda_p \sqrt{n}}  \right]
\end{split}
 \end{equation}
 with probability at least $1 - e^{-4r_{\mathrm{eff}}t^2} - \frac{2r^2}{t^2}$.
 
 First, the condition can be rewritten exactly into the assumption of Theorem \ref{theo: upUpperM1/2 subGau}, which is   
$$\textstyle n \geq  C'K^2 \kappa_p^2 t^2 \cdot \max \big\{ \big(\frac{r \lambda_p}{\delta_{(p)}} \big)^2,  r_{\mathrm{eff}} \cdot \frac{r\lambda_p}{\delta_{(p)}} \big\},\,\text{ for some $C' >0$.}$$

Next, we simplify the RHS of \eqref{spikedbound before final}. Notice that 
$$\sum_{l\notin I_p} \lambda_l \leq \sum_{i=1}^d \lambda_i = \lambda_1 r_{\mathrm{eff}}, \sum_{l \notin I_p} \lambda_l^2 \leq \lambda_1 (\sum_{i=1}^d \lambda_i) = \lambda_1^2 r_{\mathrm{eff}}.$$
Thus, we further have 
\begin{equation} \label{ineq: extra term y are neg}
    \begin{split}
        & \frac{\sqrt{\sum_{l \notin I_p} \lambda_l^2 }}{\lambda_p \sqrt{n}} \leq  \frac{ \lambda_1 \sqrt{r_{\mathrm{eff}}}}{\lambda_p \sqrt{n}} = \kappa_p\sqrt{\frac{r_{\mathrm{eff}}}{n}} \leq \frac{1}{CKt}   , \,\text{and} \\
        & \frac{CK^2 (\sum_{l \notin I_p} \lambda_l)}{\lambda_p n} \leq \frac{ CK^2\cdot \lambda_1 r_{\mathrm{eff}}}{\lambda_p n} = \kappa_p \frac{CK^2 \cdot r_{\mathrm{eff}}}{n}   \leq \frac{1}{ C\kappa_p t^2}.
    \end{split}
\end{equation}
Therefore, for any $C>24$ and $n > 4K^2$, the third term on the RHS of \eqref{spikedbound before final} is less than $3 \times 45Kt \frac{ \sqrt{r}\lambda_p }{ \sqrt{n} \delta_{(p)}}$, simplifying the RHS of \eqref{spikedbound before final} into 
$$\textstyle 280Kt \left( \kappa_p \cdot \sqrt{\frac{r_{\mathrm{eff}}}{n}} +  \sqrt{r} \frac{ \lambda_p}{\delta_{(p)} \sqrt{n}}\right).$$
This proves Theorem \ref{theo: upUpperM1/2 subGau}.

\section{Completion of the proof of the lower bound - Theorem \ref{theo: uplowerboundM1/2 subGau}} \label{subsec: prooflowerM1/2Line1} We first recall some important estimates established in Section \ref{subsec: proofupperM1/2}.
\begin{summary} \label{summary: TM1M2M3} For any $C > 24$ and any $t > 0$, the following bounds hold.
\begin{itemize}
    
    \item With probability at least $1 - \frac{r^2}{t^2},$
    \(
     \frac{\sqrt{r} x}{\delta_{(p)}} \leq 3 K \sqrt{r} \frac{\lambda_p t}{\sqrt{n} \delta_{(p)}};
    \) see \eqref{est: x spiked1}.
    \item With probability at least $1 - e^{-4 r_{\mathrm{eff}}t^2},$
    \(
     \frac{\|E\|}{\lambda_p} \leq 40 K\kappa_p t \cdot \sqrt{\frac{r_{\mathrm{eff}}}{n}};
    \) see \eqref{bound E sub-Gau}.
    \item With probability at least $1 - \frac{r^2}{t^2},$
    \(
     \frac{\sqrt{r} y}{\delta_{(p)}} \leq 9 Kt \frac{ \sqrt{r}\lambda_p }{ \sqrt{n} \delta_{(p)}} \left[ \frac{2K}{\sqrt{n}}+  \frac{1}{C\kappa_p t^2}+ \frac{1}{CKt} \right];
    \) see \eqref{est: yspikedpopu} and \eqref{ineq: extra term y are neg}.
 
\end{itemize}

\end{summary}
The constants $3,9,40$ are ad hoc. One can optimize them by revisiting the computations in Section \ref{subsec: proofupperM1/2} and Section \ref{section: moments}.

Back to our proof of Theorem \ref{theo: uplowerboundM1/2 subGau}, by Subsection \ref{subsec: uplower setting} (Inequality \eqref{Firstlowerbound0}), we have shown that if 
\(
\textstyle   \max \left\lbrace \frac{ \|E\|}{\lambda_p/2}, \frac{r x}{\delta_{(p)}}, \frac{\sqrt{r} w}{\sqrt{\lambda_p \delta_{(p)}/2}} \right\rbrace < \frac{1}{C}, 
\)
then 
\[
 \|\tilde{u}_p \tilde{u}_p^\top - u_p u_p^\top   \| \geq  \|M_1+M_2\| - \left[ \frac{21}{C} \cdot \left( \frac{\|E\|}{\lambda_p} + \frac{\sqrt{r} x}{\delta_{(p)}} \right) + \frac{3\sqrt{r} y}{\delta_{(p)}} \right].
\]
Substituting the estimates from Summary \ref{summary: TM1M2M3} one by one yields that if
\[ \textstyle
n\geq (40CK\kappa_p t)^2 \cdot \max\bigg\{r_{\mathrm{eff}} \cdot \frac{\lambda_p}{\delta_{(p)}}, \left(\frac{\lambda_p}{\delta_{(p)}} \right)^2 \bigg\},
\]
then with probability at least $1 - e^{-4r_{\mathrm{eff}}t^2}-\frac{2r^2}{t^2}$,
\begin{equation*} \label{Firstlowerbound1}
\begin{split}
    \|\tilde{u}_p \tilde{u}_p^\top - u_p u_p^\top   \|
    & \textstyle \geq \|M_1+M_2\| - \frac{840Kt}{C} \left(\kappa_p\sqrt{\frac{r_{\mathrm{eff}}}{n}}+\frac{\sqrt{r} \lambda_p}{\sqrt{n}\delta_{(p)}} \right) - 27 Kt \frac{ \sqrt{r}\lambda_p }{ \sqrt{n} \delta_{(p)}} \left[ \frac{2K}{\sqrt{n}}+  \frac{1}{C\kappa_p t^2}+ \frac{1}{CKt} \right] \\
    & \textstyle
  \geq  \|M_1+M_2\| -  \frac{900Kt}{C} \left[   \frac{\sqrt{r}\lambda_p}{\sqrt{n} \delta_{(p)}} + \kappa_p \cdot \sqrt{\frac{r_{\mathrm{eff}}}{n}} \right].
\end{split}
  \end{equation*}
The last inequality follows from the setting that $C > 24$. 

Thus, to prove Theorem \ref{theo: uplowerboundM1/2 subGau}, the only remaining task is to obtain a lower bound for $\|M_1+M_2\|$. Indeed, it is sufficient to show that with probability at least $\frac{1}{4}-\frac{1}{2T^2} -\frac{2r^2}{t^2} - e^{-4r_{\mathrm{eff}}t^2},$
\[
\|M_1+M_2\| \geq \frac{1}{2T} \left[ \sqrt{\Var S^{(p,\gamma_p)}} \cdot \frac{\lambda_p}{\sqrt{n} \delta_{(p)}} + \sqrt{\frac{s_p}{\kappa_p}} \cdot \sqrt{\frac{r_{\mathrm{eff}}}{n}} \right] - \frac{4 tK/\kappa_p}{\sqrt{n}}.
\]
Here, for each $1\leq i \leq d$,  
 $S^{(p,i)}:= (u_p^\top Y)(Y^\top u_i).$
 The weighted average of the variances of $\{\Var S^{(p,i)}\}_{i=1}^d$ is 
 $s_p:=  \frac{\sum_{i=1}^d \lambda_i \cdot \Var S^{(p,i)}}{\sum_{i=1}^d \lambda_i}.$

Since
\(
M_1+M_2 = \sum_{j \neq p}
\left(
u_p \frac{u_p^\top E u_j}{\lambda_p-\lambda_j} u_j^\top
+
u_j \frac{u_j^\top E u_p}{\lambda_p-\lambda_j} u_p^\top
\right),
\)
see \eqref{iden: F1M1M2},
and $u_1, u_2, \dots, u_d$ are orthonormal vectors, we have
\begin{equation} \label{ide: normM1+M2}
    \|M_1+M_2\|= \sqrt{\sum_{j \neq p} \frac{|u_p^\top E u_j|^2}{(\lambda_p -\lambda_j)^2} }.
\end{equation}

The next idea is to use $\delta_{(p)}$ to split $\sum_{j \neq p} \frac{|u_p^\top E u_j|^2}{(\lambda_p -\lambda_j)^2}.$ Without loss of generality, assume that $\gamma_p=p+1$, or equivalently $\delta_{(p)}=\lambda_p -\lambda_{p+1}$. We write 
\[
\sum_{j \neq p} \frac{|u_p^\top E u_j|^2}{(\lambda_p -\lambda_j)^2}  = \frac{|u_p^\top E u_{p+1}|^2}{\delta_{(p)}^2}+ \sum_{j \notin \{p,p+1\}} \frac{|u_p^\top E u_j|^2}{(\lambda_p -\lambda_j)^2} \geq \frac{|u_p^\top E u_{p+1}|^2}{\delta_{(p)}^2}+ \frac{\sum_{j \notin \{p,p+1\}} |u_p^\top E u_j|^2 }{\lambda_1^2}.
\]
The RHS can be further rewritten as 
\begin{equation} \label{ineq: lowbound1}
    \begin{split}
       \frac{|u_p^\top E u_{p+1}|^2}{\delta_{(p)}^2}+ \frac{\|Eu_p\|^2 - |u_p^\top E u_{p+1}|^2- |u_p^\top E u_{p}|^2 }{\lambda_1^2} &\geq \frac{|u_p^\top E u_{p+1}|^2}{\delta_{(p)}^2}+ \frac{\|Eu_p\|^2 }{\lambda_1^2} - 12 \left(\frac{Kt}{\sqrt{n} \kappa_p} \right)^2.
    \end{split}
\end{equation}
The last inequality is obtained by applying Lemma \ref{lemma: xyModeil1} to $|u_p^\top E u_{p+1}|$ and $|u_p^\top E u_p|$. 

\noindent\textbf{Bounding $\frac{|u_p^\top E u_{p+1}|^2}{\delta_{(p)}^2}$ from below.} We  can rewrite  $ u_p^\top E u_{p+1} $ as $ u_p^\top \left[\frac{1}{n}(\sum_{i=1}^n X_iX_i^\top -M) \right] u_{p+1}$, which equals
\begin{equation} \label{epEuj Iden}
    \begin{split}
  &  \frac{1}{n} \bigg[\sum_{i=1}^n u_p^\top X_i X_i^\top u_{p+1}  \bigg] \,\,\,(\text{since}\,\, u_p^\top M u_{p+1} =0) \\
  & = \frac{1}{n} \sum_{i=1}^n \sqrt{\lambda_p \lambda_{p+1}} \left[(u_p^\top Y_i)(Y_i^\top u_{p+1})  \right] \,\,(\text{since}\,\,X_i=M^{1/2} Y_i) \\
  & = \frac{ \sqrt{\lambda_p \lambda_{p+1}}}{n} \cdot \sum_{i=1}^n S_i^{(p,p+1)}, \,\text{where}\, S_i^{(p,p+1)}:= (u_p^\top Y_i)(Y_i^\top u_{p+1}).
    \end{split}
\end{equation}
Notice that the random variables $S_i^{(p,p+1)}$ for $1 \leq i \leq n$ are iid samples of $S^{(p,p+1)}$. Thus, by the Central Limit Theorem, 
$$\textstyle \bigg|\P\left(a <\frac{\sum_{i=1}^n S_i^{(p,p+1)}}{\sqrt{n \Var S^{(p,p+1)} }} < b  \right) - \P \left(a < \mathcal{N}(0,1) < b \right)   \bigg| = O\left(\frac{1}{\sqrt{n}} \right).$$
Thus, with probability at least $1/2-\frac{1}{T}- O(\frac{1}{\sqrt{n}})$, 
$\textstyle \sum_{i=1}^n S_i^{(p,p+1)} \geq \frac{\sqrt{n \Var S^{(p,p+1)}}}{T },$
and hence
$$ \textstyle (u_p^\top E u_{p+1})^2 \geq \frac{\lambda_p \lambda_{p+1}}{n T^2} \cdot \Var S^{(p,p+1)} \geq \frac{\lambda_p^2}{2n T^2} \cdot \Var S^{(p,p+1)}.$$
The last inequality follows from our setting that $\delta_{(p)}=\lambda_p -\lambda_{p+1}$ and thus, $p+1 \in I_p$.

Therefore, with probability at least $1/2-\frac{1}{\sqrt{T}}- O(\frac{1}{\sqrt{n}})$,
\begin{equation} \label{ineq: lowbound term 1}
  \frac{|u_p^\top E u_{p+1}|^2}{\delta_{(p)}^2} \geq \frac{\Var S^{(p,p+1)}}{2T} \cdot \left(\frac{\lambda_p}{\sqrt{n} \delta_{(p)}} \right)^2.   
\end{equation}

\noindent\textbf{Bounding $\frac{\|E u_p\|^2}{\lambda_1^2}$ from below.}
 By \eqref{iden: NormEui} and \eqref{iden: EuiNormEx}, we have 
\(
 \frac{\|E u_p\|^2}{\lambda_p}= \sum_{i=1}^d \frac{\lambda_i}{n^2} \big[ \sum_{j=1}^n (S_{j}^{(p,i)} -\E S_j^{(p,i)}) \big]^2
\), and hence
\begin{equation} \label{ide: EupNormEx}
    \begin{split}
        \E \frac{\|E u_p\|^2}{\lambda_p} & =    \sum_{i=1}^d \frac{\lambda_i}{n} \cdot \Var S^{(p,i)}. 
    \end{split}
\end{equation}

On the other hand, using \eqref{eq: lem3.2-1}, \eqref{eq: lem3.2-2}, \eqref{eq: lem3.2-3}, and the line after  \eqref{eq: lem3.2-3}, we have 
\begin{equation}
  \textstyle  \E \left(\frac{\|E u_p\|^2}{\lambda_p}  \right)^2 \leq (3+ \frac{CK^4}{n})\cdot \big[\sum_{i=1}^d \frac{\lambda_i}{n} \cdot \Var S^{(p,i)} \big]^2 \leq 4 \big[\sum_{i=1}^d \frac{\lambda_i}{n} \cdot \Var S^{(p,i)} \big]^2.
\end{equation}
Therefore, by the Raymond Paley–Zygmund inequality, for any $T>1$, we have
\begin{equation*}
    \begin{split}
  \P \left(\frac{\|E u_p\|^2}{\lambda_p} \geq \frac{1}{T} \cdot  \sum_{i=1}^d \frac{\lambda_i}{n} \cdot \Var S^{(p,i)} \right) & =     \P \left(\frac{\|E u_p\|^2}{\lambda_p} \geq \frac{1}{T} \cdot    \E \frac{\|E u_p\|^2}{\lambda_p} \right)    \\
 & \geq \big(1-\frac{1}{T}\big)^2 \cdot \frac{ \big(\E \frac{\|E u_p\|^2}{\lambda_p} \big)^2  }{\E \left(\frac{\|E u_p\|^2}{\lambda_p}  \right)^2 }.
 \end{split}
\end{equation*}
Using the lower bound of the first moment and the upper bound of the second moment above, we further obtain
\[
\textstyle \P \left(\frac{\|E u_p\|^2}{\lambda_p} \geq \frac{1}{T} \cdot  \sum_{i=1}^d \frac{\lambda_i}{n} \cdot \Var S^{(p,i)} \right)\geq  \big(1-\frac{1}{T}\big)^2 \cdot \frac{\big[\sum_{i=1}^d \frac{\lambda_i}{n} \cdot \Var S^{(p,i)} \big]^2}{4\big[\sum_{i=1}^d \frac{\lambda_i}{n} \cdot \Var S^{(p,i)} \big]^2} = \big(\frac{1}{2}-\frac{1}{2T}\big)^2.
\]
Consequently, with probability at least $\frac{1}{4}-\frac{1}{2T}$, we have 
\begin{equation} \label{sp appear}
    \begin{split}
 \frac{\|Eu_p\|^2}{\lambda_1^2} & \geq \frac{1}{T} \cdot \frac{\lambda_p}{\lambda_1^2} \cdot   \sum_{i=1}^d \frac{\lambda_i}{n} \cdot \Var S^{(p,i)}  = \frac{s_p}{T} \cdot \frac{\lambda_p \cdot \sum_{i=1}^d \lambda_i}{\lambda_1^2 n} = \frac{s_p}{T \kappa_p} \cdot \frac{r_{\mathrm{eff}}}{n}.
    \end{split}
\end{equation}

Thus, combining \eqref{ide: normM1+M2}, \eqref{ineq: lowbound1}, \eqref{ineq: lowbound term 1}, and  \eqref{sp appear} (replacing $T$ in \eqref{sp appear} with $T^2$), we finally obtain that with probability at least $\frac{1}{4} - \frac{1}{2T^2} -\frac{1}{t^2}$, 
\[
\sum_{j \neq p} \frac{|u_p^\top E u_j|^2}{(\lambda_p -\lambda_j)^2} \geq \frac{1}{2T^2} \cdot \left[ \Var S^{(p,p+1)} \cdot \left(\frac{\lambda_p}{\sqrt{n} \delta_{(p)}} \right)^2+ \frac{s_p}{\kappa_p} \cdot \frac{r_{\mathrm{eff}}}{n}  \right] - 12 \left(\frac{Kt/\kappa_p}{\sqrt{n}} \right)^2.
\]
And hence, with probability at least $\frac{1}{4}-\frac{1}{2T^2}-\frac{2r^2}{t^2} - e^{-4r_{\mathrm{eff}}t^2}$,
\[
\|M_1+M_2\| \geq \frac{1}{2T} \left[ \sqrt{\Var S^{(p,p+1)}} \cdot \frac{\lambda_p}{\sqrt{n} \delta_{(p)}} + \sqrt{\frac{s_p}{\kappa_p}} \cdot \sqrt{\frac{r_{\mathrm{eff}}}{n}} \right] - \frac{4 tK/\kappa_p}{\sqrt{n}}.
\]
Simply choosing $T=2$, we complete the proof of the lower bound (Theorem \ref{theo: uplowerboundM1/2 subGau}).

\section{\texorpdfstring{$\Gamma$}{Gamma} encloses only \texorpdfstring{$\tilde{\lambda}_p$}{lambdap} : a dimension argument} \label{subsec: dim argument}

We recall some important notations. Consider the spectral decomposition of $M$, $ M = \sum_{i=1}^{d} \lambda_i u_i u_i^\top$, in which \( \lambda_1 \geq \lambda_2 \geq \dots \geq \lambda_d  \ge 0\) are the eigenvalues, with corresponding orthonormal eigenvectors \( u_i, 1 \leq i \leq d \).  The effective rank of $M$ is $r_{\mathrm{eff}}= \frac{\sum_{i=1}^d \lambda_i}{\lambda_1}$.
We define $\tilde{\lambda}_i, \tilde{u}_i$ with respect to $\tilde{M}$. 
For each $1 \leq j \leq d-1$, the $j^{th}$ eigenvalue gap is $\delta_j =\lambda_j -\lambda_{j+1}$. Denote $\delta_{(j)}:= \min\{\delta_{j-1}, \delta_j\}$. 
For each $1 \leq j \leq d$, let $\Pi_j$ ($\tilde{\Pi}_j$) be the orthogonal projection onto the space spanned by the $j$-leading eigenvectors of $M$ ($\tilde{M}$ resp.) 
Let $r_j$ be the largest natural number such that $\lambda_{r_j+1} < \frac{\lambda_j}{2}$ and $\kappa_j=\frac{\lambda_1}{\lambda_j}.$ 

In this section, we fix $p$ as the chosen rank parameter. Our goal is to prove the following claim from Section \ref{sec: proof}. 
\begin{lemma} \label{lem: dimargu}
Let $\Gamma$ be a rectangle whose vertical edges bisect the intervals
$(\lambda_{p+1},\lambda_{p})$ and $(\lambda_p,\lambda_{p-1})$, and whose horizontal edges lie at distance $4\delta_{(p)}$ from the real axis.      There is a constant $C > 0$ so that for any $t > 0$, if 
\[
\textstyle \max\!\left\{
\sqrt p\,\sqrt{\frac{r_{\mathrm{eff}}}{n}},
\;
r_p\frac{\lambda_p}{\sqrt n\,\delta_{(p)}},
\;
\sqrt{r_p} \,\sqrt{\frac{\lambda_p r_{\mathrm{eff}}}{n\delta_{(p)}}}
\right\}
<
\frac{1}{C\kappa_p K t},
\]
then with probability at least $1-e^{-4 r_{\mathrm{eff}}t^2} -\frac{2r^2}{t^2},$ $\Gamma$ encloses only  $\tilde{\lambda}_p$ (i.e., $\tilde{\lambda}_j$ is outside of $\Gamma$ for any $j \neq p$).
\end{lemma}
To prove Lemma \ref{lem: dimargu}, we reduce the problem to two simpler claims. Indeed, the first claim shows that suitable deterministic bounds on the perturbation imply that the contour contains exactly the desired eigenvalues. The second claim verifies these deterministic bounds probabilistically in our covariance setting. Combining the two claims immediately yields Lemma \ref{lem: dimargu}.

Before stating and proving our claims, we need a few new notations. For a given index $k$, set
\begin{itemize}
    \item $x_k:=\max_{1\leq i,j \leq r_k}|u_i^\top E u_j|,$
    \item $ w_k:= \max_{1 \leq i \leq r_k} \|E u_i\|, $
    \item $ y_k:= \max_{\substack{1 \leq i< j \leq r_k \\1 \leq h \leq k}}|u_i^\top E \sum_{l >r} \frac{u_l u_l^\top}{\lambda_h - \lambda_l} E u_j|,$
    \item  $\Gamma_k$ is a rectangle, whose left vertical edge bisects the interval $(\lambda_{k+1}, \lambda_{k})$, whose right vertical edge passes the point $\|M\|+1.1\|E\|$, and whose horizontal edges lie at distance $4\delta_k$ from the real axis.
    \end{itemize}
In fact, the contour $\Gamma_k$ encloses exactly the first $k$ leading eigenvalues of $M$. Hence, the contour $\Gamma$, which isolates $\lambda_p$, is naturally obtained as the contour corresponding to the region between $\Gamma_p$ and $\Gamma_{p-1}$.

\noindent \textbf{Claim 1.}  For any $1 \leq k \leq d$, if 
\(
\max \left\{ \frac{\sqrt{k}\|E\|}{\lambda_k/2}, \frac{r_k x_k}{\delta_k}, \frac{\sqrt{r_k} w_k}{\sqrt{\lambda_k \delta_k/2}} \right\} < \frac{1}{12}\,\,\, \text{and}\, \frac{\|E\|}{\lambda_k/2} + \frac{\sqrt{r_k} x_k}{\delta_k}+ \frac{\sqrt{r_k} y_k}{\delta_k} < \frac{1}{12\sqrt{k}},
\)
then the contour $\Gamma_k$ contains exactly $k$ leading eigenvalues, $\tilde{\lambda}_1, \tilde{\lambda}_2, \cdots, \tilde{\lambda}_k$, of $\tilde{M}$.

Intuitively, Claim~1 states that if the eigenspace perturbation is sufficiently small, then each contour $\Gamma_k$ encloses exactly the intended perturbed eigenvalues.

In particular, applying Claim~1 with $k=p$, we see that if
\[ \textstyle
\max \left\{
\frac{\sqrt{p}\|E\|}{\lambda_p/2},
\frac{r_p x_p}{\delta_p},
\frac{\sqrt{r_p}\,w_p}{\sqrt{\lambda_p\delta_p/2}}
\right\}
<\frac{1}{12},\,\text{and}\,
\frac{\|E\|}{\lambda_p/2}
+\frac{\sqrt{r_p}x_p}{\delta_p}
+\frac{\sqrt{r_p}y_p}{\delta_p}
<
\frac{1}{12\sqrt p},
\]
then $\Gamma_p$ contains exactly the $p$ leading perturbed eigenvalues
\(
\tilde\lambda_1,\tilde\lambda_2,\ldots,\tilde\lambda_p.
\)
Similarly, applying Claim~1 with $k=p-1$, if
\[\textstyle
\max \left\{
\frac{\sqrt{p-1}\|E\|}{\lambda_{p-1}/2},
\frac{r_{p-1}x_{p-1}}{\delta_{p-1}},
\frac{\sqrt{r_{p-1}}\,w_{p-1}}{\sqrt{\lambda_{p-1}\delta_{p-1}/2}}
\right\}
<\frac{1}{12},
\, \text{and}\,
\frac{\|E\|}{\lambda_{p-1}/2}
+\frac{\sqrt{r_{p-1}}x_{p-1}}{\delta_{p-1}}
+\frac{\sqrt{r_{p-1}}y_{p-1}}{\delta_{p-1}}
<
\frac{1}{12\sqrt p},
\]
then $\Gamma_{p-1}$ contains exactly the first $p-1$ perturbed eigenvalues
\(
\tilde\lambda_1,\tilde\lambda_2,\ldots,\tilde\lambda_{p-1}.
\)

Since $\delta_{(p)}= \min\{\delta_p, \delta_{p-1}\}$ and 
\(
x_{p-1}\le x_p,\,\,
w_{p-1}\le w_p,\,\,
y_{p-1}\le y_p,\,\,
r_{p-1}\le r_p,\,\text{and $\lambda_p\le\lambda_{p-1}$,}
\)
 both sets of conditions are implied by the unified condition
\begin{equation}\label{condition for gamma} \textstyle
\max \left\{
\frac{\sqrt{p}\|E\|}{\lambda_p/2},
\frac{r_p x_p}{\delta_{(p)}},
\frac{\sqrt{r_p}\,w_p}{\sqrt{\lambda_p\delta_{(p)}/2}}
\right\}
<\frac{1}{12},
\,\text{and}\,
\frac{\|E\|}{\lambda_p/2}
+\frac{\sqrt{r_p}x_p}{\delta_{(p)}}
+\frac{\sqrt{r_p}y_p}{\delta_{(p)}}
<
\frac{1}{12\sqrt p}.
\end{equation}
Consequently, $\Gamma_p$ contains exactly the $p$ leading eigenvalues of $\tilde{M}$, while $\Gamma_{p-1}$ contains exactly the $p-1$ leading ones. It follows that the contour $\Gamma$ encloses only the eigenvalue $\tilde\lambda_p$.

Next, we verify that Condition \eqref{condition for gamma} holds with high probability in our covariance setting (Claim 2). 

\noindent \textbf{Claim 2.} For any $t> 0$, with probability at least $1-e^{-4 r_{\mathrm{eff}}t^2} -\frac{2r^2}{t^2}$, the followings hold. 
\begin{itemize}
    \item $\|E\| \leq 40 Kt \cdot \lambda_1 \sqrt{\frac{r_{\mathrm{eff}}}{n}}, w_p \leq 4 Kt \cdot \lambda_1 \sqrt{\frac{r_{\mathrm{eff}}}{n}}.$
    \item $x_p \leq 36Kt \cdot \frac{\lambda_1}{\sqrt{n}},\, y_p \leq  1600 K^2 \kappa_p \cdot \frac{r_{\mathrm{eff}}}{n} \lambda_1.$
\end{itemize}
Claim 2 is obtained directly from Theorem \ref{theo: ESubgau},  Lemma \ref{lem: w bound}, and Lemma \ref{lemma: xyModeil1}.  

Substituting the estimates in Claim 2, we can rewrite Condition \eqref{condition for gamma} into  %
\[
 \max \left\{ \sqrt{\frac{r_{\mathrm{eff}}}{n}}, \frac{r_p \lambda_p}{\sqrt{n} \delta_{(p)}}, \sqrt{\frac{r_p \cdot \lambda_p r_{\mathrm{eff}}}{n \delta_{(p)}}}  \right\} < \frac{1}{C\kappa_p K t},
\]
for some constant $C > 0$. This is the assumption of Lemma \ref{lem: dimargu} (Theorem \ref{theo: upUpperM1/2 subGau}) as desired.

Therefore, Lemma \ref{lem: dimargu} reduces entirely to the deterministic statement of Claim 1.

\subsection{Proof of Claim 1}
We need to show that if
\(
\max \left\{ \frac{\sqrt{k}\|E\|}{\lambda_k/2}, \frac{r_k x_k}{\delta_k}, \frac{\sqrt{r_k} w_k}{\sqrt{\lambda_k \delta_k/2}} \right\} < \frac{1}{12}\,\,\, \text{and}\, \frac{\sqrt{k}\|E\|}{\lambda_k/2} + \frac{\sqrt{r_k} x_k}{\delta_k}+ \frac{\sqrt{r_k} y_k}{\delta_k} < \frac{1}{12\sqrt{k}},
\)
then the contour $\Gamma_k$ contains exactly $k$ leading eigenvalues of $\tilde{M}$. 

First, we use Tran-Vu result \cite[Theorem 2.4]{DKTranVu1}, which guarantees that if 
\(
\max \left\{ \frac{\sqrt{k}\|E\|}{\lambda_k/2}, \frac{r_k x_k}{\delta_k}, \frac{\sqrt{r_k} w_k}{\sqrt{\lambda_k \delta_k/2}} \right\} < \frac{1}{12},
\)
then
\[
\|\tilde{\Pi}_{\Gamma_k} -\Pi_k\| \leq 12\sqrt{k} \left(\frac{\|E\|}{\lambda_k/2} + \frac{\sqrt{r_k} x_k}{\delta_k}+ \frac{\sqrt{r_k} y_k}{\delta_k} \right),
\]
where 
\[
\tilde{\Pi}_{\Gamma_k}:=\sum_{\tilde{\lambda}_i \,\,\text{inside}\,\Gamma_k } \tilde{u}_i \tilde{u}_i^\top.
\]
Moreover, the second condition of Claim 1 further implies 
$$\textstyle \|\tilde{\Pi}_{\Gamma_k} -\Pi_k\| \leq 12\sqrt{k} \left(\frac{\|E\|}{\lambda_k/2} + \frac{\sqrt{r_k} x_k}{\delta_k}+ \frac{\sqrt{r_k} y_k}{\delta_k} \right)< 1.$$
Since $\tilde{\Pi}_{\Gamma_k}, \Pi_k$ are projections, this inequality implies 
$$\dim \tilde{\Pi}_{\Gamma_k} = \dim \Pi_k = k;\,\,\text{see e.g., \cite[Lemma 4.1]{TranVUeigenvalue}.}$$
On the other hand, because $\Gamma_p$ passes the point $\|M\|+1.1\|E\|$ and the Weyl's inequality gives $|\tilde{\lambda}_1- \lambda_1| \leq \|E\|$, $\tilde{\lambda}_1$ must stay inside $\Gamma_k$. Consequently, $\Gamma_k$ must contain exactly the $k$ leading eigenvalues of $\tilde{M}$.
This proves Claim 1, and hence completes our argument. 
\bibliographystyle{amsrefs}
\bibliography{RefO}

\appendix

\section{Derivation of Theorems \ref{thm:small-dependentp}, \ref{thm:medium-dependentp}, and \ref{thm:large-dependentp} from Theorem \ref{theo: upUpperM1/2 subGau} and Theorem \ref{theo: uplowerboundM1/2 subGau}}  \label{subsection:routine}

First, in Theorem \ref{theo: uplowerboundM1/2 subGau}, we choose $t=6r$ and $C'=C_1 Kr$ for 
$\textstyle C_1 \geq 48 \times 900 \cdot \max \big\{ \sqrt{r/\Var S^{(p,\gamma_p)}}, \sqrt{\kappa_p^{3}/4s_p} \big\}.$ We directly set $C_1=C'' \kappa_p^{3/2} \sqrt{r}$, or equivalently choose $C'=C'' K (\kappa_p r)^{3/2}$ for some sufficiently large constant $C''$.  Thus, if 
$$\textstyle n \geq C K^4 (\kappa_p r)^5 \cdot \max \big\{ r_{\mathrm{eff}} \cdot \frac{r\lambda_p}{\delta_{(p)}}, \left(\frac{r \lambda_p}{\delta_{(p)}} \right)^2 \big\},$$
 then with probability at least $\frac{1}{16},$
\begin{equation*}
    \begin{split}
\textstyle \left\| \tilde{u}_p \tilde{u}_p^\top - u_p u_p^\top \right\|      &  \textstyle  \geq \frac{1}{4} \left[ \sqrt{\Var S^{(p,\gamma_p)}} \cdot \frac{\lambda_p}{\sqrt{n} \delta_{(p)}} + \sqrt{\frac{s_p}{\kappa_p}} \cdot  \sqrt{\frac{r_{\mathrm{eff}}}{n}} \right] - \frac{24 Kr/\kappa_p}{\sqrt{n}}  -\frac{5400Kr}{C'}\left( \kappa_p\sqrt{\frac{r_{\mathrm{eff} }  }{n}} +  \frac{ \sqrt{r} \lambda_p}{ \sqrt{n} \delta_{(p)} } \right) \\
& (\text{replacing $t= 6r$}) \\
 & \textstyle \geq \frac{1}{8} \left[ \sqrt{\Var S^{(p,\gamma_p)}} \cdot \frac{\lambda_p}{\sqrt{n} \delta_{(p)}} + \sqrt{\frac{s_p}{\kappa_p}} \cdot  \sqrt{\frac{r_{\mathrm{eff}}}{n}} \right] - \frac{24 Kr/\kappa_p}{\sqrt{n}}\\
 & (\text{replacing $C'=C'' K(\kappa_p r)^{3/2}$ for some sufficiently large universal constant $C''$}).
    \end{split}
\end{equation*}
If we further assume that 
\[ \textstyle
\sqrt{r_{\mathrm{eff}}}+ \frac{\lambda_p}{\delta_{(p)}} \geq 16 \times 24 \frac{Kr}{\kappa_p} \cdot \max \{\sqrt{1/\Var S^{(p,\gamma_p)}}, \sqrt{\kappa_p/s_p}\},
\]
then $\frac{1}{16} \left[ \sqrt{\Var S^{(p,\gamma_p)}} \cdot \frac{\lambda_p}{\sqrt{n} \delta_{(p)}} + \sqrt{\frac{s_p}{\kappa_p}} \cdot  \sqrt{\frac{r_{\mathrm{eff}}}{n}} \right] \geq \frac{24 Kr/\kappa_p}{\sqrt{n}} $, and hence,
\begin{equation} \label{lower concrete}
   \left\| \tilde{u}_p \tilde{u}_p^\top - u_p u_p^\top \right\|        \geq \frac{1}{16} \left[ \sqrt{\Var S^{(p,\gamma_p)}} \cdot \frac{\lambda_p}{\sqrt{n} \delta_{(p)}} + \sqrt{\frac{s_p}{\kappa_p}} \cdot  \sqrt{\frac{r_{\mathrm{eff}}}{n}} \right]. 
\end{equation}

On the other hand, applying Theorem \ref{theo: upUpperM1/2 subGau} for $t = \frac{2r}{\sqrt{\ep}}$, we obtain that if 
\[
\textstyle n \geq C \frac{(K\kappa_p r)^2}{\epsilon} \cdot \max \big\{ r_{\mathrm{eff}} \cdot \frac{r\lambda_p}{\delta_{(p)}}, \left(\frac{r \lambda_p}{\delta_{(p)}} \right)^2 \big\},
\]
then with probability at least $1 - \epsilon$, 
\begin{equation} \label{upper concrete}
     \left\| \tilde{u}_p \tilde{u}_p^\top - u_p u_p^\top \right\|        \leq 560 \frac{Kr}{\sqrt{\ep}} \left( \kappa_p \sqrt{\frac{r_{\mathrm{eff} }  }{n}} +  \frac{ \sqrt{r} \lambda_p}{ \sqrt{n} \delta_{(p)} } \right).
\end{equation}
Therefore, to ensure both upper bound \eqref{upper concrete} and lower bound \eqref{lower concrete} on $\|\tilde{u}_p \tilde{u}_p^\top - u_p u_p^\top\|$ at the same time, we simply assume 
\begin{equation} \label{condition concrete}
   \textstyle n \geq C \frac{K^4 (\kappa_p r)^5}{\epsilon} \cdot  \max \big\{ r_{\mathrm{eff}} \cdot \frac{r\lambda_p}{\delta_{(p)}}, \left(\frac{r \lambda_p}{\delta_{(p)}} \right)^2 \big\}.
\end{equation}
Theorems \ref{thm:small-dependentp}, \ref{thm:medium-dependentp}, and \ref{thm:large-dependentp} are simply the explicit combinations of  \eqref{lower concrete}, \eqref{upper concrete}, and \eqref{condition concrete}. For example, in the \textit{Medium effective rank} case, where $\frac{r\lambda_p}{\delta_{(p)}} < r_{\mathrm{eff}} \leq \frac{r}{\kappa_p^2}\big(\frac{\lambda_p}{\delta_{(p)}} \big)^2$, Condition \eqref{condition concrete} becomes 
$$ \textstyle n \geq C \frac{K^4 \kappa_p^5 r^6}{\epsilon} \cdot r_{\mathrm{eff}} \frac{\lambda_p}{\delta_p},$$
while \eqref{lower concrete} and \eqref{upper concrete} yield
\[ \textstyle
\frac{1}{16} \left[ \sqrt{\Var S^{(p,\gamma_p)}} \cdot \frac{\lambda_p}{\sqrt{n} \delta_{(p)}} + \sqrt{\frac{s_p}{\kappa_p}} \cdot \sqrt{\frac{r_{\mathrm{eff}}}{n}} \right]\leq  \left\| \tilde{u}_p \tilde{u}_p^\top - u_p u_p^\top \right\|        \leq 560 \frac{Kr}{\sqrt{\ep}} \left( \kappa_p \sqrt{\frac{r_{\mathrm{eff} }  }{n}} +  \frac{ \sqrt{r} \lambda_p}{ \sqrt{n} \delta_{(p)} } \right).
\]
This implies 
\[
 \frac{1}{16}\sqrt{\Var S^{(p,\gamma_p)}} \cdot \frac{\lambda_p}{\sqrt{n} \delta_{(p)}}  \leq \left\| \tilde{u}_p \tilde{u}_p^\top - u_p u_p^\top \right\|        \leq CK \sqrt{\frac{r^3}{\epsilon}} \cdot \frac{\lambda_p}{\sqrt{n}\delta_{(p)}}\,\,\,\text{as desired.}
\]

\section{Proof of Lemma \ref{lemma: xyModeil1} - Computation of moments}
\label{section: moments} 

In this section, we prove Lemma \ref{lemma: xyModeil1} by the moment method. 
First, let us recall our set of notations.  Let $X = [ \xi_1, \,\, \xi_2, \,\,...,\,\, \xi_d]^\top$ be a random vector with  covariance matrix $M = (m_{kl})_{1 \leq k, l \leq d},$ where $m_{kl} = m_{lk} = \Cov(\xi_k, \xi_l).$ For each $1 \leq i \leq n$, the $i^{th}$ sample of $X=[\xi_1, \xi_2, \cdots, \xi_d]$ is 
$$X_i:= [ \xi_{i1}, \,\xi_{i2},\, \cdots, \xi_{id}]^\top,\,\, \text{and hence}\,\,\, X_iX_i^\top=(\xi_{ik} \xi_{il})_{1 \leq k,l \leq d},\,\, \E \xi_{ik} \xi_{il} = m_{kl}.$$
\noindent By the definition of noise, we have 
$$E=\tilde{M} - M = \sum_{i=1}^n \frac{1}{n} \left(X_i X_i^\top - \E X_i X_i^\top \right).$$
Denote $E = (e_{lk})_{1 \leq l,k \leq d}$.  Then,
\begin{equation} \label{for: noise}
    e_{kl} = \frac{1}{n}\sum_{i=1}^{n} (\xi_{ik} \xi_{il} - m_{kl}).
\end{equation}
Our goal is to estimate 
 $$   \textstyle   x:= \max_{i,j \in I_p} \norm{u_i^\top E u_j} \,\,\text{and}\,\,y:= \max_{\substack{ i < j \in I_p }} \norm{ u_i^\top E \left(\sum_{l \notin I_p} \frac{u_l u_l^\top}{\lambda_p - \lambda_l} \right) E u_j}.$$

To reduce the complexity of indices, without loss of generality, we can assume that $1 \in I_p$. Thus, $I_p=\{1,2,\dots, r\}$. To bound $x$, we estimate the first and second moments of $\norm{u_1^\top E u_2}$ and of $\norm{u_1^\top E u_1}$. To handle $y$, we bound the first and second moments of  $   \textstyle  \norm{ u_1^\top E \left(\sum_{l>r} \frac{u_l u_l^\top}{\lambda_p - \lambda_l} \right) E u_2}$. 
\subsection{General Model} \label{proof: lemmaModel1} Without imposing additional structure on $X$, the moments of the skew-parameters $(x,y)$ are computed as follows. 
\subsubsection{Bounding $x$}
By \eqref{for: noise}, we have 
\begin{equation} \label{for: u1Eu2}
\begin{split}
u_1^\top E u_2 & =    \textstyle   \sum_{1 \leq k,l \leq d} u_{1k} u_{2l} e_{kl} = \frac{1}{n}\sum_{1 \leq k,l \leq d} \sum_{i \leq n} u_{1k} u_{2l} (\xi_{ik} \xi_{il} - m_{kl}).
\end{split}
\end{equation}
Therefore, $\E u_1^\top E u_2  = 0$. 

Next, we expand $\E (u_1^\top E u_2 )^2 $ as 
\begin{equation*} 
\begin{split}
 &  \frac{1}{n^2} \sum_{\substack{1 \leq k,l \leq d \\ 1 \leq k',l' \leq d}} \sum_{1 \leq i,i' \leq n} u_{1k} u_{2l} u_{1k'} u_{2l'} \E\left[(\xi_{ik} \xi_{il} - m_{kl}) (\xi_{i'k'} \xi_{i'l'} - m_{k'l'}) \right] \\
& = \frac{1}{n^2} \sum_{1 \leq i \leq n} \sum_{\substack{1 \leq k,l \leq d \\ 1 \leq k',l'\leq d}} u_{1k} u_{2l} u_{1k'} u_{2l'} \E\left[ (\xi_{ik} \xi_{il} - m_{kl}) (\xi_{ik'} \xi_{il'} - m_{k'l'}) \right] \\
& = \frac{1}{n^2} \sum_{1 \leq i\leq n} \sum_{\substack{1 \leq k,l \leq d \\ 1\leq k',l' \leq d}} u_{1k} u_{2l} u_{1k'} u_{2l'}  \left[ \E \left( \xi_{ik} \xi_{il} \xi_{ik'} \xi_{il'} \right) - m_{kl} m_{k'l'} \right].
\end{split} 
\end{equation*}
Since $X_1,X_2, \dots, X_n$ are iid samples of $X$, we finally obtain
\begin{equation} \label{for: secondmomnetu1Eu2}
    \E (u_1^\top E u_2 )^2 = \frac{1}{n} \sum_{\substack{1 \leq k,l \leq d \\ 1 \leq k',l' \leq d}} u_{1k} u_{2l} u_{1k'} u_{2l'}  \left( \E \left( \xi_{k} \xi_{l} \xi_{k'} \xi_{l'} \right) - m_{kl} m_{k'l'} \right).
\end{equation}

\subsubsection{Bounding $y$} Similarly, we have 
\begin{equation*}
\begin{split}
   \E u_1^\top E \sum_{d\geq l >r} \frac{u_l u_l^\top}{\lambda_p-\lambda_l} E u_2 & =  \E \left( \sum_{d \geq k,j,m,i \geq 1} \sum_{d \geq l > r} \frac{u_{1k} e_{kj} u_{lj} u_{lm} e_{mi} u_{2i}}{\lambda_p - \lambda_l} \right) \\
& = \sum_{d \geq k,j,m,i \geq 1} \sum_{d \geq l > r} u_{1k} \frac{u_{lj} u_{lm}}{\lambda_p - \lambda_l} u_{2i} \E (e_{kj} e_{mi}).
\end{split}
\end{equation*}
By \eqref{for: noise}, the RHS is rewritten as 
$$ \frac{1}{n^2} \sum_{d \geq k,j,m,i \geq 1} \sum_{d \geq l > r} u_{1k} \frac{u_{lj} u_{lm}}{\lambda_p -\lambda_l} u_{2i} \E  \left( \sum_{I=1}^n \xi_{Ik} \xi_{Ij} -m_{kj} \right) \left( \sum_{J=1}^n \xi_{Jm} \xi_{Ji} - m_{mi} \right).$$
\noindent Note that 
$$\E (\xi_{Ik} \xi_{Ij} -m_{kj})(\xi_{Jm} \xi_{Ji} - m_{mi} ) = 0 \,\, \text{for any}\,\, I \neq J,$$
$$\E (\xi_{Ik} \xi_{Ij} -m_{kj})(\xi_{Im} \xi_{Ii} - m_{mi} ) = \E (\xi_{I'k} \xi_{I'j} -m_{kj})(\xi_{I'm} \xi_{I'i} - m_{mi} ) \,\,\text{for any}\,\, 1\leq I, I' \leq n.$$
Thus, the above sum simplifies to 
$$\frac{1}{n} \sum_{d\geq k,j,m,i \geq 1} \sum_{d \geq l > r} u_{1k} \frac{u_{lj} u_{lm}}{\lambda_p -\lambda_l} u_{2i} \E  \left(  \xi_{k} \xi_{j} -m_{kj} \right) \left(  \xi_{m} \xi_{i} - m_{mi} \right).$$
Distributing $u_{1k}, u_{lj}, u_{lm}, u_{2i}$ in the positions of $\xi_k, \xi_j, \xi_m, \xi_i$ respectively, we obtain the more useful (see later) sum.
\begin{equation} \label{for: firstmomentu1EEu2}
\begin{split}
&\textstyle \sum_{l=r+1}^d \E \frac{1}{n(\lambda_p - \lambda_l)} \cdot \\
&\left[ (\sum_{k=1}^d u_{1k} \xi_k)(\sum_{j=1}^d \xi_j u_{lj}) - \E (\sum_{k=1}^d u_{1k} \xi_k)(\sum_{j=1}^d \xi_j u_{lj})  \right] \left[ (\sum_{i=1}^d u_{2i} \xi_i)(\sum_{m=1}^d \xi_m u_{lm}) - \E (\sum_{i=1}^d u_{2i} \xi_i)(\sum_{m=1}^d \xi_m u_{lm})  \right].
\end{split}
\end{equation}
\subsection{Spiked Population Model} \label{subsec: xypspikedpopulation}
Recall the setting of the spiked population model that
\begin{equation} \label{for: XModel1}
\begin{split}
X & = M^{1/2} Y = \sum_{i=1}^d \sqrt{\lambda_i} u_i u_i^\top Y = \sum_{i=1}^d \sqrt{\lambda_i} (u_i^\top Y) u_i.
\end{split}
\end{equation}
Since $ X=[\xi_1, \,\xi_2,\, \cdots, \xi_d]^\top$,  for each $1 \leq k \leq d$, we further have 
\begin{equation*}
\begin{split}
\xi_k &=  \sum_{i=1}^d \sqrt{\lambda_i} (u_i^\top Y) u_{ik}  = \sum_{i=1}^d \sqrt{\lambda_i} u_{ik} \sum_{j=1}^d u_{ij} y_j = \sum_{j=1}^d y_j \left( \sum_{i=1}^d u_{ik} u_{ij} \sqrt{\lambda_i} \right).
\end{split}
\end{equation*}
Set $ s_{jk}:=\sum_{i=1}^d u_{ik} u_{ij} \sqrt{\lambda_i}.$ The above identity is shortly written as 
\begin{equation}\label{for: xik}
    \xi_k = \sum_{j=1}^d y_j s_{jk}.
\end{equation}
Therefore, 
\begin{equation} \label{for: mkl}
    m_{kl} = \E (\xi_k \xi_l) = \sum_{d \geq j,j' \geq 1} (\E y_j y_{j'}) s_{jk} s_{j'l} = \sum_{\alpha=1}^d s_{\alpha k} s_{\alpha l}.
\end{equation}
The last equality follows from the fact that the entries $y_1,y_2, \dots, y_d$ are 8-wise independent random variables with zero mean and variance one. 
\subsubsection{Bounding $x$ and proving the first part of Lemma \ref{lemma: xyModeil1}} Using \eqref{for: xik}, we have
$$ \E (\xi_k \xi_l \xi_{k'} \xi_{l'}) = \sum_{d \geq j_1,j_2,j_3,j_4 \geq 1} \E (y_{j_1} y_{j_2} y_{j_3} y_{j_4})s_{j_1 k} s_{j_2 l} s_{j_3 k'} s_{j_4 l'}.$$
Note that 
$$\E (y_{j_1} y_{j_2} y_{j_3} y_{j_4}) = \begin{cases}
    1 \,\,\text{if either}\,\, (j_1= j_2 = \alpha \neq \beta= j_3=j_4)\,\text{or}\,(j_1=j_3 \neq j_2=j_4)\,\text{or}\,(j_1=j_4 \neq j_2=j_3),\\
    \E y^4 \,\,\text{if}\,\,j_1=j_2 =j_3=j_4= \alpha, \\
    0 \,\,\,\text{otherwise}
\end{cases}. $$
Thus, $\E (\xi_k \xi_l \xi_{k'} \xi_{l'}) $ equals
\begin{equation*}
\begin{split}
 & \sum_{d\geq \alpha \neq \beta \geq 1} \left(s_{\alpha k} s_{\alpha l} s_{\beta k'} s_{\beta l'}+s_{\alpha k} s_{\beta l} s_{\alpha k'} s_{\beta l'}+s_{\alpha k} s_{\beta l} s_{\beta k'} s_{\alpha l'} \right)+ \E y^4 \sum_{\alpha=1}^d s_{\alpha k} s_{\alpha l} s_{\alpha k'} s_{\alpha l'} \\
 & = \sum_{d\geq \alpha, \beta \geq 1} \left(s_{\alpha k} s_{\alpha l} s_{\beta k'} s_{\beta l'}+s_{\alpha k} s_{\beta l} s_{\alpha k'} s_{\beta l'}+s_{\alpha k} s_{\beta l} s_{\beta k'} s_{\alpha l'} \right)+ (\E y^4 -3) \sum_{\alpha=1}^d s_{\alpha k} s_{\alpha l} s_{\alpha k'} s_{\alpha l'}.
\end{split}
\end{equation*}
Moreover, by \eqref{for: mkl}, we have
$$   \textstyle   m_{kl} m_{k'l'} = \left(\sum_{d\geq \alpha \geq 1} s_{\alpha k} s_{\alpha l} \right)  \cdot \left(\sum_{d\geq \beta \geq 1} s_{\beta k'} s_{\beta l'} \right) =   \sum_{d\geq \alpha, \beta \geq 1} s_{\alpha k} s_{\alpha l} s_{\beta k'} s_{\beta l'}.$$
These identities imply
\begin{equation} \label{for: xiklk'l'}
 \E (\xi_k \xi_l \xi_{k'} \xi_{l'}) - m_{kl} m_{k'l'}= \sum_{d\geq\alpha, \beta \geq 1} \left(s_{\alpha k} s_{\beta l} s_{\alpha k'} s_{\beta l'}+s_{\alpha k} s_{\beta l} s_{\beta k'} s_{\alpha l'} \right)+ (\E y^4 -3) \sum_{\alpha=1}^d s_{\alpha k} s_{\alpha l} s_{\alpha k'} s_{\alpha l'}.
\end{equation}
Combining \eqref{for: secondmomnetu1Eu2} and \eqref{for: xiklk'l'}, we rewrite $\E (u_1^\top E u_2)^2 $ as
\begin{equation*}
    \begin{split}
            &  \frac{1}{n}\sum_{d \geq k,k',l,l' \geq 1} u_{1k} u_{1k'} u_{2l} u_{2l'} \bigg[\sum_{\alpha, \beta =1}^d \left(s_{\alpha k} s_{\beta l} s_{\alpha k'} s_{\beta l'}+s_{\alpha k} s_{\beta l} s_{\beta k'} s_{\alpha l'} \right)+ (\E y^4 -3) \sum_{\alpha=1}^d s_{\alpha k} s_{\alpha l} s_{\alpha k'} s_{\alpha l'}  \bigg] \\
        & =\frac{1}{n} ( M_1 + M_2 + M_3), \,\,\text{where}\\
        M_1 &=\sum_{d \geq k,k',l,l' \geq 1} u_{1k} u_{1k'} u_{2l} u_{2l'} \sum_{\alpha, \beta=1}^d s_{\alpha k} s_{\beta l} s_{\alpha k'} s_{\beta l'} ,\\
        M_2 &=\sum_{d \geq k,k',l,l' \geq 1} u_{1k} u_{1k'} u_{2l} u_{2l'} \sum_{\alpha, \beta=1}^d \left(s_{\alpha k} s_{\beta l} s_{\beta k'} s_{\alpha l'} \right) ,\\
        M_3& = \sum_{d \geq k,k',l,l' \geq 1} u_{1k} u_{1k'} u_{2l} u_{2l'} (\E y^4 -3) \sum_{\alpha=1}^d s_{\alpha k} s_{\alpha l} s_{\alpha k'} s_{\alpha l'}.
    \end{split}
\end{equation*}
We compute $M_1$ as follows. 
\begin{equation}
\begin{split}
M_1&=\sum_{d \geq k,l,k',l' \geq 1} \sum_{\alpha, \beta=1}^d (u_{1k} s_{\alpha k})(u_{1k'} s_{\alpha k'}) (u_{2l} s_{\beta l}) (u_{2l'} s_{\beta l'})\\& = \sum_{\alpha, \beta=1}^d (\sum_{k=1}^d u_{1k} s_{\alpha k})(\sum_{k'=1}^d u_{1k'} s_{\alpha k'}) (\sum_{l=1}^d u_{2l} s_{\beta l}) (\sum_{l'=1}^d u_{2l'} s_{\beta l'}).
\end{split}
\end{equation}
By definition of $s_{\alpha k}$ and the fact that $\{ u_i\}_{i=1}^d$ is an orthonormal system, we have
\begin{equation} \label{id: u1ksak}
    \sum_{k=1}^d u_{1k} s_{\alpha k} = \sum_{k=1}^d \sum_{i=1}^{d} u_{1k} u_{ik} u_{i \alpha} \sqrt{\lambda_i}= \sum_{i=1}^d \sqrt{\lambda_i} u_{i \alpha} \cdot \bigg(\sum_{k=1}^d u_{1k} u_{ik} \bigg) = \sqrt{\lambda_1} u_{1\alpha}.
\end{equation}
Therefore, 
\begin{equation} \label{M1Model1}
    \begin{split}
        M_1 & =    \textstyle   \sum_{\alpha, \beta=1}^d \lambda_1 \lambda_2 u_{1\alpha}^2 u_{2 \beta}^2  = \lambda_1 \lambda_2 \left( \sum_{\alpha=1}^d u_{1\alpha}^2 \right) \left( \sum_{\beta=1}^d u_{2\beta}^2 \right) = \lambda_1 \lambda_2. 
    \end{split}
\end{equation}
By a similar computation, we also obtain 
\begin{equation} \label{M2M3Model1}
    \begin{split}
  M_2 &=    \textstyle   \sum_{d \geq k,k',l,l' \geq 1} u_{1k} u_{1k'} u_{2l} u_{2l'} \sum_{\alpha, \beta=1}^d \left(s_{\alpha k} s_{\beta l} s_{\beta k'} s_{\alpha l'} \right) \\
  & =     \textstyle   \lambda_1 \lambda_2 \sum_{\alpha, \beta =1}^d u_{1\alpha} u_{1\beta} u_{2\alpha} u_{2\beta} = \lambda_1 \lambda_2 \left( \sum_{\alpha=1}^d u_{1\alpha} u_{2\alpha} \right) \left( \sum_{\beta=1}^d u_{1\beta} u_{2\beta} \right) =0, \\
  M_3 & =    \textstyle   \sum_{d \geq k,k',l,l' \geq 1} u_{1k} u_{1k'} u_{2l} u_{2l'} (\E y^4 -3) \sum_{\alpha=1}^d s_{\alpha k} s_{\alpha l} s_{\alpha k'} s_{\alpha l'} \\
  & =    \textstyle    \lambda_1 \lambda_2 \sum_{\alpha=1}^d (\E y_{\alpha}^4 -3) u_{1\alpha}^2 u_{2\alpha}^2.
    \end{split}
\end{equation}
Together \eqref{M1Model1} and \eqref{M2M3Model1} imply
\begin{equation*}
   \textstyle   \E (u_1^\top E u_2)^2  = \frac{\lambda_1 \lambda_2}{n} \left(1+ \sum_{\alpha} (\E y_{\alpha}^4 -3)  u_{1 \alpha}^2 u_{2 \alpha}^2 \right) = \frac{\lambda_1 \lambda_2}{n} \cdot \Var S^{(1,2)}.
\end{equation*}
%
Since $\max_{1 \leq \alpha\leq d} \|y_{\alpha}\|_{\psi_2}^2 \leq K$ (i.e., $\E y_{\alpha}^4 \leq 4K^2$) for some $K \geq 1$, then 
\[
1+ \sum_{\alpha} (\E y_{\alpha}^4 -3)  u_{1 \alpha}^2 u_{2 \alpha}^2 \leq \max\{4K^2,1\} =4K^2.
\]
%
%
Therefore, by Markov inequality and the fact that $\E u_1^\top E u_2 =0$, with probability at least $1 - t^{-2}$,
\[ \textstyle
\norm{u_1^\top E u_2} \leq 2K \frac{\sqrt{\lambda_1 \lambda_2} t}{\sqrt{n}}.
\]

By a similar computation, we also have
$$   \textstyle   \E (u_1^\top E u_1)^2= \frac{\lambda_1^2}{n} \left(2+  \sum_{\alpha} (\E y_{\alpha}^4 -3) u_{1 \alpha}^4  \right).$$
We have the constant $2$ here because $M_2 = \lambda_1^2$ (not $0$) in $\E (u_1^\top E u_1)^2$.

\noindent Thus, with probability at least $1 -t^{-2}$, we also have $\norm{ u_1^\top E u_1} \leq 2K\frac{\lambda_1 t}{\sqrt{n}}$. These inequalities prove the first part of Lemma \ref{lemma: xyModeil1}.
\subsubsection{Bounding $y$ and proving the second part of Lemma \ref{lemma: xyModeil1}} By \eqref{for: firstmomentu1EEu2}, we (recall) have that $\E u_1^\top E \sum_{d\geq l >r} \frac{u_l u_l^\top}{\lambda_p - \lambda_l} E u_2 $ equals
\begin{equation*} 
\begin{split}
&\sum_{l=r+1}^d \frac{1}{n(\lambda_p - \lambda_l)} \cdot \\
&\E \bigg[ (\sum_{k=1}^d u_{1k} \xi_k)(\sum_{j=1}^d \xi_j u_{lj}) - \E (\sum_{k=1}^d u_{1k} \xi_k)(\sum_{j=1}^d \xi_j u_{lj})  \bigg] \bigg[ (\sum_{i=1}^d u_{2i} \xi_i)(\sum_{m=1}^d \xi_m u_{lm}) - \E (\sum_{i=1}^d u_{2i} \xi_i)(\sum_{m=1}^d \xi_m u_{lm})  \bigg].
\end{split}
\end{equation*}
Similar to the previous subsection, we have
$$ \textstyle \sum_{k=1}^d u_{1k} \xi_{k} = \sum_{k=1}^d \sum_{j=1}^d u_{1k} y_j s_{jk} =\sum_{j=1}^d y_j \cdot \sum_{k=1}^d u_{1k} s_{jk} = \sqrt{\lambda_1} \sum_{j=1}^d y_j u_{1j},$$
where the last equality follows from \eqref{id: u1ksak}. Thus,  $\E u_1^\top E \sum_{d\geq l >r} \frac{u_l u_l^\top}{\lambda_p - \lambda_l} E u_2 $ equals
\begin{equation*}
\begin{split}
& \textstyle \frac{1}{n} \sum_{l =r+1}^d \frac{\lambda_l \sqrt{\lambda_1 \lambda_2}}{\lambda_p - \lambda_l} \E \left( (\sum_{j_1=1}^d y_{j_1} u_{1 j_1})(\sum_{j_2=1}^d y_{j_2} u_{l j_2}) - \E (\sum_{j_1=1}^d y_{j_1} u_{1 j_1})(\sum_{j_2=1}^d y_{j_2} u_{l j_2}) \right)\\
& \textstyle  \cdot \left( (\sum_{j_3=1}^d y_{j_3} u_{2 j_3})(\sum_{j_4=1}^d y_{j_4} u_{l j_4}) - \E (\sum_{j_3=1}^d y_{j_3} u_{2 j_3})(\sum_{j_4=1}^d y_{j_4} u_{l j_4}) \right).
\end{split}
\end{equation*}
We expand this enormous expectation into 
$$ \sum_{d \geq j_1,j_2,j_3,j_4 \geq 1} \E \left[ \left(y_{j_1} y_{j_2} u_{1j_1}u_{lj_2} - \E y_{j_1} y_{j_2} u_{1j_1}u_{lj_2}   \right) \left(y_{j_3} y_{j_4} u_{2j_3}u_{lj_4} - \E y_{j_3} y_{j_4} u_{2j_3}u_{lj_4} \right) \right].$$
This term is nontrivial when either $(j_1=j_2 \neq j_3=j_4)$ or $(j_1=j_3 \neq j_2= j_4)$ or $(j_1 =j_4 \neq j_2=j_3)$ or $(j_1=j_2=j_3=j_4)$. The expectation simplifies to 
\begin{equation*}
    \begin{split}
    & \sum_{d \geq \alpha \neq \beta \geq 1} \E \left[ \left(y_{\alpha} y_{\alpha} u_{1\alpha}u_{l\alpha} - \E y_{\alpha} y_{\alpha} u_{1\alpha}u_{l\alpha}   \right) \left(y_{\beta} y_{\beta} u_{2\beta}u_{l\beta} - \E y_{\beta} y_{\beta} u_{2\beta}u_{l\beta} \right) \right] \\
  & + \sum_{d\geq \alpha \neq \beta \geq 1} \E \left[ \left(y_{\alpha} y_{\beta} u_{1\alpha}u_{l\beta} - \E y_{\alpha} y_{\beta} u_{1\alpha}u_{l\beta}   \right) \left(y_{\alpha} y_{\beta} u_{2\alpha}u_{l\beta} - \E y_{\alpha} y_{\beta} u_{2\alpha}u_{l\beta} \right) \right]   \\
  & + \sum_{d\geq \alpha \neq \beta \geq 1} \E \left[ \left(y_{\alpha} y_{\beta} u_{1\alpha}u_{l\beta} - \E y_{\alpha} y_{\beta} u_{1\alpha}u_{l\beta}   \right) \left(y_{\beta} y_{\alpha} u_{2\beta}u_{l\alpha} - \E y_{\beta} y_{\alpha} u_{2\beta}u_{l\alpha} \right) \right] \\
  & + \sum_{\alpha=1}^d \E \left[ \left(y_{\alpha} y_{\alpha} u_{1\alpha}u_{l\alpha} - \E y_{\alpha} y_{\alpha} u_{1\alpha}u_{l\alpha}   \right) \left(y_{\alpha} y_{\alpha} u_{2\alpha}u_{l\alpha} - \E y_{\alpha} y_{\alpha} u_{2\alpha}u_{l\alpha} \right) \right] 
   \end{split}
\end{equation*}
Since $\E y_\alpha y_\beta =0, \E y_\alpha^2 = \E y_\beta^2 =1 $, this sum simplifies to
\begin{equation*}
    \begin{split}
  &\textstyle   \sum_{d \geq \alpha \neq \beta \geq 1} u_{1 \alpha} u_{l \beta} u_{2 \alpha} u_{l \beta} + \sum_{d \geq \alpha \neq \beta \geq 1} u_{1 \alpha} u_{l \beta} u_{2 \beta} u_{l \alpha} + \sum_{\alpha=1}^d (\E y^4 -1) u_{1\alpha} u_{2\alpha} u_{l \alpha}^2 \\
  & = \textstyle  \sum_{d \geq \alpha,\beta \geq 1} u_{1 \alpha} u_{l \beta} u_{2 \alpha} u_{l \beta} + \sum_{d \geq \alpha, \beta \geq 1} u_{1 \alpha} u_{l \beta} u_{2 \beta} u_{l \alpha} + \sum_{\alpha=1}^d (\E y^4 -3) u_{1\alpha} u_{2\alpha} u_{l \alpha}^2 \\
  & =\textstyle  \left( \sum_{\alpha=1}^d u_{1\alpha}u_{2\alpha} \right) \left( \sum_{\beta=1}^d u_{l \beta}^2 \right)+ \left(\sum_{\alpha=1}^d u_{1\alpha} u_{l \alpha} \right) \left(\sum_{\beta=1}^d u_{2\beta} u_{l\beta} \right) + \sum_{\alpha=1}^d (\E y^4 -3) u_{1\alpha} u_{2\alpha} u_{l \alpha}^2 \\
  & = \textstyle  \sum_{\alpha=1}^d  (\E y_{\alpha}^4 -3) u_{1\alpha} u_{2\alpha} u_{l \alpha}^2.
\end{split}
\end{equation*}   
The last equality follows the fact that $u_1, u_2, u_l$ are pairwise orthogonal. Therefore, 
\begin{equation} \label{for: firstmomentu1EEu2Model1}
\begin{split}
\textstyle \E u_1^\top E \sum_{d\geq l >r} \frac{u_l u_l^\top}{\lambda_p - \lambda_l} E u_2 = \frac{\sqrt{\lambda_1 \lambda_2}}{n} \sum_{l =r+1}^d \frac{\lambda_l}{\lambda_p - \lambda_l}  \left( \sum_{\alpha} (\E y_{\alpha}^4 -3) u_{1\alpha} u_{2\alpha} u_{l \alpha}^2 \right).
\end{split}
\end{equation}
Moreover, by the triangle inequality, we have
\begin{equation*}
    \begin{split}
  \textstyle
 \sum_{l =r+1}^d \frac{\lambda_l}{\lambda_p - \lambda_l}  \left( \sum_{\alpha} (\E y_{\alpha}^4 -3) u_{1\alpha} u_{2\alpha} u_{l \alpha}^2 \right) & \leq \big| \sum_{l =r+1}^d \frac{\lambda_l}{\lambda_p - \lambda_l}  \left( \sum_{\alpha} (\E y_{\alpha}^4 -3) u_{1\alpha} u_{2\alpha} u_{l \alpha}^2 \right) \big|   \\
 & \leq \sum_{l=r+1}^d \frac{\lambda_l}{\lambda_p -\lambda_l} \sum_{\alpha} |\E y_\alpha^4 -3| \cdot |u_{1\alpha}| \cdot |u_{2\alpha}| \cdot u_{l \alpha}^2.
    \end{split}
\end{equation*}
Note that for any $l > r$, $\frac{\lambda_l}{\lambda_p -\lambda_l} \leq \frac{\lambda_l}{\lambda_p/2} \leq 1,$ and for any $1 \leq \alpha \leq d$, $|\E y_\alpha^4 -3| \leq 4K^2$, we further obtain 
\begin{equation*}
    \begin{split}
 \sum_{l=r+1}^d \frac{\lambda_l}{\lambda_p -\lambda_l} \sum_{\alpha} |\E y_\alpha^4 -3| \cdot |u_{1\alpha}| \cdot |u_{2\alpha}| \cdot u_{l \alpha}^2 & \leq 4K^2 \cdot \sum_{l=r+1}^d \sum_{\alpha=1}^d    |u_{1\alpha}| \cdot |u_{2\alpha}| \cdot u_{l \alpha}^2 \\
 & =4K^2 \sum_{\alpha=1}^d |u_{1\alpha}| \cdot |u_{2\alpha}| \cdot \sum_{l=r+1}^d u_{l\alpha}^2 \\
 & \leq 4K^2 \sum_{\alpha=1}^d |u_{1\alpha}| \cdot |u_{2\alpha}|.
    \end{split}
\end{equation*}
The last inequality follows the fact that $\{u_1, u_2, \dots, u_d\}$ is an orthonormal basis. Furthermore, by the Cauchy-Schwarz inequality, 
\[
4K^2 \sum_{\alpha=1}^d |u_{1\alpha}| \cdot |u_{2\alpha}| \leq 4K^2 ( \sum_{\alpha} u_{1\alpha}^2)^{1/2} ( \sum_{\alpha} u_{2\alpha}^2)^{1/2} = 4K^2.
\]
Together, these estimates and \eqref{for: firstmomentu1EEu2Model1} lead to 
\begin{equation} \label{expectation: u1EEu2}
    \textstyle \E u_1^\top E \sum_{d\geq l >r} \frac{u_l u_l^\top}{\lambda_p - \lambda_l} E u_2 \leq 4K^2 \cdot\frac{\sqrt{\lambda_1 \lambda_2}}{n}
\end{equation}

Next, we compute the second moment of $u_1^\top E \sum_{l=r+1}^d \frac{u_l u_l^\top}{\lambda_p - \lambda_l} E u_2$, which is 
$$ \frac{1}{n^4} \sum_{\substack{d \geq k,j,m,i \geq 1\\ d \geq k',j',m',i' \geq 1}} \sum_{\substack{l,l' =r+1}}^d u_{1k} u_{1k'} u_{2i} u_{2i'} \frac{u_{lj} u_{lm} u_{l'j'} u_{l'm'}}{(\lambda_p - \lambda_{l})(\lambda_p - \lambda_{l'})} \E (\sum_{I_1=1}^n \cdots)(\sum_{I_2=1}^n \cdots)(\sum_{I_3=1}^n \cdots)(\sum_{I_4=1}^n \cdots),$$
$$\text{where}\,\, \sum_{I_1=1}^n \cdots = \sum_{I_1}^n \xi_{I_1 k} \xi_{I_1 l} - \E \xi_{I_1 k} \xi_{I_1 l}  \,,\,\,\, \sum_{I_2=1}^n \cdots = \sum_{I_1}^n \xi_{I_2 m} \xi_{I_2 i} - \E \xi_{I_2 m} \xi_{I_2 i}  ,$$
$$\sum_{I_3=1}^n \cdots = \sum_{I_3=1}^n \xi_{I_3 k'} \xi_{I_3 l'} - \E \xi_{I_3 k'} \xi_{I_3 l'} \,,\,\,\,\text{and}\,\, \sum_{I_4=1}^n \cdots = \sum_{I_4}=1^n \xi_{I_4 m'} \xi_{I_4 i'} - \E \xi_{I_4 m'} \xi_{I_4 i'}.$$
Notice that when we expand $\E (\sum_{I_1} \cdots)(\sum_{I_2} \cdots)(\sum_{I_3} \cdots)(\sum_{I_4} \cdots)$, the nontrivial terms correspond to the cases that either $(I_1=I_2 \neq I_3 = I_4)$ or $(I_1=I_3 \neq I_2 =I_4)$ or $(I_1=I_4 \neq I_2 =I_3)$ or $(I_1=I_2 =I_3=I_4)$. The numbers of choices for the tuple $(I_1, I_2, I_3, I_4)$ are $(n-1)n$,  $(n-1)n$, $(n-1)n$, and $n$ respectively. Moreover, in each case, the sum is unchanged across the different choices of $(I_1,I_2,I_3, I_4)$. Therefore,  
$$ \textstyle \E \left( u_1^\top E \sum_{l=r+1}^d \frac{u_l u_l^\top}{\lambda_p - \lambda_l} E u_2 \right)^2 = M_1+M_2 + M_3 + M_4,\,\,\text{in which}$$
\begin{equation*}
    \begin{split}
 & M_1:= \frac{n-1}{n^3}  \sum_{\substack{d\geq k,j,m,i \geq 1\\d  \geq k',j',m',i' \geq 1}} \sum_{\substack{d\geq l,l' >r}} u_{1k} u_{1k'} u_{2i} u_{2i'} \frac{u_{lj} u_{lm} u_{l'j'} u_{l'm'}}{(\lambda_p - \lambda_{l})(\lambda_p - \lambda_{l'})} \cdot \\
 & \left( \E (\xi_k \xi_j -m_{kj})(\xi_{m} \xi_{i} -m_{mi}) \right) \left( \E (\xi_{k'} \xi_{j'} -m_{k'l'})(\xi_{m'} \xi_{i'} -m_{m'i'}) \right), \\
 & M_2:= \frac{n-1}{n^3}  \sum_{\substack{d\geq k,j,m,i \geq 1\\d  \geq k',j',m',i' \geq 1}} \sum_{\substack{d\geq l,l' >r}} u_{1k} u_{1k'} u_{2i} u_{2i'} \frac{u_{lj} u_{lm} u_{l'j'} u_{l'm'}}{(\lambda_p - \lambda_{l})(\lambda_p - \lambda_{l'})} \cdot \\
 & \left( \E (\xi_k \xi_j -m_{kj})(\xi_{k'} \xi_{j'} -m_{k'j'}) \right) \left( \E (\xi_{m} \xi_{i} -m_{mi})(\xi_{m'} \xi_{i'} -m_{m'i'}) \right)\\
 &M_3:=   \frac{n-1}{n^3}  \sum_{\substack{d\geq k,j,m,i \geq 1\\d  \geq k',j',m',i' \geq 1}} \sum_{\substack{d\geq l,l' >r}} u_{1k} u_{1k'} u_{2i} u_{2i'} \frac{u_{lj} u_{lm} u_{l'j'} u_{l'm'}}{(\lambda_p - \lambda_{l})(\lambda_p - \lambda_{l'})} \cdot \\
  & \left( \E (\xi_k \xi_j -m_{kj})(\xi_{m'} \xi_{i'} -m_{m'i'}) \right) \left( \E (\xi_{k'} \xi_{l'} -m_{k'l'})(\xi_{m} \xi_{i} -m_{mi}) \right), 
   \end{split}
\end{equation*}
$M_4:=   \sum_{\substack{d \geq k,j,m,i \geq 1\\d \geq k',j',m',i' \geq 1 \\ d\geq l,l' >r}} \frac{u_{1k} u_{1k'} u_{2i} u_{2i'} u_{lj} u_{lm} u_{l'j'} u_{l'm'}\E (\xi_k \xi_j -m_{kj})(\xi_{m} \xi_{i} -m_{mi})(\xi_{k'} \xi_{j'} -m_{k'j'})(\xi_{m'} \xi_{i'} -m_{m'i'})}{n^3(\lambda_p - \lambda_{l})(\lambda_p - \lambda_{l'}) } . $

Similar to our computation for the first moment above, we obtain 
\begin{equation*} 
\begin{split}
& M_2 = \frac{(n-1)\lambda_1 \lambda_2}{n^3} \sum_{\substack{d \geq l,l'>r}} \frac{\lambda_l \lambda_{l'}}{(\lambda_p - \lambda_{l'})(\lambda_p - \lambda_l)} \cdot \\
&\E \bigg[ (\sum_{j_1=1}^d y_{j_1} u_{1 j_1})(\sum_{j_2=1}^d y_{j_2} u_{l j_2}) - \E (\sum_{j_1=1}^d\cdots)(\sum_{j_2=1}^d \cdots) \bigg] \bigg[ (\sum_{j_3=1}^d y_{j_3} u_{1 j_3})(\sum_{j_4=1}^d y_{j_4} u_{l' j_4}) - \E (\sum_{j_3=1}^d \cdots)(\sum_{j_4=1}^d \cdots)  \bigg]   \cdot \\
&  \E \bigg[ (\sum_{j_1=1}^d y_{j_1} u_{2 j_1})(\sum_{j_2=1}^d y_{j_2} u_{l j_2}) - \E (\sum_{j_1=1}^d \cdots)(\sum_{j_2=1}^d \cdots)  \bigg] \bigg[ (\sum_{j_3=1}^d y_{j_3} u_{2 j_3})(\sum_{j_4=1}^d y_{j_4} u_{l' j_4}) - \E (\sum_{j_3=1}^d \cdots)(\sum_{j_4=1}^d \cdots) \bigg],   
\end{split}
\end{equation*}
which simplifies to
\begin{equation*} 
\begin{split}
& \textstyle  \frac{(n-1)\lambda_1 \lambda_2}{n^3} \sum_{\substack{d \geq l,l'>r}} \frac{\lambda_l \lambda_{l'}}{(\lambda_p - \lambda_{l'})(\lambda_p - \lambda_l)} \cdot \\
& \textstyle \left( \sum_{\beta=1}^d u_{l \beta} u_{l' \beta}+ \sum_{\alpha=1}^d (\E y_\alpha^4 - 3)  u_{1\alpha}^2 u_{l\alpha} u_{l' \alpha}  \right) \left( \sum_{\beta=1}^d u_{l \beta} u_{l' \beta}+  \sum_{\alpha=1}^d (\E y_\alpha^4 - 3) u_{2\alpha}^2 u_{l\alpha} u_{l' \alpha}  \right) \\
& \leq \textstyle  \frac{\lambda_1 \lambda_2}{n^2} \sum_{l=r+1}^d \frac{\lambda_l^2}{(\lambda_p - \lambda_{l})^2} \left( 1+  \sum_{\alpha=1}^d (\E y_{\alpha}^4 - 3) u_{1\alpha}^2 u_{l\alpha}^2  \right) \left( 1+ \sum_{\alpha=1}^d (\E y_{\alpha}^4 - 3) u_{2\alpha}^2 u_{l\alpha}^2  \right)+ \\
&\textstyle  \frac{ \lambda_1 \lambda_2}{n^2} \sum_{d\geq l \neq l' >r} \frac{\lambda_l \lambda_{l'}}{(\lambda_p - \lambda_l)(\lambda_p - \lambda_{l'})} \left[\sum_{\alpha=1}^d (\E y_{\alpha}^4 -3) u_{1\alpha}^2 u_{l\alpha} u_{l' \alpha}\right] \left[\sum_{\alpha=1}^d (\E y_{\alpha}^4 -3) u_{2\alpha}^2 u_{l\alpha} u_{l' \alpha} \right].
\end{split}
\end{equation*}
The RHS is grouped into 
\begin{equation*}
    \begin{split}
     & \textstyle  \frac{\lambda_1 \lambda_2}{n^2} \sum_{l=r+1}^d \frac{\lambda_l^2}{(\lambda_p - \lambda_{l})^2} \left( 1+  \sum_{\alpha=1}^d (\E y_{\alpha}^4 - 3) u_{1\alpha}^2 u_{l\alpha}^2 +  \sum_{\alpha=1}^d (\E y_{\alpha}^4 - 3) u_{2\alpha}^2 u_{l\alpha}^2  \right)+ \\
&\textstyle  \frac{ \lambda_1 \lambda_2}{n^2} \sum_{d\geq l, l' >r} \frac{\lambda_l \lambda_{l'}}{(\lambda_p - \lambda_l)(\lambda_p - \lambda_{l'})} \left[\sum_{\alpha=1}^d (\E y_{\alpha}^4 -3) u_{1\alpha}^2 u_{l\alpha} u_{l' \alpha}\right] \left[\sum_{\alpha=1}^d (\E y_{\alpha}^4 -3) u_{2\alpha}^2 u_{l\alpha} u_{l' \alpha} \right].   
    \end{split}
\end{equation*}
Since $ 1+  \sum_{\alpha=1}^d (\E y_{\alpha}^4 - 3) u_{1\alpha}^2 u_{l\alpha}^2 +  \sum_{\alpha=1}^d (\E y_{\alpha}^4 - 3) u_{2\alpha}^2 u_{l\alpha}^2 \leq 8K^2$, the first sub-sum is at most 
\[ \textstyle
8K^2 \cdot\frac{\lambda_1 \lambda_2}{n^2} \sum_{l=r+1}^d \frac{\lambda_l^2}{(\lambda_p - \lambda_{l})^2}.
\]
Next, we rewrite the second sub-sum as 
\begin{equation*}
    \begin{split}
& \frac{\lambda_1 \lambda_2}{n^2} \cdot \sum_{\alpha, \beta} (\E y_\alpha^4-3)(\E y_{\beta}^4 -3)  u_{1\alpha}^2 u_{2\beta}^2 \cdot \sum_{l,l'=r+1}^d \frac{\lambda_l \lambda_l'}{(\lambda_p- \lambda_l)(\lambda_p- \lambda_{l'})} \cdot u_{l \alpha} u_{l\beta} u_{l' \alpha} u_{l'\beta} \\
& = \frac{\lambda_1 \lambda_2}{n^2} \cdot \sum_{\alpha, \beta} (\E y_\alpha^4-3)(\E y_{\beta}^4 -3)  u_{1\alpha}^2 u_{2\beta}^2 \cdot \left(\sum_{l=r+1}^d \frac{\lambda_l}{\lambda_p-\lambda_l} u_{l\alpha} u_{l\beta}  \right)^2.
    \end{split}
\end{equation*}
By the Cauchy-Schwarz inequality and the facts that $\sum_{l=1}^d u_{l\alpha}^2= \sum_{l=1}^d u_{l\beta}^2 =1, $ we have
\[
\left(\sum_{l=r+1}^d \frac{\lambda_l}{\lambda_p-\lambda_l} u_{l\alpha} u_{l\beta}  \right)^2 \leq \left(\sum_{l=r+1}^d \frac{\lambda_l^2}{(\lambda_p-\lambda_l)^2} u_{l\alpha}^2  \right)^2 \left( \sum_{l=r+1}^d u_{l\beta}^2 \right)^2 \leq 1.
\]
Therefore, the second sub-sum is at most 
\begin{equation*}
    \begin{split}
 \frac{\lambda_1 \lambda_2}{n^2} \cdot \sum_{\alpha, \beta} |\E y_\alpha^4-3| \cdot |\E y_{\beta}^4 -3| \cdot  u_{1\alpha}^2 u_{2\beta}^2 & \textstyle \leq     \frac{\lambda_1 \lambda_2}{n^2} \cdot (4K^2)^2 \cdot \sum_{\alpha, \beta}  u_{1\alpha}^2 u_{2\beta}^2 =  (4K^2 \frac{\sqrt{\lambda_1 \lambda_2}}{n} )^2.  
    \end{split}
\end{equation*}
These estimates on the first and second sub-sums lead to
\begin{equation} \label{M2_1}
    \textstyle
M_2 \leq (4K^2 \frac{\sqrt{\lambda_1 \lambda_2}}{n} )^2+ 8K^2 \cdot\frac{\lambda_1 \lambda_2}{n^2} \sum_{l=r+1}^d \frac{\lambda_l^2}{(\lambda_p - \lambda_{l})^2}.
\end{equation}

Arguing similarly, we also bound $M_1, M_3$ as 
\begin{equation} \label{M2_M1}
\begin{split}
    M_1 & = \textstyle  \frac{(n-1)\lambda_1 \lambda_2}{n^3} \sum_{\substack{d \geq l>r\\d\geq l'>r}} \frac{\lambda_l \lambda_{l'}}{(\lambda_p - \lambda_l')(\lambda_p - \lambda_l)}  \left[ \sum_{\alpha=1}^d (\E y_{\alpha}^4 -3)u_{1\alpha} u_{2\alpha} u_{l \alpha}^2 \right] \left[ \sum_{\alpha=1}^d (\E y_{\alpha}^4 -3)u_{1\alpha} u_{2\alpha} u_{l' \alpha}^2 \right] \\
    & \leq \textstyle  \frac{ \lambda_1 \lambda_2}{n^2} \sum_{\alpha, \beta} (\E y_\alpha^4-3)(\E y_{\beta}^4-3) u_{1\alpha} u_{2\alpha} u_{1\beta}u_{2\beta} \cdot  \left(\sum_{l=r+1}^{d} \frac{\lambda_l}{\lambda_p - \lambda_{l}} u_{l\alpha}^2 \right) \left(\sum_{l=r+1}^{d} \frac{\lambda_l}{\lambda_p - \lambda_{l}} u_{l\beta}^2 \right) \\
    & \leq (4K^2)^2 \frac{\lambda_1 \lambda_2}{n^2} (\sum_{\alpha} |u_{1\alpha}| \cdot |u_{2\alpha}|)^2 \leq (4K^2 \frac{\sqrt{\lambda_1 \lambda_2}}{n} )^2.
\end{split}
\end{equation}
\begin{equation} \label{M2_2}
\begin{split}
   M_3 & =  \textstyle  \frac{(n-1)\lambda_1 \lambda_2}{n^3} \sum_{\substack{d \geq l>r\\d\geq l'>r}} \frac{\lambda_l \lambda_{l'}}{(\lambda_p - \lambda_{l'})(\lambda_p - \lambda_l)}  \left[  \sum_{\alpha=1}^d (\E y_{\alpha}^4 -3) u_{1\alpha} u_{2\alpha} u_{l \alpha}u_{l' \alpha} \right]^2 \\
&  \leq \textstyle  \frac{ \lambda_1 \lambda_2}{n^2} \sum_{\alpha, \beta} (\E y_\alpha^4-3)(\E y_{\beta}^4-3) u_{1\alpha} u_{2\alpha} u_{1\beta}u_{2\beta} \cdot  \left(\sum_{l=r+1}^{d} \frac{\lambda_l}{\lambda_p - \lambda_{l}} u_{l\alpha} u_{l \beta} \right)^2 \\
& \leq \textstyle  \frac{ \lambda_1 \lambda_2}{n^2} \sum_{\alpha, \beta} |\E y_\alpha^4-3| \cdot |\E y_{\beta}^4-3| \cdot |u_{1\alpha} u_{2\alpha}| |u_{1\beta}u_{2\beta}| \leq (4K^2 \frac{\sqrt{\lambda_1 \lambda_2}}{n} )^2.
\end{split}
   \end{equation}

Finally, we estimate $M_4$ as follows.  
\begin{equation} \label{M2_3}
\begin{split}
M_4 & = \textstyle \frac{\lambda_1 \lambda_2}{n^3} \sum_{d \geq l,l' > r} \frac{\lambda_l \lambda_{l'}}{(\lambda_p- \lambda_l)(\lambda_p - \lambda_{l'})} \E (\sum_{j_1} y_{j_1} u_{1 j_1})
^2(\sum_{j_2} y_{j_2} u_{l j_2})^2 (\sum_{j_3} y_{j_3} u_{2 j_3})^2(\sum_{j_4} y_{j_4} u_{l' j_4})^2   \\
& \leq \textstyle  C^2 K^4 \cdot \frac{\lambda_1 \lambda_2}{n^3} \left( \sum_{l=r+1}^d \frac{\lambda_l}{\lambda_p - \lambda_l} \right)^2, \,\,\, \text{for some universal constant} \,\,\, C^2.
\end{split}
\end{equation}
Together \eqref{expectation: u1EEu2}, (\ref{M2_1}), \eqref{M2_M1}, (\ref{M2_2}), and (\ref{M2_3}) imply that with probability at least $1 - t^{-2}$, 
$$   \textstyle  \norm{ u_1^\top E \left(\sum_{l > r} \frac{u_l u_l^\top}{\lambda_p - \lambda_l} \right) E u_2} \leq 3t\left[ \frac{ 4K^2\sqrt{\lambda_1 \lambda_2}}{n}+  \frac{CK^2 \sqrt{\lambda_1 \lambda_2}}{n^{3/2}} \left( \sum_{l > r } \frac{ \lambda_l}{ \lambda_p -\lambda_{l}} \right)+ \frac{K\sqrt{\lambda_1 \lambda_2} }{n} \sqrt{\sum_{l > r} \frac{\lambda_l^2}{(\lambda_p -\lambda_l)^2}} \right].$$
This proves the second part of Lemma \ref{lemma: xyModeil1}.

\section{Proof of Lemma \ref{lem: S value}}
Let $\{u_1, u_2, \cdots, u_d\}$ be an orthonormal system and $Y=[y_1, y_2, \dots, y_d]^T$ be a random vector, whose entries $\{y_k\}_{k=1}^d$ are $8$-wise independent sub-Gaussian random variables with mean $0$, variance $1$, and $\|y_k\|^2_{\psi_2} \leq K$ for some $K \geq 1$. For $1 \leq i,j \leq d$, define 
\[
S^{(i,j)}:= (u_i^\top Y) \cdot (u_j^\top Y).
\]
Consequently, 
$$\Var S^{(i,j)} = \begin{cases}
    1+  \sum_{k=1}^d \E (y_k^4 -3) u_{ik}^2 u_{jk}^2 &\,\,\text{if $i \neq j$}\\
    2 + \sum_{k=1}^d \E (y_k^4 -3) u_{ik}^4 &\,\,\,\text{if $i=j$}
\end{cases}.$$
We show that 
\begin{itemize}
    \item $\Var S^{(i,j)} \leq 4K^2$. 
    \item If additionally $ \min_{1 \leq k \leq d} \E y_k^4 \geq 1+ 2c$ for some $c > 0$, then 
    \[ 
    \Var S^{(i,j)} \geq c.
    \]
\end{itemize}
Without loss of generality, we prove the above claims for $i \neq j$. 

For the upper bound, since $\|y_k\|^2_{\psi_2} \leq K$, $\E y_k^4 \leq 4K^2$. And hence, 
\[
\Var S^{(i,j)} \leq 1 + \sum_{k=1}^d (4K^2 - 3) u_{ik}^2 u_{jk}^2 = 1 + (4K^2-3) \sum_{k=1}^d u_{ik}^2 u_{jk}^2 \leq  1+4K^2 -3 \leq 4K^2.
\]

For the lower bound, since $ \min_{1 \leq k \leq d} \E y_k^4 \geq 1+ 2c$, we have 
\[
\Var S^{(i,j)} \geq 1 - 2(1-c) \sum_{k=1}^d u_{ik}^2 u_{jk}^2.
\]
We will show that the RHS is at least $c$. It is equivalent to show that
\begin{equation} \label{uij1/2}
 \sum_{k=1}^d u_{ik}^2 u_{jk}^2 \leq \frac{1}{2}.   
\end{equation}

Notice that $\sum_{k} u_{ik} u_{jk} =0$, we have 
\[
\sum_{ u_{ik} u_{jk} > 0} u_{ik} u_{jk} = \sum_{ u_{ik} u_{jk} < 0} - u_{ik} u_{jk} = h. 
\]
And hence, $2h = \sum_{k=1}^d |u_{ik} u_{jk}| \leq \sqrt{ \big(\sum_{k=1}^d u_{ik}^2 \big) \big( \sum_{k=1}^d  u_{jk}^2 \big)} =1$. The inequality here follows from the Cauchy-Schwarz inequality.  Equivalently, 
\[
h \leq \frac{1}{2}.
\]
Therefore, 
\begin{equation}
    \begin{split}
     \sum_{k=1}^d u_{ik}^2 u_{jk}^2 & = \sum_{ u_{ik} u_{jk} > 0} (u_{ik} u_{jk})^2 + \sum_{ u_{ik} u_{jk} < 0} (u_{ik} u_{jk})^2 \\
     & \leq (\sum_{ u_{ik} u_{jk} > 0} u_{ik} u_{jk})^2 + (\sum_{ u_{ik} u_{jk} < 0} u_{ik} u_{jk})^2  = 2h^2 \leq \frac{1}{2}. 
    \end{split}
\end{equation}
This proves \eqref{uij1/2}, and completes our proof. 
\end{document}